\documentclass[a4paper]{amsart}

\usepackage{enumerate, amsmath, amsfonts, amssymb, amsthm, wasysym, graphics, graphicx, xcolor, url, hyperref, hypcap, a4wide, stmaryrd, relsize, pdflscape}
\hypersetup{colorlinks=true, citecolor=darkblue, linkcolor=darkblue}
\usepackage{tikz}\usetikzlibrary{trees,snakes,shapes,arrows,matrix,calc}
\graphicspath{{figures/}}

\title[Signed tree associahedra]{Signed tree associahedra}

\thanks{Supported by grant MTM2011-22792 of MICINN and by ANR grant EGOS (12 JS02 002 01).}
\author{Vincent Pilaud}
\address{CNRS \& LIX, \'Ecole Polytechnique, Palaiseau}
\email{vincent.pilaud@lix.polytechnique.fr}
\urladdr{http://www.lix.polytechnique.fr/~pilaud/\vspace*{-1.5cm}}

\newtheorem{theorem}{Theorem}
\newtheorem{corollary}[theorem]{Corollary}
\newtheorem{proposition}[theorem]{Proposition}
\newtheorem{lemma}[theorem]{Lemma}
\newtheorem{definition}[theorem]{Definition}

\theoremstyle{definition}
\newtheorem{example}[theorem]{Example}
\newtheorem{remark}[theorem]{Remark}
\newtheorem{question}[theorem]{Question}

\newcommand{\R}{\mathbb{R}} 
\newcommand{\Z}{\mathbb{Z}} 
\newcommand{\HH}{\mathbb{H}} 
\renewcommand{\b}[1]{\mathbf{#1}} 

\newcommand{\set}[2]{\left\{ #1 \;\middle|\; #2 \right\}} 
\newcommand{\bigset}[2]{\big\{ #1 \;|\; #2 \big\}} 
\newcommand{\biggset}[2]{\bigg\{ #1 \;\bigg|\; #2 \bigg\}} 
\newcommand{\ssm}{\smallsetminus} 
\newcommand{\dotprod}[2]{\langle \, #1 \; | \; #2 \, \rangle} 
\newcommand{\one}{{1\!\!1}} 
\newcommand{\eqdef}{\mbox{\,\raisebox{0.2ex}{\scriptsize\ensuremath{\mathrm:}}\ensuremath{=}\,}} 
\newcommand{\defeq}{\mbox{~\ensuremath{=}\raisebox{0.2ex}{\scriptsize\ensuremath{\mathrm:}} }} 
\newcommand{\simplex}{\triangle} 
\newcommand{\binomspec}[3]{\begin{pmatrix} #1 \\ #2 \end{pmatrix}_{\!\!#3}} 
\renewcommand{\implies}{\Rightarrow} 

\newcommand{\Asso}{\mathsf{Asso}} 
\newcommand{\Perm}{\mathsf{Perm}} 
\newcommand{\Para}{\mathsf{Para}} 
\newcommand{\Zono}{\mathsf{Zono}} 
\newcommand{\Defo}{\mathsf{Defo}} 
\newcommand{\Mink}{\mathsf{Mink}} 
\newcommand{\ground}{\mathsf{V}} 
\newcommand{\graphG}{\mathsf{G}} 
\newcommand{\tree}{\mathsf{T}} 
\newcommand{\treeInterval}{\widetilde{\tree}} 
\newcommand{\phantomTree}{\mathsf{Y}} 
\newcommand{\phantomTreeInterval}{\widetilde{\phantomTree}} 
\newcommand{\pathG}{\mathsf{P}} 
\newcommand{\polygon}{\mathsf{Q}} 
\newcommand{\tripodWhite}{\,\raisebox{-1pt}{\includegraphics[scale=.35]{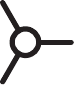}}} 
\newcommand{\tripodBlack}{\,\raisebox{-1pt}{\includegraphics[scale=.35]{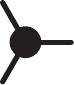}}} 
\newcommand{\subtrees}{\mathcal{Z}} 
\newcommand{\tubes}{\mathcal{W}} 
\newcommand{\building}{\mathcal{B}} 
\newcommand{\curves}{\mathcal{X}} 
\newcommand{\paths}{\mathcal{P}} 
\newcommand{\nested}{\mathsf{N}} 
\newcommand{\nestedComplex}{\mathcal{N}} 
\newcommand{\spine}{\mathsf{S}} 
\newcommand{\spinePoset}{\mathcal{S}} 
\newcommand{\spineFan}{\mathcal{F}} 
\newcommand{\spineFlipGraph}{\mathcal{G}} 
\newcommand{\linear}{\mathcal{L}} 
\newcommand{\orientation}{\mathsf{O}} 
\newcommand{\boundary}{\partial} 
\newcommand{\leaves}{\mathsf{L}} 
\newcommand{\source}{\mathsf{sc}} 
\newcommand{\sink}{\mathsf{sk}} 
\newcommand{\tubeToBuildingBlock}{\mathsf{b}} 
\newcommand{\buildingBlockToTube}{\mathsf{w}} 
\newcommand{\tubeToSubtree}{\mathsf{z}} 
\newcommand{\subtreeToTube}{\mathsf{w}} 
\newcommand{\subtreeToCurve}{\chi} 
\newcommand{\arcToSubtree}{\mathsf{z}} 
\newcommand{\diagonalToSubtree}{\mathsf{z}} 
\newcommand{\diagonalToTube}{\mathsf{w}} 
\newcommand{\diagonalToBuildingBlock}{\mathsf{b}} 
\newcommand{\spineToNested}{\mathsf{N}} 
\newcommand{\nestedToSpine}{\mathsf{S}} 
\newcommand{\nestedToCurves}{\mathsf{X}} 
\newcommand{\curvesToSpine}{\mathsf{S}} 
\newcommand{\buildingBlockToVertex}{\mathsf{s}} 
\newcommand{\setToSpine}{\mathsf{S}} 
\newcommand{\negNested}{\preceq} 
\newcommand{\posNested}{\succeq} 
\newcommand{\negDisjoint}{\perp} 
\newcommand{\posDisjoint}{\;\top\;} 
\newcommand{\vertex}[1]{\b{a}(#1)} 
\newcommand{\face}[1]{\b{f}(#1)} 
\newcommand{\Hyp}{\b{H}^=} 
\newcommand{\HS}{\b{H}^\ge} 
\newcommand{\primalCone}{\mathsf{C}^\star} 
\newcommand{\normalCone}{\mathsf{C}} 
\newcommand{\surjectionPermAsso}{\kappa} 
\newcommand{\surjectionAssoPara}{\lambda} 
\newcommand{\surjectionPermPara}{\mu} 
\newcommand{\braidFan}{\mathcal{B\!F}} 
\newcommand{\fan}{\mathcal{F}} 
\newcommand{\singletons}{\xi} 
\newcommand{\Narayana}{\mathrm{Nar}} 
\newcommand{\origin}{\b{O}} 
\newcommand{\bary}{\b{b}} 
\newcommand{\projDown}{\pi_\downarrow} 
\newcommand{\projUp}{\pi^\uparrow} 

\DeclareMathOperator{\conv}{conv} 
\DeclareMathOperator{\cone}{cone} 
\DeclareMathOperator{\inv}{inv} 

\newcommand{\fref}[1]{Figure~\ref{#1}} 
\newcommand{\ie}{\textit{i.e.}~} 
\newcommand{\eg}{\textit{e.g.}~} 
\newcommand{\aka}{\textit{a.k.a.}~} 
\newcommand{\ordinal}{\textsuperscript{th}} 
\newcommand{\ex}{^{\textrm{ex}}} 
\definecolor{darkblue}{rgb}{0,0,0.7} 
\newcommand{\darkblue}{\color{darkblue}} 
\newcommand{\defn}[1]{\emph{\darkblue #1}} 
\usepackage{todonotes}

\newcommand{\para}[1]{\medskip\noindent#1~---} 

\makeatletter
\def\l@section{\@tocline{1}{4pt}{0pc}{}{}}
\makeatother
\let\oldtocsection=\tocsection

\renewcommand{\tocsection}[2]{\hspace{0em}\bf\oldtocsection{#1}{#2}}

\begin{document}

\begin{abstract}
An associahedron is a polytope whose vertices correspond to the triangulations of a convex polygon and whose edges correspond to flips between them. A particularly elegant realization of the associahedron, due to S.~Shnider and S.~Sternberg and popularized by \mbox{J.-L.~Loday}, has been generalized in two directions: on the one hand by A.~Postnikov to obtain a realization of the graph associahedra of M.~Carr and S.~Devadoss, and on the other hand by C. Hohlweg and C. Lange to obtain multiple realizations of the associahedron parametrized by a sequence of signs. The goal of this paper is to unify and extend these two constructions to signed tree associahedra.

We define the notions of signed tubes and signed nested sets on a vertex-signed tree, generalizing the classical notions of tubes and nested sets for unsigned trees. The resulting signed nested complexes are all simplicial spheres, but they are not necessarily isomorphic, even if they arise from signed trees with the same underlying unsigned structure. We then construct a signed tree associahedron realizing the signed nested complex, obtained by removing certain well-chosen facets from the classical permutahedron. We study relevant properties of its normal fan and of certain orientations of its $1$-skeleton, in connection to the braid arrangement and to the weak order. Our main tool, both for combinatorial and geometric perspectives, is the notion of spines on a vertex-signed tree, which extend the families of Schr\"oder and binary search trees.
\end{abstract}

\vspace*{-.6cm}
\maketitle


\section{Introduction}
\label{sec:introduction}

A $d$-dimensional \defn{associahedron} is a simple convex polytope whose vertices correspond to the triangulations of a convex~$(d+3)$-gon and whose edges correspond to flips between these triangulations. More generally, the face lattice of the polar of a $d$-dimensional associahedron is isomorphic to the simplicial complex of crossing-free subsets of internal diagonals of the $(d+3)$-gon. See \fref{fig:associahedra} for $3$-dimensional examples. Originally defined as combinatorial objects by J.~Stasheff in his work on the homotopy associativity of $H$-spaces~\cite{Stasheff}, associahedra were later realized as boundary complexes of convex polytopes by different methods~\cite{Lee, GelfandKapranovZelevinsky, BilleraFillimanSturmfels, RoteSantosStreinu, Loday, HohlwegLange, PilaudSantos-brickPolytope, CeballosSantosZiegler}. The variety of these constructions and their surprizing properties reflect the rich combinatorial and geometric structure of the associahedra.

\begin{figure}[t]
  \capstart
  \centerline{\includegraphics[scale=1.1]{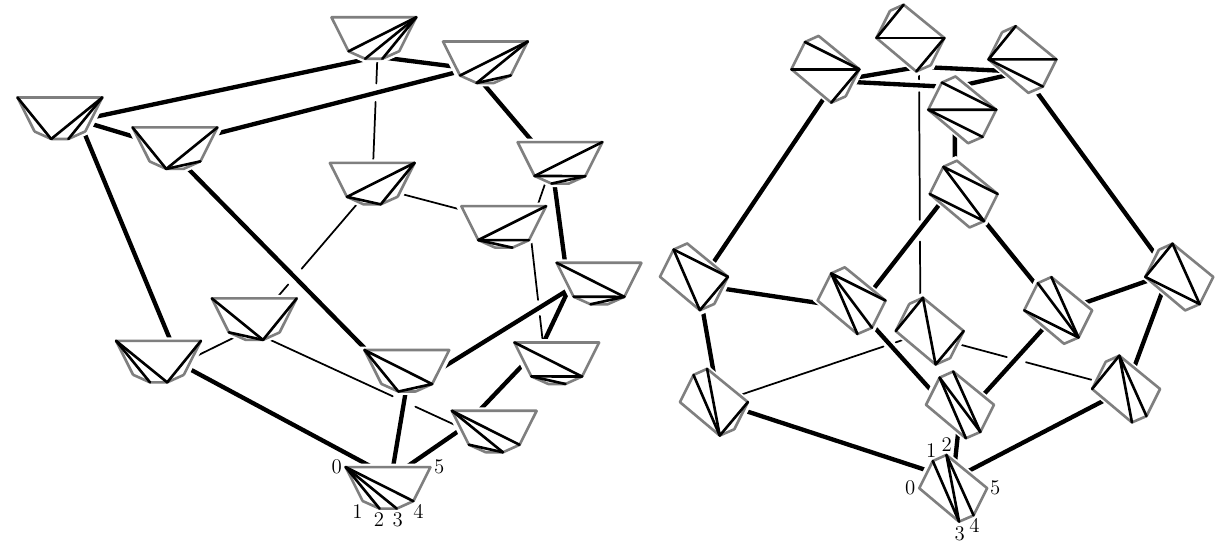}}
  \vspace*{-.15cm}
  \caption{Two polytopal realizations of the $3$-dimensional associahedron, with vertices labeled by triangulations of convex hexagons~\cite{Loday, HohlwegLange, LangePilaud-spines}.}
  \label{fig:associahedra}
  \vspace*{-.1cm}
\end{figure}

\begin{figure}
  \capstart
  \centerline{\includegraphics[scale=1.1]{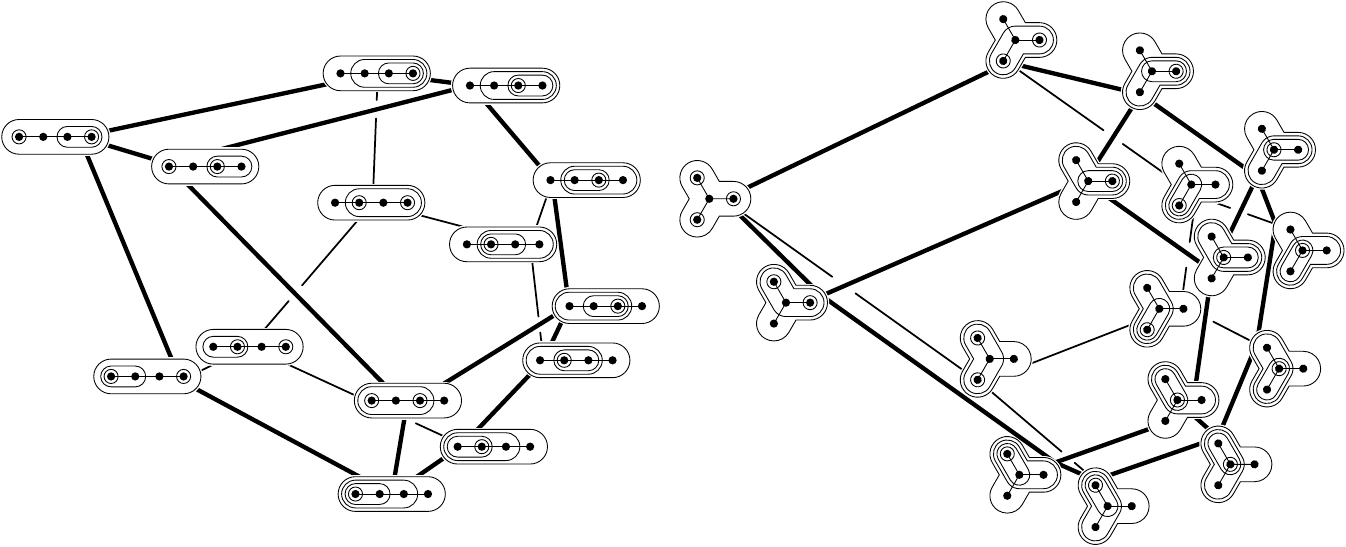}}
  \vspace*{-.15cm}
  \caption{Two $3$-dimensional tree associahedra, with vertices labeled by maximal nested sets of tubes of the trees~\cite{CarrDevadoss, Devadoss, Postnikov}.}
  \label{fig:graphAssociahedra}
\end{figure}
 
In this paper, we focus on a family of realizations, studied under different perspectives in the series of papers~\cite{ShniderSternberg, StasheffShnider, Loday, HohlwegLange, PilaudSantos-brickPolytope, LangePilaud-spines}. We skip here the details of these constructions since they will appear as specifications of the construction of the present paper. However, let us underline some relevant combinatorial and geometric properties of the resulting associahedra. First, they have particularly elegant vertex and facet descriptions with integer vertex coordinates and normal vectors. Second, they are geometrically related to the braid arrangement and to the permutahedron: they are constructed from the permutahedron by gliding some facets to infinity and their normal fans coarsen that of the permutahedron (\ie the braid arrangement). Finally, well-chosen orientations of their $1$-skeletons provides combinatorial connections to the Cambrian lattices~\cite{Reading-CambrianLattices}, obtained as lattice quotients of the weak order on the symmetric group. But besides these properties, the combinatorial richness of these constructions of the associahedron essentially lies in the variety of these realizations. Namely, the $d$-dimensional associahedra constructed by C.~Hohlweg and C.~Lange in~\cite{HohlwegLange} are parametrized by the choice of certain labelings of the convex $(d+3)$-gon, or equivalently by a sequence of~$d+1$ signs in~$\{-,+\}$. Although the combinatorics of the resulting polytopes coincide, their geometry varies: distinct parameters leads to different vertex and facet descriptions, to geometrically different normal fans, and to different Cambrian lattices. For instance, choosing the sequence~$(-)^{d+1}$ yields J.-L.~Loday's associahedron~\cite{Loday} and the classical Tamari Lattice~\cite{TamariFestschrift}.

More recently, M.~Carr and S.~Devadoss~\cite{CarrDevadoss, Devadoss} defined and constructed \defn{graph associahedra}. Given a finite graph~$\graphG$, a  \defn{$\graphG$-associahedron} is a simple convex polytope whose combinatorial structure encodes the connected subgraphs of~$\graphG$ and their nested structure. To be more precise, the face lattice of the polar of a $\graphG$-associahedron is isomorphic to the \defn{nested complex} on~$\graphG$, defined as the simplicial complex of all collections of tubes (connected induced subgraphs) of~$\graphG$ which are pairwise either nested, or disjoint and non-adjacent. See \fref{fig:graphAssociahedra} for $3$-dimensional examples. The graph associahedra of certain special families of graphs happen to coincide with well-known families of polytopes: classical associahedra are path associahedra, cyclohedra are cycle associahedra, and permutahedra are complete graph associahedra. Compare for instance the leftmost associahedron of \fref{fig:associahedra} to the leftmost path associahedron of \fref{fig:graphAssociahedra} for the correspondence between triangulations and maximal nested sets on a path. To our knowledge, graph associahedra (or their generalizations as nestohedra) have been constructed in three different ways: first by successive truncations of faces of the standard simplex~\cite{CarrDevadoss}, then as Minkowski sums of faces of the standard simplex~\cite{Postnikov, FeichtnerSturmfels}, and finally from their normal fans by exhibiting explicit inequality descriptions~\cite{Zelevinsky}. We observe that these realizations can be chosen to have nice integer vertex coordinates and normal vectors~\cite{Devadoss, Postnikov, Zelevinsky}, and that the resulting polytopes belong to the class of \defn{generalized permutahedra} defined in~\cite{Postnikov}. These polytopes are obtained from the classical permutahedron by perturbing the right hand sides of its facet defining inequalities such that no facet passes by a vertex. Equivalently, generalized permutahedra are precisely the polytopes whose normal fans coarsen the braid arrangement~\cite{Postnikov, PostnikovReinerWilliams}. It turns out that the normal fan of the three above-mentioned realizations of the $\graphG$-associahedron is always the same: its rays are given by the characteristic vectors of the tubes of~$\graphG$, and its cones are spanned by the rays corresponding to the nested sets of the nested complex on~$\graphG$. For example, the normal fan of these realizations of the path associahedron is always that of J.-L.~Loday's associahedron~\cite{Loday}, and the corresponding lattice is always the Tamari lattice~\cite{TamariFestschrift}.

In view of the combinatorial diversity of C.~Hohlweg and C.~Lange's realizations of the (path) associahedron, it seems therefore natural to look for similar constructions for the graph associahedra (and more generally for the nestohedra). Such constructions would yield in particular geometrically distinct normal fans and different poset structures on maximal nested sets, exactly as C.~Hohlweg and C.~Lange's associahedra correspond to the different Cambrian lattices and fans~\cite{Reading-CambrianLattices}. This paper is a first step towards this direction: it focusses on the case of tree associahedra. In a forthcoming paper in progress with S.~\v Cuki\'c and C.~Lange, who independently considered possible generalizations of C.~Hohlweg and C.~Lange's associahedra, we will go deeper into this line of research to extend the main results of this paper to signed nestohedra.

Given a tree~$\tree$ on a signed ground set~$\ground \eqdef \ground^- \sqcup \ground^+$, we define the notions of signed tubes and signed nested sets on~$\tree$, extending tubes and nested sets of unsigned trees. The resulting signed nested complex~$\nestedComplex(\tree)$ is a simplicial sphere of dimension~$|\ground|-2$. Somewhat unexpectedly, signed nested complexes defined by different signatures on the same underlying tree are not always isomorphic. However, they are if the signatures only differ from each other by some signs on the legs of the tree (subtrees of~$\tree$ containing at least a leaf and no vertex of degree~$3$ or higher in~$\tree$), as it clearly happens for the path associahedron.

Our main tool to study the combinatorial properties and to construct geometric realizations of the signed nested complexes is the notion of signed spines on~$\tree$. Directly inspired from~\cite{LangePilaud-spines, IgusaOstroff}, these spines are directed and labeled trees, whose label sets partition the ground set~$\ground$, and with a specific local condition around each node, determined by the combinatorics of the signed ground tree~$\tree$. We prove that the contraction poset on the signed spines on~$\tree$ is isomorphic to the signed nested complex on~$\tree$. Furthermore, we interpret ridges of the nested complex as a simple flip operation on maximal signed spines. Note that in the situation of an unsigned tree, the spines are just the Hasse diagrams of the nested poset on the nested sets.

From signed spines, we construct a pointed complete simplicial fan which realizes the signed nested complex and coarsens the braid arrangement. Namely, each spine~$\spine$ defines a cone whose facet normal vectors are the incidence vectors of~$\spine$. It defines in particular a surjection~$\surjectionPermAsso$ from linear orders on~$\ground$ to maximal nested sets of~$\tree$, whose properties are investigated. The key feature of the spine fan is that different signatures on the same underlying tree lead to distinct simplicial fans, even when the signed nested complexes are isomorphic.

The spine fans provide the foundations for the construction of signed tree associahedra. These polytopes are obtained from the permutahedron by gliding some facets to infinity. The normal vectors of the remaining facets are the characteristic vectors of the signed building sets of~$\tree$, or equivalently of all source sets of the signed spines on~$\tree$. Moreover, the vertex corresponding to a maximal signed spine~$\spine$ has simple integer coordinates, counting certain paths in~$\spine$. We then investigate some interesting geometric aspects of the signed tree associahedra, such as their pairs of parallel facets, their common vertices with the permutahedron, and their isometry classes.

Next, we study poset structures on maximal signed nested sets whose Hasse diagrams are given by certain well-chosen orientations of the $1$-skeleton of the signed tree associahedron. Contrarily to the Tamari and Cambrian lattices on triangulations of the $(n+3)$-gon, we prove that these posets are not quotients of the weak order by the surjection~$\surjectionPermAsso$ mentioned above, as soon as the ground tree is not a path. We use these orientations of the $1$-skeleton of the signed tree associahedron to derive properties of the $f$- and $h$-vectors of the signed nested complex.

Finally, we compute explicitly the coefficients of the decomposition of the signed tree associahedron~$\Asso(\tree)$ as a Minkowski sum and difference of dilated faces of the standard simplex of~$\R^\ground$, extending the formulas of~\cite{Lange} for C.~Hohlweg and C.~Lange's realizations of the associahedron.


\section{Signed nested complex}
\label{sec:signedNestedComplex}

\subsection{Open subtrees, signed tubes, and signed building blocks}

We fix a finite \defn{signed ground set}~$\ground = \ground^- \sqcup \ground^+$ with~$\nu$ elements, partitioned into a negative set~$\ground^-$ and a positive set~$\ground^+$. For any subset~$X \subseteq \ground$, we denote by~$X^- \eqdef X \cap \ground^-$ and~$X^+ \eqdef X \cap \ground^+$, and we write~$X = X^- \sqcup X^+$. We also fix a \defn{signed ground tree}~$\tree$, whose vertex set is the signed ground set~$\ground$.

Throughout the paper, we will illustrate all definitions and properties with the signed ground tree~$\tree\ex$ on the signed ground set~$\ground\ex \eqdef \{1,3,4,5\} \sqcup \{0,2,6,7,8,9\}$ represented in \fref{fig:exmTree}. Its negative vertices are colored in white, while its positive ones are colored in black.

\begin{figure}[h]
  \capstart
  \centerline{\includegraphics[scale=1]{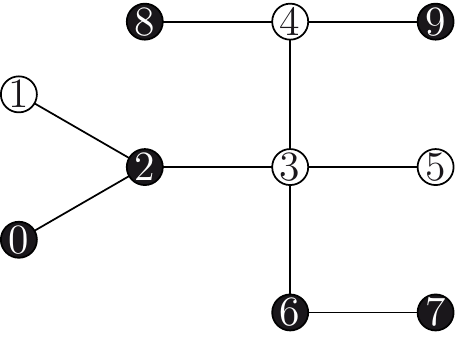}}
  \caption{The signed ground tree~$\tree\ex$ on~$\ground\ex = \{1,3,4,5\} \sqcup \{0,2,6,7,8,9\}$.}
  \label{fig:exmTree}
\end{figure}

We now define the three notions of open subtrees, signed tubes and signed building blocks of~$\tree$. These notions are all equivalent: they capture signed connected substructures of~$\tree$. However, it will be useful to have these different perspectives in mind throughout the paper. We refer to \fref{fig:exmOSSTSBB} for a concrete illustration of these notions and their connections.

\begin{definition}
\label{def:openSubtree}
An \defn{open subtree}~$Z$ of~$\tree$ is a connected component of the complement~$\tree \ssm U$ of a subset~$U$ of~$\ground$. In other words, an open subtree of~$\tree$ is a non-empty subtree of~$\tree$ whose leaves are excluded, except maybe if they are leaves of~$\tree$. The \defn{boundary} of~$Z$ is the set~$\boundary Z$ of excluded leaves of~$Z$. The connected components of~$\tree \ssm \ground^-$ (resp.~of~$\tree \ssm \ground^+$) are called \defn{negative} (resp.~\defn{positive}) \defn{irrelevant} open subtrees of~$\tree$; the other ones are called \defn{relevant}. The \defn{open subtree family} of~$\tree$ is the collection~$\subtrees(\tree)$ of all open subtrees of~$\tree$.
\end{definition}

\begin{definition}
\label{def:signedTube}
A \defn{signed tube} of~$\tree$ is a pair~$W = (W^-, W^+)$ of open subtrees of~$\tree$ such that
\[
\boundary W^- \subseteq \ground^- \cap W^+ \qquad \text{and} \qquad \boundary W^+ \subseteq \ground^+ \cap W^-.
\]
Observe that this implies that $\tree = W^- \cup W^+$. The signed tubes of the form~$(Z, \tree)$ (resp.~$(\tree, Z)$), where~$Z$ is a negative (resp.~positive) irrelevant open subtree of~$\tree$, are called \defn{negative} (resp.~\defn{positive}) \defn{irrelevant} signed tubes of~$\tree$; the other ones are called \defn{relevant}. The \defn{signed tube family} of~$\tree$ is the collection~$\tubes(\tree)$ of all signed tubes of~$\tree$.
\end{definition}

\begin{definition}
\label{def:signedBuildingBlock}
A subset~$B$ of~$\ground$ is \defn{negative convex} (resp.~\defn{positive convex}) in~$\tree$ if any negative (resp.~positive) vertex lying on the unique path in~$\tree$ between two vertices of~$B$ is also in~$B$. A \defn{signed building block} of~$\tree$ is a subset~$B$ of~$\ground$ which is negative convex and whose complement~$\ground \ssm B$ is positive convex. The signed building blocks~$\varnothing$ and~$\ground$ are called \defn{irrelevant}, and all the others are called \defn{relevant}. The \defn{signed building set} of~$\tree$ is the collection~$\building(\tree)$ of all signed building blocks~of~$\tree$.
\end{definition}

\begin{figure}
  \capstart
  \centerline{\includegraphics[scale=1]{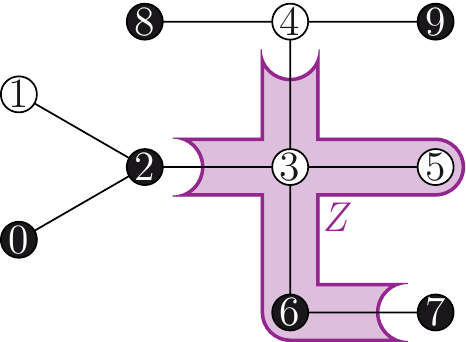} \qquad \includegraphics[scale=1]{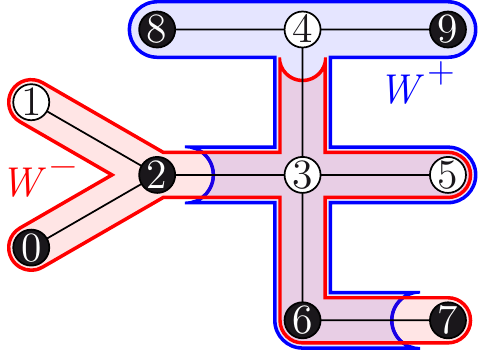} \qquad \includegraphics[scale=1]{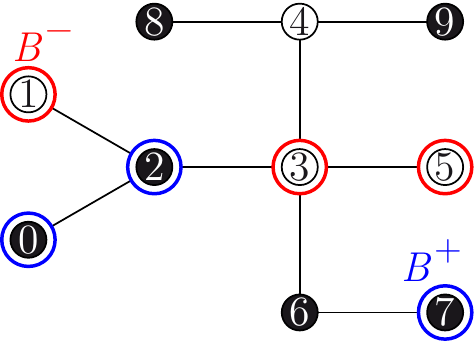}}
  \caption{An open subtree~$Z$ (left), a signed tube~$W = (W^-, W^+)$ (middle), and a signed building block~$B = B^- \sqcup B^+$ (right) on the signed ground tree~$\tree\ex$ of \fref{fig:exmTree}. These three structures are related by the maps between~$\subtrees(\tree\ex)$, $\tubes(\tree\ex)$ and~$\building(\tree\ex)$ described in this section.}
  \label{fig:exmOSSTSBB}
\end{figure}

We now prove that these three notions are equivalent. We refer to \fref{fig:exmOSSTSBB} for illustrations of the connections between these objects.

\begin{lemma}
\label{lem:tubeVsSubtrees}
The map~$\tubeToSubtree : \tubes(\tree) \to \subtrees(\tree)$ defined for~$W = (W^-, W^+) \in \tubes(\tree)$ by
\[
\tubeToSubtree(W) \eqdef W^- \cap W^+
\]
 is a bijection from the signed tubes to the open subtrees of~$\tree$, which sends the relevant signed tubes to the relevant opens subtrees of~$\tree$. We denote by~$\subtreeToTube$ the inverse map.
\end{lemma}

\begin{proof}
The map~$\tubeToSubtree$ is well-defined since the intersection of two open subtrees of~$\tree$ is an open subtree of~$\tree$. To prove that~$\tubeToSubtree$ is bijective, we can directly describe its inverse map~$\subtreeToTube$. Indeed, given an open subtree~$Z$ of~$\tree$, the signed tube~$\subtreeToTube(Z)$ is the pair~$(W^-, W^+)$ where~$W^-$ is the connected component of~$\tree \ssm (\boundary Z)^-$ containing~$Z$ while~$W^+$ is the connected component of~$\tree \ssm (\boundary Z)^+$ containing~$Z$. Clearly, we have~$\boundary W^- \subseteq \ground^- \cap W^+$ while $\boundary W^+ \subseteq \ground^+ \cap W^-$, and $W^- \cap W^+ = Z$. It also follows immediately from the definitions of~$\tubeToSubtree$ and~$\subtreeToTube$ that relevant (resp.~irrelevant) signed tubes are sent to relevant (resp.~irrelevant) open subtrees of~$\tree$.
\end{proof}

\begin{lemma}
\label{lem:tubesVsBuildingBlocks}
The map~$\tubeToBuildingBlock : \tubes(\tree) \to \building(\tree)$ defined for~$W = (W^-, W^+) \in \tubes(\tree)$ by
\[
\tubeToBuildingBlock(W) \eqdef (\ground^- \cap W^-) \sqcup (\ground^+ \ssm W^+)
\]
restricts to a bijection from the relevant signed tubes to the relevant signed building blocks of~$\tree$. We denote by~$\buildingBlockToTube$ the inverse map.
\end{lemma}

\begin{remark}
\label{rem:strange}
Before proving Lemma~\ref{lem:tubesVsBuildingBlocks}, we observe that for any signed tube~$W = (W^-,W^+)$ of~$\tree$, the set~$\tubeToBuildingBlock(W)$ is a subset of~$W^-$ and its complement~$\ground \ssm \tubeToBuildingBlock(W)$ is a subset of~$W^+$. Indeed, consider an element~$v \in \tubeToBuildingBlock(W)$. If~$v \in \ground^-$ then~$v \in W^-$ by definition, while if~$v \in \ground^+$ then $v \notin W^+$ and thus~$v \in W^-$ since~$\tree = W^- \cup W^+$. Thus, $\tubeToBuildingBlock(W) \subseteq W^-$. We prove similarly that~$\ground \ssm \tubeToBuildingBlock(W) \subseteq W^+$. Note that the vertices of~$\tubeToSubtree(W)$ are precisely those which belong to both~$W^-$ and~$W^+$.
\end{remark}

\begin{proof}[Proof of Lemma~\ref{lem:tubesVsBuildingBlocks}]
We first prove that~$\tubeToBuildingBlock$ is well-defined. Assume that a negative vertex~$v \in \ground^-$ lies in between two vertices~$u$ and~$w$ of~$\tree$, which both belong to~$\tubeToBuildingBlock(W)$. By Remark~\ref{rem:strange}, $u$ and~$w$ both belong to~$W^-$ which is convex. Therefore, $v$ also belongs to~$W^-$ and thus to~$\tubeToBuildingBlock(W)$. We obtain that~$\tubeToBuildingBlock(W)$ is negative convex, and we prove similarly that its complement is positive convex. 

The map~$\tubeToBuildingBlock$ clearly sends relevant signed tubes to relevant signed building blocks. To see that it defines a bijection between these two sets, we can directly define its inverse map~$\buildingBlockToTube$. Indeed, consider a relevant signed building block~$B$ of~$\tree$. Since~$B$ is negative convex, it is contained in a connected component~$W^-$ of~$\tree \ssm (\ground^- \ssm B)$, and since~$\ground \ssm B$ is positive convex, it is contained in a connected component~$W^+$ of~$\tree \ssm (\ground^+ \cap B)$. Therefore, the union of $W^-$ and~$W^+$ cover the complete tree~$\tree$ and by definition, $\boundary W^- \subseteq \ground^-$ and~$\boundary W^+ \subseteq \ground^+$. Therefore, $W = (W^-, W^+)$ defines a signed tube of~$\tree$ and~$\tubeToBuildingBlock(W) = B$.
\end{proof}

\begin{example}
\label{exm:buildingBlocksFromEdge}
Given any edge~$e$ of~$\tree$, the vertex sets~$X$ and~$Y$ of the two connected components of~$\tree \ssm \{e\}$ define two signed building blocks of~$\building(\tree)$. Indeed, $X$ and~$Y$ are complementary and both negative and positive convex. For example, $\buildingBlockToTube(X)^-$ is the connected component of~$\tree \ssm Y^-$ containing~$e$ while~$\buildingBlockToTube(X)^+$ is the connected component of~$\tree \ssm X^+$ containing~$e$, and the open subtree~$\tubeToSubtree(\buildingBlockToTube(X))$ is the connected component of~$\tree \ssm (X^+ \cup Y^-)$ containing~$e$.
\end{example}

We now provide some relevant examples of the notions introduced here, making connections to classical tubes and to diagonals of a convex polygon.

\enlargethispage{.3cm}
\begin{example}[Unsigned tree]
If~$\tree$ is a signed tree with only negative signs, then all its signed tubes are of the form~$(Z, \tree)$, where~$Z$ is an open subtree of~$\tree$. We can thus forget the positive part of each signed tube to obtain the classical tubes of~$\tree$. See~\cite{CarrDevadoss, Devadoss, Postnikov} and \fref{fig:graphAssociahedra}.
\end{example}

\begin{example}[Signed path]
\label{exm:signedPath}
Consider a signed path~$\pathG$, labeled by~$1, \dots, \nu$ from one endpoint to the other. Let~$\polygon \subset \R^2$ be a convex~$(\nu+2)$-gon, with vertices labeled by~$0,1,\dots,\nu+1$ from left to right (no two vertices lie on the same vertical line), and such that the vertex of~$\polygon$ labeled by~$v \in [\nu]$ lies on the lower convex hull of~$\polygon$ if~$v \in [\nu]^-$ and on the upper convex hull of~$\polygon$ if~$v \in [\nu]^+$. See \fref{fig:diagonals}. Each diagonal~$\delta$ of~$\polygon$ projects to an open subpath~$\diagonalToSubtree(\delta)$ of~$\pathG$. The corresponding signed tube~$\diagonalToTube(\delta)$ can be seen as the pair~$(\diagonalToTube(\delta)^-, \diagonalToTube(\delta)^+)$ formed by the subpath~$\diagonalToTube(\delta)^-$ of the lower hull of~$\polygon$ below the line supporting~$\delta$ and the subpath~$\diagonalToTube(\delta)^+$ of the upper hull of~$\polygon$ above the line supporting~$\delta$. The corresponding signed building block~$\diagonalToBuildingBlock(\delta)$ is the set of labels~$j \in [\nu]$ of the points of~$\polygon$ which lie below the line supporting~$\delta$, where we include the endpoints of~$\delta$ if they are up and exclude them if they are down vertices. Note that the boundary diagonals of~$\polygon$ (except the diagonals~$[0,1]$ and~$[\nu,\nu+1]$) correspond to the irrelevant open subtrees, signed tubes and signed building blocks on~$\pathG$. See~\cite{HohlwegLange, LangePilaud-spines} and \fref{fig:diagonals}.

\begin{figure}[h]
  \capstart
  \centerline{\includegraphics[width=1\textwidth]{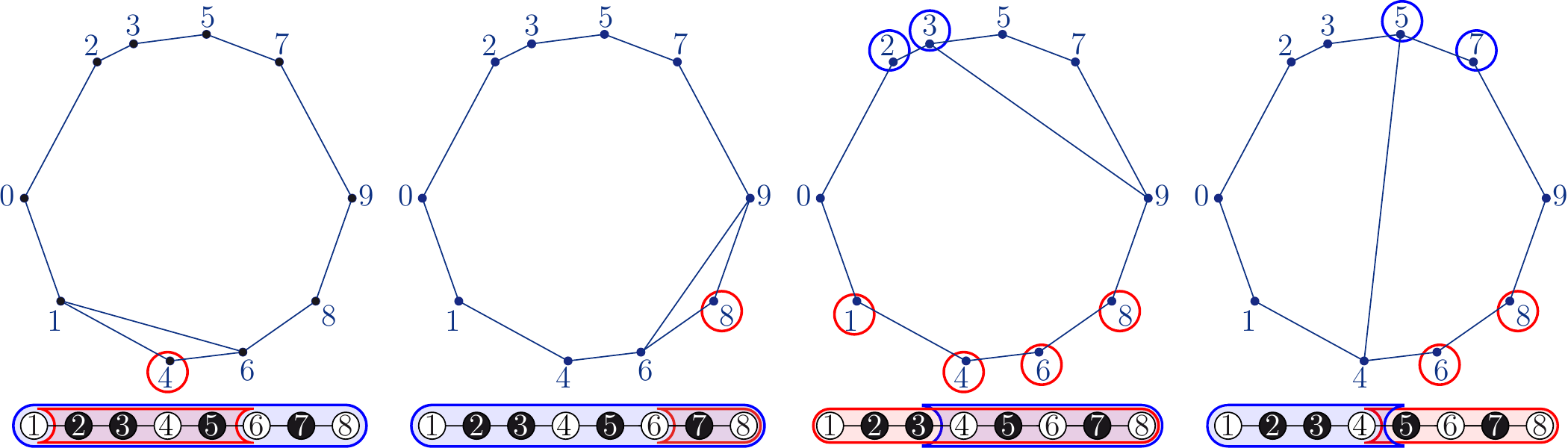}}
  \caption{Correspondence between the diagonals of the $(\nu+2)$-gon~$\polygon$, the signed tubes on~$\pathG$ and the signed building blocks on~$\pathG$.}
  \label{fig:diagonals}
\end{figure}
\end{example}

\begin{example}[Tripod]
\label{exm:tripod}
Call \defn{tripods} the two signed trees~$\tripodWhite$ and~$\tripodBlack$ with three negative leaves and one $3$-valent internal vertex, negative in ~$\tripodWhite$ and positive in~$\tripodBlack$. \fref{fig:tripodSubtreesTubesBuildingBlocks} presents all relevant open subtrees, signed tubes, and signed building blocks on these trees, up to the automorphisms of these trees.

\begin{figure}[h]
  \capstart
  \centerline{\includegraphics[width=\textwidth]{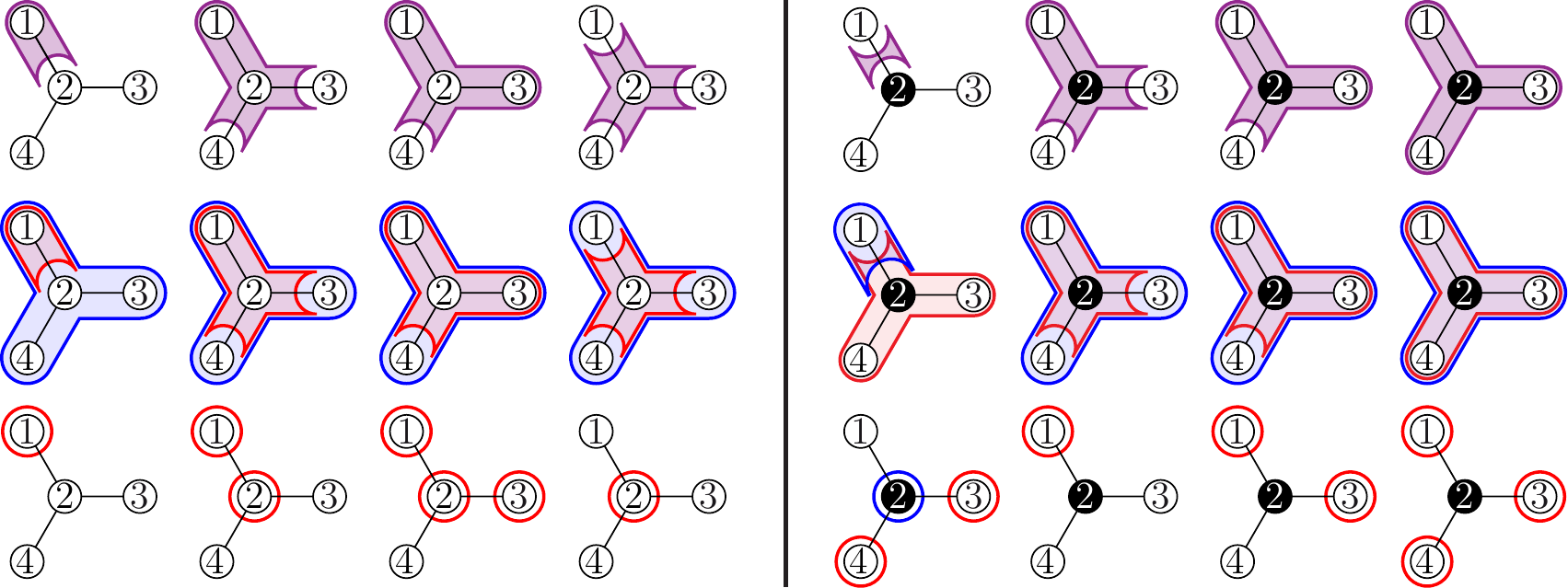}}
  \caption[All relevant open subtrees, signed tubes, and signed building blocks (up to tree automorphisms) on the tripods.]{All relevant open subtrees, signed tubes, and signed building blocks (up to tree automorphisms) on the tripods~$\tripodWhite$ and~$\tripodBlack$.}
  \label{fig:tripodSubtreesTubesBuildingBlocks}
\end{figure}
\end{example}

To complete our presentation, we provide a geometric representation of the open subtrees, signed tubes and signed building blocks of~$\tree$. This representation is inspired from the case of signed paths, presented in Example~\ref{exm:signedPath}. Consider the space~$\treeInterval \eqdef \tree \times [-1,1]$ obtained as the Cartesian product of the ground tree~$\tree$ by the interval~$[-1,1]$. We lift each vertex~$v \in \ground$, either to~$(v,-1)$ if~$v \in \ground^-$ or to~$(v,1)$ if~$v \in \ground^+$. See \fref{fig:exmProduct}\,(left). To visualize an open subtree~$Z$ of~$\tree$, we lift it to a (branched) curve~$\subtreeToCurve(Z)$ in~$\treeInterval$, joining the lifted vertices of~$\boundary Z$, and disjoint from the remaining lifted vertices. If~$Z$ contains a leaf~$\ell$ of~$\tree$, then the endpoint of the edge of~$Z$ incident to~$\ell$ is lifted to the point~$(\ell,0) \in \treeInterval$. See \fref{fig:exmProduct}\,(right).

The curve~$\subtreeToCurve(Z)$ separates~$\treeInterval$ into two connected components, one above it and one below it. The signed tube~$\subtreeToTube(Z)$ is then the pair~$(\subtreeToTube(Z)^-, \subtreeToTube(Z)^+)$ formed by the open subtree~$\subtreeToTube(Z)^-$ of~$\tree \times \{-1\}$ containing all edges below~$\subtreeToCurve(Z)$ and the open subtree~$\subtreeToTube(Z)^+$ of~$\tree \times \{1\}$ containing all edges above~$\subtreeToCurve(Z)$. Moreover, the signed building set~$\tubeToBuildingBlock(\subtreeToTube(Z))$ is the set of lifted vertices located below~$\subtreeToCurve(Z)$, including the vertices of~$\boundary Z \cap \ground^+$ but excluding that of~$\boundary Z \cap \ground^-$. See \fref{fig:exmProduct}\,(right).

\begin{figure}
  \capstart
  \centerline{\includegraphics[scale=1]{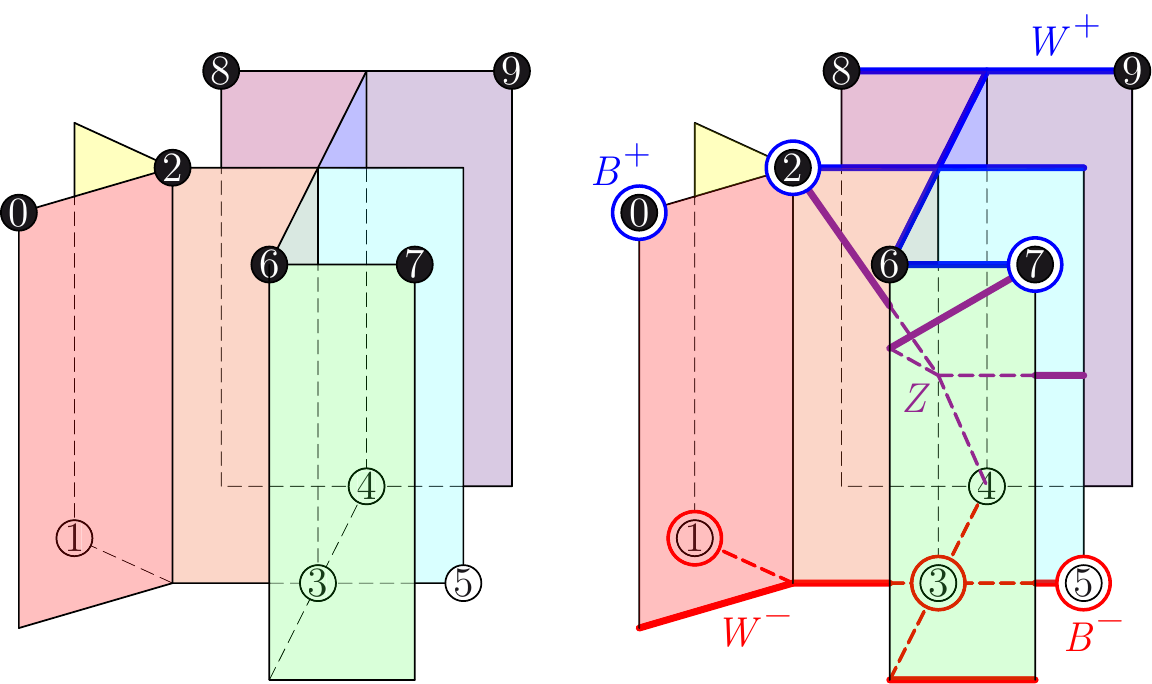}}
  \caption{Geometric representation of open subtrees, signed tubes and signed building blocks on the tree~$\tree\ex$ of \fref{fig:exmTree}. Left: The space~$\tree\ex \times [-1,1]$, where negative vertices of~$\tree\ex$ appear below while positive ones appear above. Right: the open subtree~$Z$, the signed tube~$\subtreeToTube(Z) = W = (W^-, W^+)$ and the signed building block~$\tubeToBuildingBlock(\subtreeToTube(Z)) = B = B^- \sqcup B^+ = \{0,1,2,3,7\}$ of \fref{fig:exmOSSTSBB}.}
  \label{fig:exmProduct}
\end{figure}


\subsection{Signed nested sets and the signed nested complex}

We now define a notion of compatibility between the objects introduced in the previous section. It is easier to define compatibility on signed tubes and to transport it to open subtrees and signed building blocks via the maps~$\tubeToSubtree$ and~$\tubeToBuildingBlock$ defined in the previous section. 

\begin{definition}
Let~$W_1 = (W_1^-, W_1^+)$ and~$W_2 = (W_2^-, W_2^+)$ be two signed tubes of~$\tree$. Define the binary relations
\begin{align*}
& W_1 \negNested W_2 \text{ (read ``\,}W_1 \text{ \defn{negative nested} in } W_2\text{'')} \iff W_1^- \subseteq W_2^- \; \text{ and } \; W_1^+ \supseteq W_2^+, \\
& W_1 \posNested W_2 \text{ (read ``\,}W_1 \text{ \defn{positive nested} in } W_2\text{'')} \iff W_1^- \supseteq W_2^- \; \text{ and } \; W_1^+ \subseteq W_2^+, \\
& W_1 \negDisjoint W_2 \text{ (read ``\,}W_1, W_2 \text{ \defn{negative disjoint}'')} \iff W_1^- \cap W_2^- = \varnothing \; \text{ and } \; W_1^+ \cup W_2^+ = \ground, \\
& W_1 \posDisjoint W_2 \text{ (read ``\,}W_1, W_2 \text{ \defn{positive disjoint}'')} \iff W_1^- \cup W_2^- = \ground \; \text{ and } \; W_1^+ \cap W_2^+ = \varnothing.
\end{align*}
We say that~$W_1$ and~$W_2$ are
\begin{enumerate}[(i)]
\item \defn{signed nested} iff $W_1 \negNested W_2$ or $W_1 \posNested W_2$,
\item \defn{signed disjoint} iff $W_1 \negDisjoint W_2$ or $W_1 \posDisjoint W_2$,
\item \defn{signed compatible} if they are either signed nested or signed disjoint.
\end{enumerate}
A \defn{signed nested set} of~$\tree$ is a collection~$\nested$ of relevant signed tubes of~$\tubes(\tree)$ which are pairwise compatible. The \defn{signed nested complex} on~$\tree$ is the simplicial complex~$\nestedComplex\tubes(\tree)$ of all signed nested sets of~$\tree$. In other words, it is the clique complex on the signed compatility relation on relevant signed tubes of~$\tree$.
\end{definition}

\begin{remark}
\label{rem:compatibleBoundaryTubes}
\begin{enumerate}[(i)]
\item Any irrelevant signed tube of~$\tree$ is signed compatible with all the signed tubes of~$\tree$. Therefore, we only consider relevant signed tubes in the signed nested complex~$\nestedComplex\tubes(\tree)$, since considering irrelevant signed tubes would just result in a folded cone over~$\nestedComplex\tubes(\tree)$.
\item Any signed tube~$W$ of~$\tree$ is negative nested with the negative irrelevant tubes given by the connected components of~$W^- \ssm \ground^-$ and positive nested with the positive irrelevant tubes given by the connected components of~$W^+ \ssm \ground^+$.
\end{enumerate}
\end{remark}

\begin{remark}
\label{rem:compatibleSignedBuildingBlocks}
The signed nestedness and signed disjointness relations can be interpreted on the signed building set as follows. Let~$W_1$ and~$W_2$ be two signed tubes with corresponding signed building blocks~$B_1 \eqdef \tubeToBuildingBlock(W_1)$ and~$B_2 \eqdef \tubeToBuildingBlock(W_2)$, respectively. Then
\begin{enumerate}[(i)]
\item $W_1 \negNested W_2$ iff $B_1 \subseteq B_2$;
\item $W_1 \posNested W_2$ iff $B_1 \supseteq B_2$;
\item $W_1 \negDisjoint W_2$ iff $B_1 \cap B_2 = \varnothing$ and $B_1 \cup B_2 \notin \building(\tree)$; we then write~$B_1 \negDisjoint B_2$;
\item $W_1 \posDisjoint W_2$ iff $B_1 \cup B_2 = \ground$ and $B_1 \cap B_2 \notin \building(\tree)$; we then write~$B_1 \posDisjoint B_2$.
\end{enumerate}
We say that the two signed building blocks~$B_1$ and~$B_2$ are \defn{signed compatible} when the corresponding signed tubes~$W_1$ and~$W_2$ are. We also call \defn{signed nested complex} the simplicial complex~$\nestedComplex\building(\tree)$ of all collections of pairwise compatible relevant signed building blocks of~$\tree$.
\end{remark}

\begin{remark}
\label{rem:compatibleOpenSubtrees}
Let~$W_1$ and~$W_2$ be two signed tubes with corresponding open subtrees~$Z_1 \eqdef \tubeToSubtree(W_1)$ and~$Z_2 \eqdef \tubeToSubtree(W_2)$, respectively. We say that~$Z_1$ and~$Z_2$ are \defn{signed compatible} when the corresponding signed tubes~$W_1$ and~$W_2$ are. In fact, the signed compatibility relation can be visualized on open subtrees and their representation in~$\treeInterval \eqdef \tree \times [-1,1]$ as follows. The open subtrees~$Z_1$ and~$Z_2$ are signed compatible iff their representing curves~$\subtreeToCurve(Z_1)$ and~$\subtreeToCurve(Z_2)$ in~$\treeInterval$ are non-crossing (\ie interior disjoint). See \fref{fig:exmNonCrossing} for illustrations. To be more precise, the curves~$\subtreeToCurve(Z_1)$ and~$\subtreeToCurve(Z_2)$ can be chosen to be non-crossing. In fact, the curves~$\subtreeToCurve(Z)$ representing all open subtrees~$Z \in \subtrees(\tree)$ in~$\treeInterval$ can be chosen simultaneously such that non-crossing curves represent signed compatible open subtrees of~$\tree$. In this representation, the curves representing irrelevant open subtrees can lie on the boundary of~$\treeInterval$. We denote by~$\curves(\tree)$ a set of curves representing all open subtrees of~$\tree$ with these properties. We also call \defn{signed nested complex} the simplicial complex~$\nestedComplex\curves(\tree)$ of crossing-free subsets of relevant curves of~$\curves(\tree)$. Observe also that a nested set~$\nested$ of~$\tree$ corresponds to a \defn{dissection} of~$\treeInterval$ by a set~$\nestedToCurves(\nested)$ of curves of~$\curves(\tree)$. We call \defn{cells} the connected components of the complement of the curves of~$\nestedToCurves(\nested)$ in~$\treeInterval$. See \fref{fig:exmDissection} for an example.
\end{remark}

\begin{figure}[h]
  \capstart
  \centerline{\includegraphics[width=1.1\textwidth]{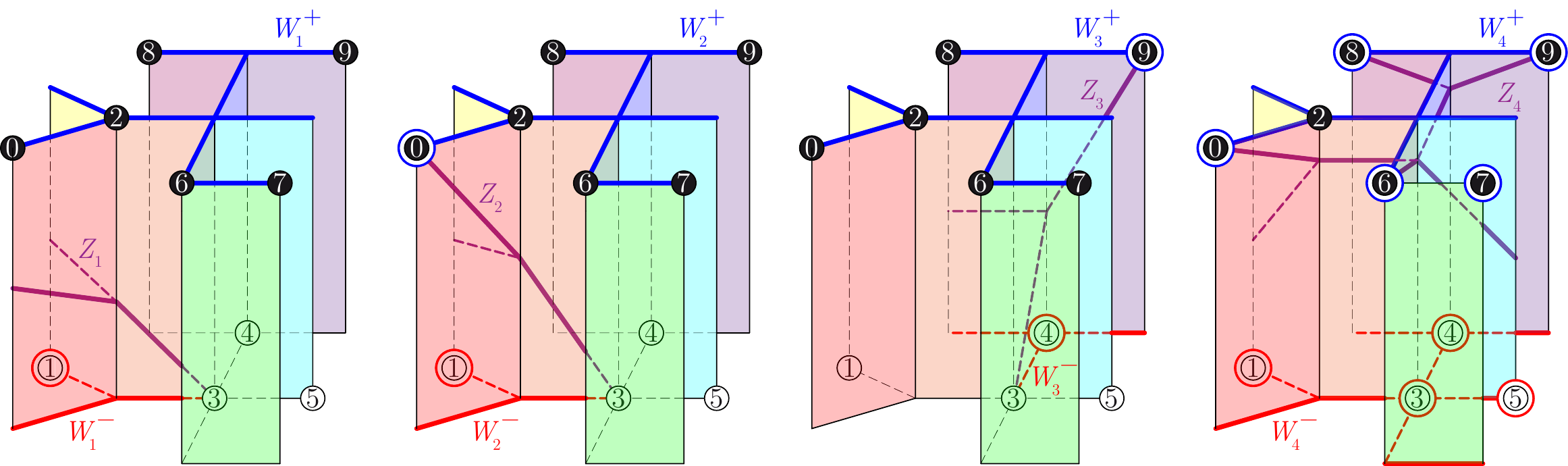}}
  \caption{Four signed compatible open subtrees~$Z_1, Z_2, Z_3, Z_4$ on the ground tree~$\tree\ex$ of \fref{fig:exmTree} with their corresponding signed tubes~$W_1, W_2, W_3, W_4$ and signed building blocks~$B_1, B_2, B_3, B_4$. We have~$W_1 \negNested W_2 \negNested W_4$ and $W_3 \negNested W_4$ while $W_1, W_2 \negDisjoint W_3$. The open subtree~$Z$ of \fref{fig:exmProduct} is compatible with $Z_1$ and~$Z_2$ but not with~$Z_3$ and~$Z_4$.}
  \label{fig:exmNonCrossing}
\end{figure}

\begin{figure}[h]
  \capstart
  \centerline{\includegraphics[width=1.05\textwidth]{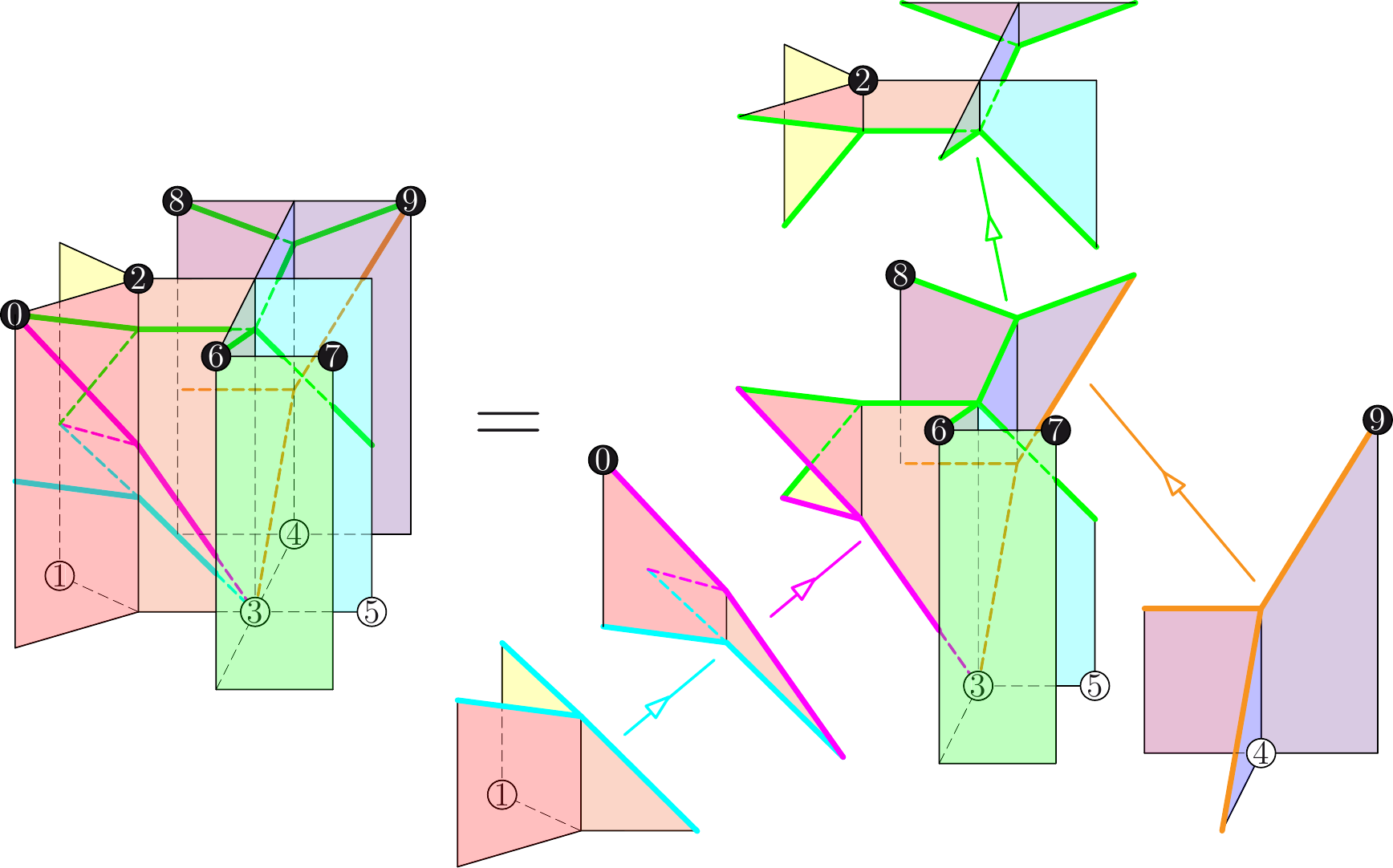}}
  \caption{The four signed compatible open subtrees~$Z_1$, $Z_2$, $Z_3$, $Z_4$ of \fref{fig:exmNonCrossing} form a signed nested set (left), which dissects the space~$\treeInterval\ex$ into five cells (right).}
  \label{fig:exmDissection}
\end{figure}

We have overloaded on purpose the term ``signed nested complex'' since the three simplicial complexes~$\nestedComplex\tubes(\tree)$, $\nestedComplex\building(\tree)$ and~$\nestedComplex\curves(\tree)$ are isomorphic. We only specify the setting when it is necessary and we just write~$\nestedComplex(\tree)$ to denote the signed nested complex in general.

We conclude again by the special situations of unsigned trees, signed path and tripods, in connection to the classical nested complex and to the simplicial associahedron.

\begin{example}[Unsigned tree, continued]
If~$\tree$ has only negative vertices, then the signed tubes~$(Z,\tree)$ and~$(Z',\tree)$ are signed nested or signed disjoint iff the open subtrees~$Z$ and~$Z'$ are nested or disjoint. Therefore, the signed nested complex~$\nestedComplex(\tree)$ is the classical nested complex on the unsigned tree~$\tree$. See~\cite{CarrDevadoss, Devadoss, Postnikov} and \fref{fig:graphAssociahedra}.
\end{example}

\begin{example}[Signed path, continued]
Consider a signed path~$\pathG$ and its corresponding polygon~$\polygon$ (see Example~\ref{exm:signedPath} and \fref{fig:diagonals}). Two internal diagonals~$\delta$ and~$\delta'$ of~$\polygon$ are non-crossing iff their corresponding signed building blocks~$\diagonalToBuildingBlock(\delta)$ and~$\diagonalToBuildingBlock(\delta')$ are signed compatible. For example, the first three diagonals of \fref{fig:diagonals} are compatible, while the last two are not. The signed nested complex~$\nestedComplex(\pathG)$ is thus isomorphic to the simplicial associahedron on~$\polygon$. See \fref{fig:associahedra}.
\end{example}

\begin{example}[Tripod, continued]
\fref{fig:tripodNestedComplex} represents the signed nested complexes~$\nestedComplex\tubes(\tripodWhite)$ and~$\nestedComplex\tubes(\tripodBlack)$ for the two tripods~$\tripodWhite$ and~$\tripodBlack$.

\begin{figure}[h]
  \capstart
  \centerline{\includegraphics[width=1.2\textwidth]{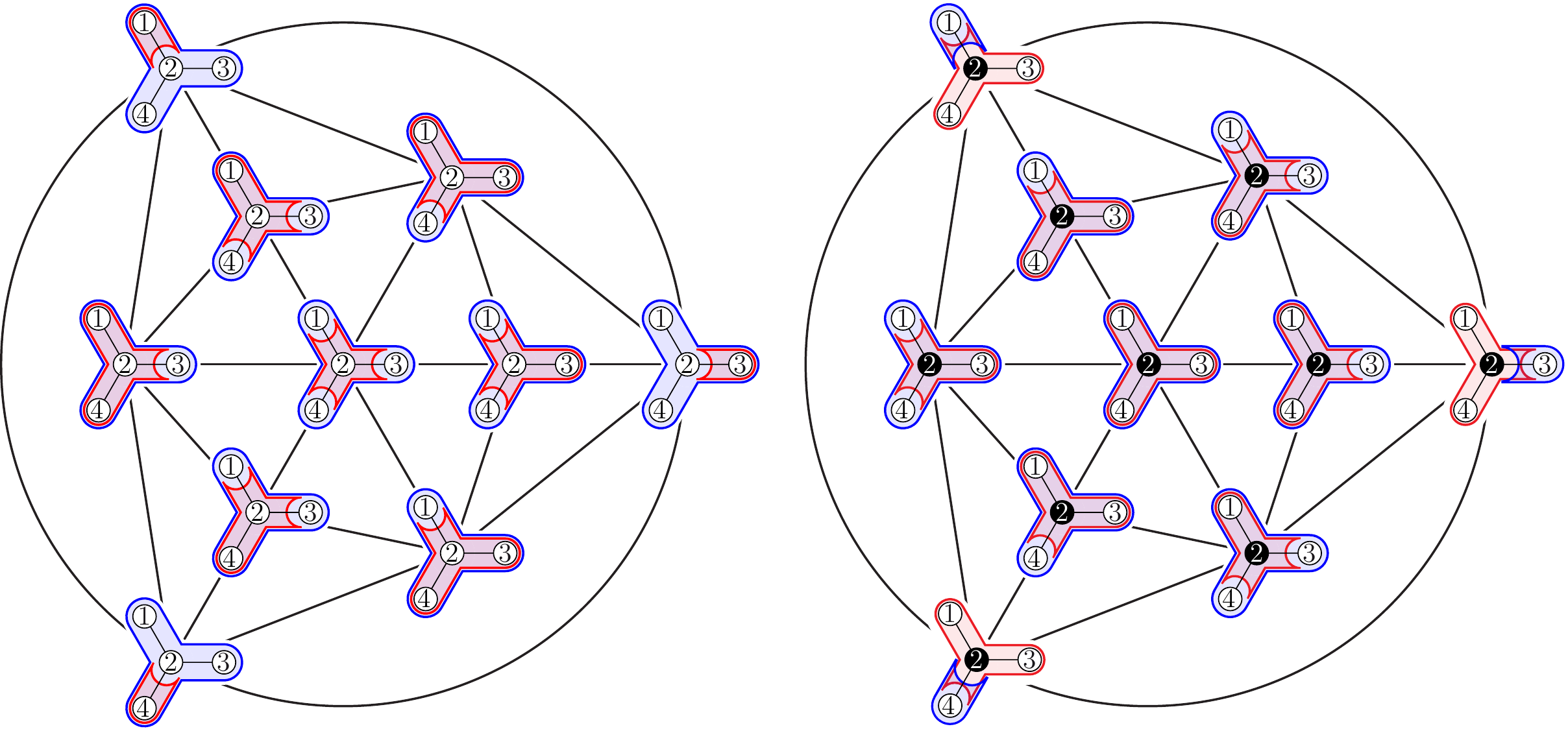}}
  \caption[The signed nested complexes of the tripods.]{The signed nested complexes of the tripods~$\tripodWhite$ and~$\tripodBlack$.}
  \label{fig:tripodNestedComplex}
\end{figure}
\end{example}


\subsection{Isomorphic signed nested complexes}
\label{subsec:isomorphisms}

We are interested in isomorphisms between signed nested complexes given by signed trees with the same underlying unsigned structure. We can already observe that the following operations on the signed tree~$\tree$ preserve the isomorphism class of the signed nested complex~$\nestedComplex(\tree)$.

\enlargethispage{.3cm}
\begin{proposition}
\label{prop:nestedComplexPreservingOperations}
Let~$\tree$ be a signed tree and let~$\tree'$ be a signed tree obtained from~$\tree$ by one of the following operations:
\begin{enumerate}[(i)]
\item changing simultaneously the signs of all vertices of~$\tree$,
\item relabeling the vertices of~$\tree$ while preserving their signs,
\item applying a graph automorphism of~$\tree$ to the signs of~$\tree$,
\item changing the sign of a leaf of~$\tree$,
\item switching two vertices of~$\tree$, adjacent to each other and of degree at most~$2$.
\end{enumerate}
Then the signed nested complexes~$\nestedComplex(\tree)$ and~$\nestedComplex(\tree')$ are isomorphic.
\end{proposition}

\begin{proof}
Points (i) to~(iii) are immediate. We treat separately Points~(iv) and~(v) below. To simplify the arguments, we prefer to use signed building blocks rather than open subtrees or signed tubes. 

%
%

\para{(iv)}
Assume that~$\ell$ is a leaf of~$\tree$, and let~$\tree'$ be the tree obtained by changing the sign of~$\ell$ in~$\tree$. Since~$\ell$ is a leaf, it does not belong to the interior of any path in~$\tree$. Therefore, changing the sign of~$\ell$ does not affect negative and positive convex sets of~$\tree$. It follows that the signed building sets~$\building(\tree)$ and~$\building(\tree')$ coincide. Since the signed compatibility can be seen on the signed building blocks (see Remark~\ref{rem:compatibleSignedBuildingBlocks}), the signed nested complexes~$\nestedComplex(\tree)$ and~$\nestedComplex(\tree')$ are isomorphic.

\para{(v)}
Switching two vertices with the same sign clearly boils down to relabeling these vertices and the corresponding signed nested complexes are therefore isomorphic by~(iii). Consider now two signed ground trees
\[
\tree \eqdef \raisebox{-3pt}{\includegraphics[scale=.9]{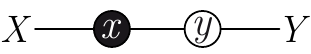}} \qquad\text{and}\qquad \tree' \eqdef \raisebox{-3pt}{\includegraphics[scale=.9]{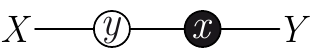}}
\]
which differ by switching vertices~$x \in \ground^+$ and~$y \in \ground^-$. The subtrees~$X$ and~$Y$ are the connected components of the trees~$\tree$ and~$\tree'$ when deleting~$x$ and~$y$. Note that~$X$ and~$Y$ might be empty sets. The vertices~$x$ and~$y$ are therefore adjacent to each other and of degree at most~$2$.

According to Example~\ref{exm:buildingBlocksFromEdge}, the edge~${x \! - \! y}$ of~$\tree$ defines a signed building block~${B_\circ \eqdef X \cup \{x\}}$ of~$\building(\tree)$, while the edge~$y \! - \! x$ defines a signed building block~${B_\circ' \eqdef Y \cup \{x\}}$ of~$\building(\tree')$. We prove that the signed nested complexes~$\nestedComplex(\tree)$ and~$\nestedComplex(\tree')$ are isomorphic in five steps:
\begin{enumerate}[(a)]

\item $B_\circ$ is the only signed building block of~$\building(\tree)$ such that~$x \in B_\circ$ and~$y \notin B_\circ$. Indeed, for any~$B \in \building(\tree)$ with~$x \in B$ and~$y \notin B$, we have~$Y \cap B = \varnothing$ since $B$ is negative convex in~$\tree$, and~$X \subset B$ since the complement of~$B$ is positive convex in~$\tree$.

\smallskip
\item The sets~$\building(\tree)$ and~$\building(\tree')$ only differ by~$B_\circ$ and~$B_\circ'$, \ie $\building(\tree) \ssm \{B_\circ\} = \building(\tree') \ssm \{B_\circ'\}$. Observe first that~$B_\circ$ is not in~$\building(\tree')$ since it is not negative convex in~$\tree'$ and that~$B_\circ'$ is not in~$\building(\tree)$ since it is not negative convex in~$\tree$. Consider now a signed building block~${B \in \building(\tree) \ssm \{B_\circ\}}$. We prove here that~$B$ is negative convex in~$\tree'$, the positive convexity being similar. Let~$u,v,w \in \ground$ be such that~$u,w \in \building$,~$v \in \ground^-$, and~$v$ lies in between~$u$ and~$w$ in~$\tree'$. If $v$ also lies in between~$u$ and~$w$ in~$\tree$, then~$v \in B$ since~$B$ is negative convex in~$\tree$. Otherwise, we have~$u = x$, $v = y$ and~$w \in Y$ since only $x$ and~$y$ are exchanged from~$\tree$ to~$\tree'$. Since~$B \ne B_\circ$ and~$u = x \in B$, Step~(a) ensures that~$y \in B$, so that~$B$ is indeed negative convex in~$\tree'$. 

\smallskip
\item For any relevant signed building set~$B \in \building(\tree)$ distinct from~$B_\circ$, we have~$B \subseteq B_\circ \implies B \negDisjoint B_\circ'$ and~$B \supseteq B_\circ \implies B \posDisjoint B_\circ'$. Indeed, if~$B \subseteq B_\circ$, then~$y \notin B$ so that~$x \notin B$ by Step~(a). Therefore, we have~$B \subseteq X$ and thus~$B \cap B_\circ' = \varnothing$. Moreover, since~$B \ne \varnothing$, the union~$B \cup B_\circ'$ is not negative convex in~$\tree'$ and thus not in~$\building(\tree')$. The proof for the second implication is similar.

\smallskip
\item For any signed building set~$B \in \building(\tree)$ distinct from~$B_\circ$, we have~$B \negDisjoint B_\circ \implies B \subseteq B_\circ'$ and~$B \posDisjoint B_\circ \implies B \supseteq B_\circ'$. Assume that~$B \negDisjoint B_\circ$. Since~$B \cap B_\circ = \varnothing$, we have~$B \subset Y \cup \{y\}$. Moreover, since~$B \cup B_\circ \notin \building(\tree)$, we have~$y \notin B$. Otherwise, adding~$B_\circ$ to~$B$ would preserve the negative convexity and the positive convexity of the complement. It follows that~$B \subseteq Y \subset B_\circ'$. The proof for the second implication is similar.

\smallskip
\item For any~$B,B' \in \building(\tree)$ distinct from~$B_\circ$, if~$B$ and~$B'$ are signed compatible in~$\nestedComplex(\tree)$, then they are signed compatible in~$\nestedComplex(\tree')$. If $B$ and~$B'$ are nested, they are signed compatible in~$\tree$ and in~$\tree'$. Therefore we only have to check that~${B \cup B' \notin \building(\tree) \implies B \cup B' \notin \building(\tree')}$ and~${B \cap B' \notin \building(\tree) \implies B \cap B' \notin \building(\tree')}$. Assume that~$B \cup B' \in \building(\tree') \ssm \building(\tree)$. According to Step~(b), we have~${B \cup B' = B_\circ'}$. Since~$Y' \cup \{x\}$ is not negative convex for any~$\varnothing \ne Y' \subseteq Y$, we must have~$B = \{x\}$ and~$B' = Y$ or the opposite. But~$\{x\}$ is not in~$\building(\tree)$ since its complement is not positive convex. The proof for the second implication is similar.

\end{enumerate}
From Steps~(b)\,--\,(e), we derive that the map~$\phi : \building(\tree) \to \building(\tree')$ defined by~$\phi(B_\circ) = B_\circ'$ and by ${\phi(B) = B}$ for~$B \in \building(\tree) \ssm \{B_\circ\}$ induces an isomorphism between the signed nested complexes~$\nestedComplex(\tree)$ and~$\nestedComplex(\tree')$.
\end{proof}

Observe that all transformations of Proposition~\ref{prop:nestedComplexPreservingOperations} preserve the underlying unsigned structure of the tree~$\tree$. Combining these transformations, we obtain the following statement.

\begin{corollary}
Consider two signed trees~$\tree$ and~$\tree'$ with the same underlying unsigned tree, and whose signs only differ on their legs (a \defn{leg} of~$\tree$ is a subtree of~$\tree$ containing at least a leaf of~$\tree$ and no vertex of~$\tree$ of degree~$3$ or higher). Then the signed nested complexes~$\nestedComplex(\tree)$ and~$\nestedComplex(\tree')$ are isomorphic. In particular, the signed nested complex of any signed path on~$\nu$ vertices is isomorphic to the simplicial \mbox{$(\nu-1)$-dimensional} associahedron.
\end{corollary}

To illustrate Proposition~\ref{prop:nestedComplexPreservingOperations}, we have represented in \fref{fig:exmTrees} three different signatures with the same underlying unsigned structure. Negative vertices are colored white, while positive ones are colored black. The signed nested complexes on the first two signed trees~$\tree_1$ and~$\tree_2$ of \fref{fig:exmTrees} are isomorphic since we have just changed all signs simultaneously and then the signs of the vertices~$5$, $6$, and~$8$ which belong to legs of the tree. However, these two simplicial complexes are not isomorphic to the signed nested complex of the third signed tree~$\tree_3$. To see it, we can observe\footnote{The author thanks Sonja \v Cuki\' c for this observation.} (by computer) that the complexes~$\nestedComplex(\tree_1)$ and~$\nestedComplex(\tree_3)$ do not even share the same vertex-facet incidence numbers. Indeed, the nested complexes~$\nestedComplex(\tree_1)$ and~$\nestedComplex(\tree_3)$ both have $165$ vertices and $143\,932$ facets, but the building block~$\{0,1,3,5,6,7\}$ of~$\building(\tree_1)$ is contained in precisely~$4\,000$ facets of~$\nestedComplex(\tree_1)$ while no building block of~$\building(\tree_3)$ is contained in precisely~$4\,000$ facets of~$\nestedComplex(\tree_3)$. A similar situation already happens for the ground tree with four leaves and two vertices of degree~$3$, for which different signatures yield non-isomorphic signed nested complexes.

\begin{figure}[b]
  \capstart
  \centerline{\includegraphics[scale=1]{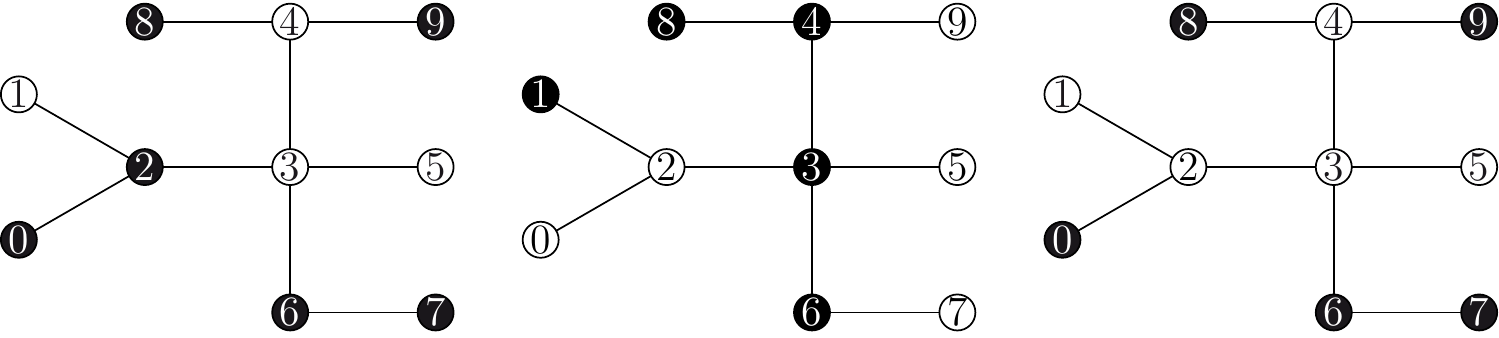}}
  \caption{Three different signed trees~$\tree_1, \tree_2, \tree_3$ on the same underlying tree.}
  \label{fig:exmTrees}
\end{figure}

\enlargethispage{-.1cm}
To conclude, we underline the problem to characterize isomorphic signed nested complexes.

\begin{question}
If~$\tree$ and~$\tree'$ are two signed ground trees with the same underlying unsigned structure such that $\nestedComplex(\tree)$ and $\nestedComplex(\tree')$ are isomorphic, are~$\tree$ and~$\tree'$ necessarily obtained from each other by a combination of the operations of Proposition~\ref{prop:nestedComplexPreservingOperations}?
\end{question}


\subsection{Links of signed nested complexes and phantom trees}
\label{subsec:links}

It is well-known that any face of the classical associahedron is a Cartesian product of smaller classical associahedra. In the simplicial setting, this translates to the fact that the link of any face of the simplicial associahedron is a join of smaller simplicial associahedra. A similar property also holds for unsigned graph associahedra (and for nestohedra), see~\cite[Theorem~2.9]{CarrDevadoss} and~\cite[Section~3]{Zelevinsky}. This property is no longer true for signed nested complexes on trees in general. However, we can force this property to the cost of a slight extension of signed nested complexes to what we call phantom trees.

A \defn{phantom tree} on a signed ground set~$\ground \eqdef \ground^- \sqcup \ground^+$ is a finite tree~$\phantomTree$ whose vertex set is partitioned into a set of \defn{standard vertices}, bijectively labeled by~$\ground$, and a set of~\defn{phantom vertices} which do not receive any label. For example, we have represented in \fref{fig:exmGeneralizedTrees} two phantom trees with ground sets~$\{1,3,5\} \sqcup \{0,2,7\}$ and~$\{4\} \sqcup \{6,8,9\}$, respectively. We then define open subtrees, signed tubes, signed building blocks on a phantom tree~$\phantomTree$, as well as the signed compatibility relation between them exactly as before. Note that the phantom vertices of~$\phantomTree$ should not be considered as vertices: they cannot belong to the boundary of an open subtree, and they are only used to fork some open subtrees. The family of signed nested complexes on phantom trees is now sufficiently rich to be closed under link.

\begin{figure}[h]
  \capstart
  \centerline{\includegraphics[scale=1]{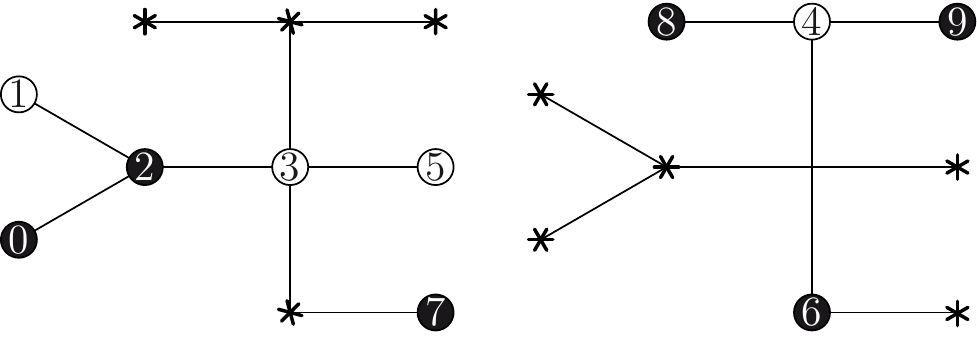}}
  \caption{The two phantom trees~$\underline\phantomTree$ and~$\overline\phantomTree$ corresponding to the signed tube of \fref{fig:exmOSSTSBB} in the signed ground tree~$\tree\ex$ of \fref{fig:exmTree}. Phantom vertices are marked with stars.}
  \label{fig:exmGeneralizedTrees}
\end{figure}

\begin{proposition}
The link of any face of the signed nested complex~$\nestedComplex(\phantomTree)$ on a phantom tree~$\phantomTree$ is a join of signed nested complexes on phantom trees.
\end{proposition}

\begin{proof}
It is enough to prove the result for a vertex of~$\nestedComplex(\phantomTree)$, since the statement for any face then follows by induction. Consider a relevant signed building block~$B \in \building(\phantomTree)$. Let~$\underline\phantomTree$ denote the phantom tree obtained from~$\phantomTree$ by turning to phantoms the vertices of~$\ground \ssm B$, and similarly let~$\overline\phantomTree$ denote the phantom tree obtained from~$\phantomTree$ by turning to phantoms the vertices of~$B$. See \fref{fig:exmGeneralizedTrees}. We claim that the link of~$B$ in~$\nestedComplex(\phantomTree)$ is isomorphic to the join of the signed nested complexes~$\nestedComplex(\underline\phantomTree)$ and~$\nestedComplex(\overline\phantomTree)$. This can be easily seen using the geometric representation of open subtrees as curves of~$\phantomTreeInterval \eqdef \phantomTree \times [-1,1]$. Indeed, the curve~$\chi \in \curves(\phantomTree)$ corresponding to~$B$ splits the space~$\phantomTreeInterval$ into two cells, the lower one containing the vertices of~$B$ and the upper one containing the vertices of~$\ground \ssm B$. The crossing-free subsets of curves in each of these cells thus correspond to crossing-free subsets of curves in~$\underline\phantomTree \times [-1,1]$ and~$\overline\phantomTree \times [-1,1]$ respectively. The result follows.
\end{proof}

To simplify our presentation, we focus on classical trees and we only consider phantom trees when we need to deal with links of signed nested sets on classical trees. We invite however the reader to check that we could extend the results of this paper to all phantom trees. Namely, the definition and properties of spines (Section~\ref{sec:spines}) directly translate to the case of phantom trees, and the geometric realizations of the signed nested complex as a complete simplicial fan (Section~\ref{sec:spineFan}) and as a convex polytope (Section~\ref{sec:signedTreeAssociahedron}) are then obtained from the properties of the spines.


\section{Signed spines}
\label{sec:spines}

In this section, we introduce and study spines on the ground tree~$\tree$. They generalize the definition of \emph{spines} (\aka \emph{mixed cobinary trees}) given independently in~\cite{LangePilaud-spines} and~\cite{IgusaOstroff} for the path associahedron. As for the path associahedron, they will play an essential role in this paper, both for combinatorial and geometric perspectives.

\subsection{Signed spine poset}

Consider a directed tree~$\spine$ whose vertices are labeled by non-empty subsets of~$\ground$. If~$r$ is an arc of~$\spine$, we call \defn{source label set} of~$r$ in~$\spine$ the union~$\source(r)$ of all labels which appear in the connected component of~$\spine \ssm \{r\}$ containing the source of~$r$. The \defn{sink label set}~$\sink(r)$ of~$r$ in~$\spine$ is defined similarly. Note that~$\source(r)$ and~$\sink(r)$ partition the label set of~$\spine$.

\begin{definition}
\label{def:spine}
A \defn{signed spine} on~$\tree$ is a directed and labeled tree~$\spine$ such that
\begin{enumerate}[(i)]
\item the labels of the nodes of~$\spine$ form a partition of the signed ground set~$\ground$, and
\item at a node of~$\spine$ labeled by~$U = U^- \sqcup U^+$, the source label sets of the different incoming arcs are subsets of distinct connected components of~$\tree \ssm U^-$, while the sink label sets of the different outgoing arcs are subsets of distinct connected components of~$\tree \ssm U^+$.
\end{enumerate}
\end{definition}

\fref{fig:exmSpines} represents two examples of signed spines on the signed ground tree~$\tree\ex$ of \fref{fig:exmTree} (for convenience of the reader, the ground tree~$\tree\ex$ is repeated on the left of \fref{fig:exmSpines}). In each label of the spines, we distinguish the negative vertices in white from the positive vertices in black. For example, the vertex \raisebox{-.15cm}{\includegraphics{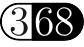}} of the rightmost spine of \fref{fig:exmSpines} has negative vertices $\{3\}$ and positive vertices~$\{6,8\}$. We can check the local Condition~(ii) of Definition~\ref{def:spine} around this vertex: indeed~$\{0,1\}$, $\{5\}$ and~$\{4,9\}$ are subsets of distinct connected components of $\tree\ex \ssm \{3\}$, while $\{2\}$ and~$\{7\}$ are subsets of distinct connected components of $\tree\ex \ssm \{6,8\}$.

\begin{figure}[h]
  \capstart
  \centerline{\includegraphics[scale=1]{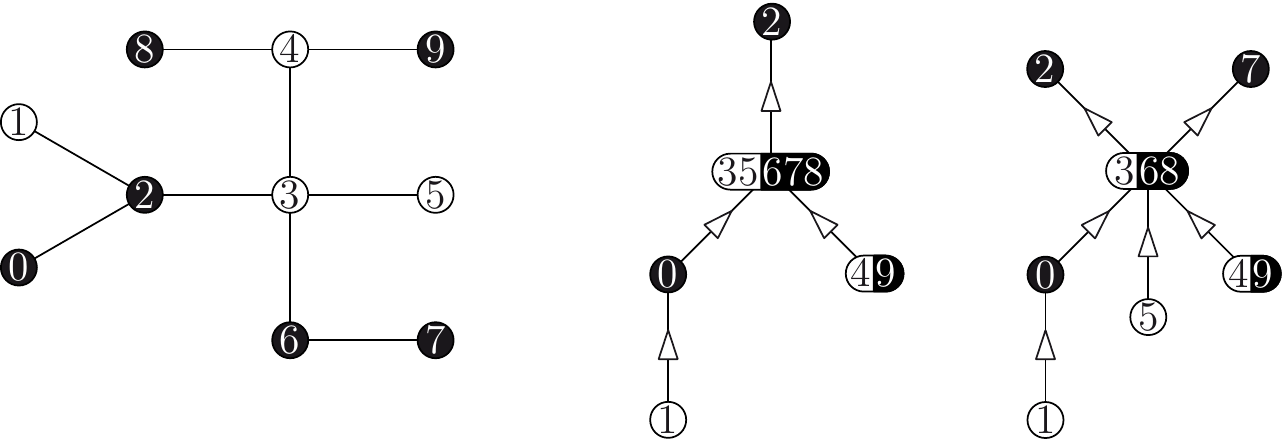}}
  \caption{Two signed spines (right) on the signed ground tree~$\tree\ex$ (left).}
  \label{fig:exmSpines}
\end{figure}

We now consider arc contraction and arc insertion in signed spines on~$\tree$.

\begin{lemma}
\label{lem:contraction}
Contracting an arc in a signed spine on~$\tree$ leads to a new signed spine on~$\tree$.
\end{lemma}

\begin{proof}
Let~$\spine$ be a signed spine on~$\tree$, and~$r$ be an arc of~$\spine$ from a node~$s$ to a node~$s'$ of~$\spine$, labeled by~$U = U^- \sqcup U^+$ and~$U' = U'^- \sqcup U'^+$, respectively. Let~$\bar\spine$ denote the directed and labeled tree obtained by contraction of~$r$ in~$\spine$, and let~$\bar s$ be the node of~$\bar\spine$ with label~$\bar U \eqdef U \cup U'$ obtained by merging the nodes~$s$ and~$s'$ of~$\spine$. The labels of the nodes of~$\bar\spine$ partition the ground set~$\ground$ and the local Condition~(ii) of Definition~\ref{def:spine} clearly holds around all nodes of~$\bar \spine$ distinct from~$\bar s$, since their incoming and outgoing arcs as well as their source and sink label sets are not modified by the contraction. To check this condition around~$\bar s$, let~$i_1, \dots, i_p$ and~$r, i'_1, \dots, i'_{p'}$ denote the incoming arcs of~$\spine$ at~$s$ and~$s'$ respectively. Then the incoming arcs of~$\bar\spine$ at~$\bar s$ are the the arcs~$i_1, \dots, i_p, i'_1, \dots, i'_{p'}$, and their source label sets belong to distinct connected components of~$\tree \ssm \bar U^-$. Indeed, $\source(i_\alpha)$ is separated from~$\source(i_\beta)$ by~$U^-$ in~$\tree$, and $\source(i'_{\alpha'})$ is separated from~$\source(i_\beta)$ and from~$\source(i'_{\beta'})$ by~$U'^-$ in~$\tree$, for all~$\alpha, \beta \in [p]$ and~$\alpha', \beta' \in [p']$. We prove similarly that the sink label sets of the outgoing arcs of~$\bar\spine$ at~$\bar s$ belong to distinct connected components of~$\tree \ssm \bar U^+$. This concludes the proof of the statement.
\end{proof}

\begin{lemma}
\label{lem:opening}
Let~$\spine$ be a signed spine on~$\tree$ with a node labeled by a set~$U$ containing at least two elements. For any~$u \in U$, there exists a signed spine on~$\tree$ whose nodes are labeled exactly as that of~$\spine$, except that the label~$U$ is partitioned into~$\{u\}$ and~$U \ssm \{u\}$. 
\end{lemma}

\begin{proof}
Assume that~$u \in U^-$. We consider the signed spine~$\spine'$ obtained from~$\spine$ replacing label~$U$ by~$U \ssm \{u\}$ and pulling below it a new node labeled by~$\{u\}$ together with all the incoming arcs whose source label set belongs to a connected component of~$\tree \ssm U^-$ incident to node~$u$. The proof that~$\spine'$ is a spine on~$\tree$ is very similar to that of the previous proof and left to the reader. The case~$u \in U^+$ is similar.
\end{proof}

Lemmas~\ref{lem:contraction} and~\ref{lem:opening} motivate the following definition.

\begin{definition}
The \defn{signed spine poset} on~$\tree$ is the poset~$\spinePoset(\tree)$ of arc contractions on the signed spines on~$\tree$.
\end{definition}

\begin{corollary}
The signed spine poset is a pure graded poset of rank~$\nu$.
\end{corollary}

For example, \fref{fig:exmMaximalSpines} shows two maximal signed spines on the ground tree~$\tree\ex$ of \fref{fig:exmTree}, which both refine the signed spines of \fref{fig:exmSpines}.

\begin{figure}[h]
  \capstart
  \centerline{\includegraphics[scale=1]{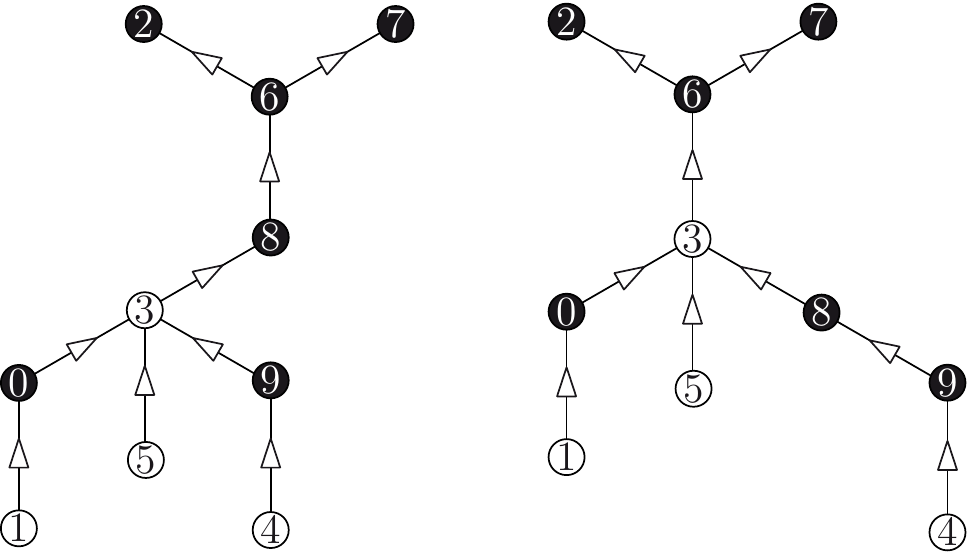}}
  \caption{Two maximal signed spines on the signed ground tree~$\tree\ex$ of \fref{fig:exmTree}.}
  \label{fig:exmMaximalSpines}
\end{figure}

\begin{example}[Unsigned path]
\label{exm:BST}
When the ground tree is a path~$\pathG$ labeled by~$[\nu]$ increasingly from one leaf to the other and with only negative signs, the maximal spines on~$\pathG$ are precisely the \defn{binary search trees} with label set~$[\nu]$, \ie the binary trees where the label of each node is larger than all labels in its left child and smaller than all labels in its right child. Since the labels of the nodes of a binary search tree can be reconstructed from the tree by infix labeling, the maximal spines on~$\pathG$ are in bijection with binary trees with~$\nu$ nodes. They are thus counted by the Catalan number~$C_\nu = \frac{1}{\nu + 1}\binom{2\nu}{\nu}$.
More generally, spines on~$\pathG$ are plane trees whose node label sets partition~$[\nu]$, and where a node labeled by~$\{u_1, \dots, u_k\}$ has $k+1$ children such that all labels of the $i$\ordinal{} child are strictly inbetween $u_{i-1}$ and~$u_i$ (where by convention~$u_0 = 0$ and $u_{k+1} = \nu + 1$).
\end{example}


\subsection{From signed spines to signed nested sets}

We now explore the connection between signed spines and signed nested sets to show that the signed spine poset~$\spinePoset(\tree)$ is isomorphic to the inclusion poset of the signed nested complex~$\nestedComplex(\tree)$.

\begin{lemma}
\label{lem:arcToBuildingBlock}
For any arc~$r$ of a signed spine~$\spine$ on~$\tree$, the source label set~$\source(r)$ is a relevant signed building set of~$\tree$.
\end{lemma}

\begin{proof}
We have to prove that the source label set~$\source(r)$ is negative convex while the sink label set~$\sink(r)$ is positive convex. Assume by contradiction that~$v \in \ground^-$ lies in between~$u$ and~$w$ in~$\tree$ and that~$u,w \in \source(r)$ and~$v \notin \source(r)$. Consider the path~$\pi$ in the signed spine~$\spine$ from the arc~$r$ to the node~$s$ whose label set contains~$v$. If the last arc~$r'$ of~$\pi$ is incoming at~$s$, then the node~$s$ contradicts the local Condition~(ii) of Definition~\ref{def:spine}, since~$u$ and~$w$ lie in the same incoming source label set~$\source(r')$, but in distinct connected components of~$\tree \ssm \{v\}$. Otherwise, the last arc~$r'$ of~$\pi$ is outgoing at~$s$, and there is thus a node~$s'$ of~$\spine$ where the path~$\pi$ has two incoming arcs. This node~$s'$ contradicts again the local Condition~(ii) of Definition~\ref{def:spine}, since $\{u,w\}$ and $v$ lie in distinct incoming source label sets, but $v$ lies in between~$u$ and~$w$ in~$\tree$. We prove similarly that~$\sink(r)$ is positive convex. Finally, the source label set~$\source(r)$ is relevant: it is neither~$\varnothing$ nor~$\ground$ since~$r$ has at least one vertex in its source and one vertex in its sink.
\end{proof}

\begin{lemma}
For any signed spine~$\spine$ on~$\tree$, the collection~$\spineToNested(\spine) \eqdef \set{\source(r)}{r \text{ arc of } \spine}$ is a signed nested set of~$\nestedComplex\building(\tree)$.
\end{lemma}

\begin{figure}
  \capstart
  \centerline{
  	\begin{tabular}{c@{\hspace{1cm}}c@{\hspace{1cm}}c}
	\includegraphics{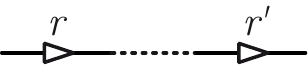} & \includegraphics{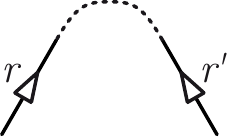} & \includegraphics{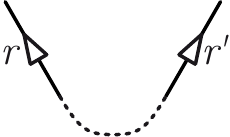} \\[.2cm]
	$\source(r) \subseteq \source(r')$ & $\source(r) \negDisjoint \source(r')$ & $\source(r) \posDisjoint \source(r')$
	\end{tabular}
  }
  \caption{The relative position of two arcs~$r$ and~$r'$ in a signed spine~$\spine$ determines the compatibility relation between their source label sets~$\source(r)$ and~$\source(r')$.}
  \label{fig:relativePositions}
\end{figure}

\begin{proof}
Consider two arcs~$r$ and~$r'$ of~$\spine$. Since~$\spine$ is a tree, there is a path~$\pi$ in~$\spine$ between~$r$ and~$r'$. If this path~$\pi$ connects the head of one arc to the tail of the other, then the source label sets~$\source(r)$ and~$\source(r')$ are nested. In contrast, if the path~$\pi$ connects the two heads of~$r$ and~$r'$, then their source label sets~$\source(r)$ and~$\source(r')$ are separated by the label set of any node of~$\spine$ where~$\pi$ has two incoming arcs. Therefore, $\source(r) \negDisjoint \source(r')$. Similarly, if the path~$\pi$ connects the two tails of~$r$ and~$r'$, then $\source(r) \posDisjoint \source(r')$. See \fref{fig:relativePositions} for an illustration.
\end{proof}

\begin{theorem}
\label{theo:spineToNested}
The map~$\spineToNested$ is a poset isomorphism between the signed spine poset~$\spinePoset(\tree)$ and the inclusion poset of the signed nested complex~$\nestedComplex\building(\tree)$ on~$\tree$.
\end{theorem}

\begin{proof}
Observe first that the injectivity of~$\spineToNested$ follows from the fact that two directed and labeled trees with the same source label sets coincide. To see it, we show that we can reconstruct a spine~$\spine$ from its source label sets. First the labels of the nodes of~$\spine$ are the equivalence classes under the relation~$v \equiv w$ if there is no arc~$r$ of~$\spine$ such that~$|\{v,w\} \cap \source(r)| = 1$. Second, the leaves of~$\spine$ correspond to the source label sets containing either only one or all but one label sets of~$\spine$. We can then delete one leaf~$\ell$ of~$\spine$, and reconstruct by induction the tree~$\spine \ssm \{\ell\}$. Finally the only possible vertex of~$\spine \ssm \{\ell\}$ to which the leaf~$\ell$ can be glued is the unique vertex which is in the intersection of all source sets~$B$ of~$\spine \ssm \{\ell\}$  such that~$B \cup \ell$ is a source set of~$\spine$, but not in the union of all source sets~$B$ of~$\spine \ssm \{\ell\}$  such that~$B \cup \ell$ is not a source set of~$\spine$.

The surjectivity of~$\spineToNested$ is more difficult to show. We prove by induction that for any signed nested set~$\nested \in \nestedComplex\building(\tree)$,
\begin{enumerate}[(a)]
\item there exists a signed spine~$\spine$ on~$\tree$ such that~$\spineToNested(\spine) = \nested$, and
\item for any signed building block~$B$ of~$\tree$ which is signed compatible with~$\nested$, there exists a unique node~$\buildingBlockToVertex(B,\spine)$ of~$\spine$ such that  for any incoming arc~$i$ of~$\spine$ at~$\buildingBlockToVertex(B,\spine)$, either $\source(i) \subseteq B$ or~$\source(i) \negDisjoint B$, while for any outgoing arc~$o$ of~$\spine$ at~$\buildingBlockToVertex(B,\spine)$, either $\source(o) \supseteq B$ or~$\source(o) \posDisjoint B$.
\end{enumerate}
These properties are proved by induction on the size of~$\nested$. To initialize, observe that the minimal spine with only one vertex labeled by the complete ground set~$\ground$ is sent by~$\spineToNested$ to the empty signed nested set, and that Property~(b) above clearly holds. Consider now any signed nested set~${\nested \in \nestedComplex\building(\tree)}$, and let~$B_\circ$ be an arbitrary signed building block of~$\nested$ and~$\nested_\circ \eqdef \nested \ssm \{B_\circ\}$. By induction hypothesis, there exists a spine~$\spine_\circ$ such that~$\spineToNested(\spine_\circ) = \nested_\circ$ and a unique vertex~$\buildingBlockToVertex_\circ \eqdef \buildingBlockToVertex(B_\circ, \spine_\circ)$ in~$\spine_\circ$ satisfying Property~(b) above since $B_\circ$ is compatible with~$\nested_\circ$. We now prove that we can split the node~$\buildingBlockToVertex_\circ$ of~$\spine_\circ$ into two nodes~$\underline s$ and~$\overline s$, related by an arc~$r$, such that the resulting signed spine~$\spine$ is sent by~$\spineToNested$ to the signed nested set~$\spineToNested(\spine) = \nested$.

Let~$U$ be the label set of the node~$\buildingBlockToVertex_\circ$ of~$\spine_\circ$, and let~$\underline U \eqdef U \cap B_\circ$ and~$\overline U \eqdef U \ssm B_\circ$.
Let~$i_1, \dots, i_p$ and~$o_1, \dots, o_q$ denote the incoming and outgoing arcs of~$\spine_\circ$ at node~$\buildingBlockToVertex_\circ$, and let~$X \subseteq [p]$ and~$Y \subseteq [q]$ be the set of indices such that~$\source(i_x) \subseteq B_\circ \subseteq \source(o_y)$ for~$x \in X$ and~$y \in Y$. We replace the node~$\buildingBlockToVertex_\circ$ of~$\spine_\circ$ by a node~$\underline s$ labeled by~$\underline U$ and a node~$\overline s$ labeled by~$\overline U$ related by an arc~$r$ from~$\underline s$ to~$\overline s$. Moreover, apart from the arc~$r$, the node~$\underline s$ has incoming arcs~$i_x$ for~$x \in X$ and outgoing arcs~$o_z$ for~$z \notin Y$, while $\overline s$ has incoming arcs~$i_z$ for~$z \notin X$ and outgoing arcs~$o_y$ for~$y \in Y$. The resulting directed and labeled tree is denoted by~$\spine$.

We claim that~$\spine$ is a signed spine on~$\tree$ and that~$\spineToNested(\spine) = \nested$. Observe first that the source label set~$\source(r)$ of the new arc~$r$ in~$\spine$ is the union of the label~$\overline U$, of the source label sets~$\source(i_x)$ for~$x \in X$ and of the sink label sets~$\sink(o_z)$ for~$z \notin Y$. All these sets are subsets of~$B_\circ$: by assumption~$\source(i_x) \subseteq B_\circ$ for all~$x \in X$, and~$\source(o_z) \posDisjoint B_\circ$ so that~$\source(o_z) \cup B_\circ = \ground$ and~$\sink(o_z) \subseteq B_\circ$ for all~$z \notin Y$. Therefore, $\source(r) \subseteq B_\circ$ and we prove similarly that~$\sink(r) \subseteq \ground \ssm B$. We conclude that~$\source(r) = B_\circ$. Moreover, we did not perturb the source label sets of the arcs of~$\spine_\circ$ while opening the node~$\buildingBlockToVertex_\circ$ in~$\spine_\circ$, and~thus
\[
\spineToNested(\spine) = \set{\source(t)}{t \text{ arc of } \spine} = \{\source(r)\} \cup \set{\source(t)}{t \text{ arc of } \spine_\circ} = \{B_\circ\} \cup \nested_\circ = \nested.
\]
We next prove that~$\spine$ is indeed a signed spine on~$\tree$. First, the label sets of~$\spine$ clearly partition~$\ground$ since we only split the label of~$\buildingBlockToVertex_\circ$ in the spine~$\spine_\circ$. Moreover, while opening the node~$\buildingBlockToVertex_\circ$ in~$\spine_\circ$, we did not perturb the labels and the arcs incident to the other nodes of~$\spine_\circ$. It follows that the local Condition~(ii) of Definition~\ref{def:spine} is still fulfilled around these nodes in~$\spine$. It remains to check this local condition for the nodes~$\underline s$ and~$\overline s$. The incoming arcs of~$\spine$ at~$\underline s$ are the arcs~$i_x$ for~$x \in X$, whose source label sets~$\source(r)$ are all contained in~$B_\circ$. Since they are separated in~$\tree$ by the vertices of~$U^-$, they remain separated by the vertices of~$U^- \cap B_\circ$. The outgoing arcs of~$\spine$ at~$\underline s$ are the arc~$r$ for which~$\sink(r) = \ground \ssm B_\circ$ and the arcs~$o_z$ for~$z \notin Y$, for which~$\source(o_z) \posDisjoint B_\circ$, so that~$\sink(o_z) \subseteq B_\circ$. Since they are separated in~$\tree$ by the vertices of~$U^+$, they remain separated by the vertices of~$U^+ \ssm B_\circ$. The proof is symmetric for the node~$\overline s$ of~$\spine$.

Finally, we have to prove that the induction Property~(b) still holds in~$\spine$. For a given signed building block~$B$ compatible with~$\nested$, we want to find a node~$\buildingBlockToVertex(B, \spine)$ which satisfies Property~(b). Observe that this node must be unique: otherwise, opening independently each node of~$\spine$ satisfying Property~(b) would result in distinct spines with signed nested set~$\nested \cup \{B\}$ which was already excluded. To prove the existence of~$\buildingBlockToVertex(B, \spine)$, we use the induction hypothesis. Since~$B$ is compatible with~$\nested_\circ \subset \nested$, there exists a node~$\buildingBlockToVertex(B, \spine_\circ)$ in~$\spine_\circ$ which satisfies Property~(b). If this node is distinct from~$\buildingBlockToVertex_\circ = \buildingBlockToVertex(B_\circ, \spine_\circ)$, then we do not perturb the signed building blocks corresponding to its incoming and outgoing arcs while opening~$\buildingBlockToVertex_\circ$, so that the node~$\buildingBlockToVertex(B, \spine) \eqdef \buildingBlockToVertex(B, \spine_\circ)$ still fits. Finally, if~$\buildingBlockToVertex(B_\circ, \spine_\circ) = \buildingBlockToVertex(B, \spine_\circ)$, then we choose~$\buildingBlockToVertex(B, \spine) \eqdef \underline s$ if~$B \subseteq B_\circ$ or~$B \negDisjoint B_\circ$ and $\buildingBlockToVertex(B, \spine) \eqdef \overline s$ if~$B \supseteq B_\circ$ or~$B \posDisjoint B_\circ$.
\end{proof}

We denote by~$\nestedToSpine : \nestedComplex\building(\tree) \to \spinePoset(\tree)$ the inverse map of~$\spineToNested : \spinePoset(\tree) \to \nestedComplex\building(\tree)$. Since we have Theorem~\ref{theo:spineToNested}, we can give a direct description of this map~$\nestedToSpine$. Namely, consider two signed building blocks~$B$ and~$B'$ of a signed nested set~$\nested \in \nestedComplex\building(\tree)$, and let~$r$ and~$r'$ denote the arcs of~$\nestedToSpine(\nested)$ such that~$B = \source(r)$ and~$B' = \source(r')$.~Then,
\begin{enumerate}[(i)]
\item the head of~$r$ and the tail of~$r'$ coincide iff~$B \subseteq B'$ and~$\nexists \, B'' \in \nested$ with~$B \subseteq B'' \subseteq B'$ or~$B \negDisjoint B'' \posDisjoint B'$;
\item the heads of~$r,r'$ coincide iff~$B \negDisjoint B'$ and~$\nexists \, B'' \in \nested$ with~$B \subseteq B'' \negDisjoint B'$ or~$B \negDisjoint B'' \supseteq B'$;
\item the tails of~$r,r'$ coincide iff~$B \posDisjoint B'$ and~$\nexists \, B'' \in \nested$ with~$B \posDisjoint B'' \subseteq B'$ or~$B \supseteq B'' \posDisjoint B'$.
\end{enumerate}
This gives a description of the vertices of the spine~$\nestedToSpine(\nested)$ in terms of equivalence classes of signed building blocks of~$\nested$ (by the relations above). It also gives a direct definition of the directed graph underlying~$\nestedToSpine(\nested)$ as a quotient of a collection of disjoint arcs labeled by~$\nested$ by identification of some of their endpoints. For each node~$s$ of~$\nestedToSpine(\nested)$ with incoming arcs~$I$ and outgoing arcs~$O$, the label of~$s$ can then be computed as
\[
\bigg( \bigcap_{o \in O} \source(o) \bigg) \ssm \bigg( \bigcup_{i \in I} \source(i) \bigg) = \ground \ssm \bigg( \bigcup_{i \in I} \source(i) \cup \bigcup_{o \in O} \sink(o) \bigg).
\]

To conclude, we present an alternative way to visualize the correspondence between signed spines and signed nested sets on~$\tree$, based on open subtrees and their representation in the space~${\treeInterval \eqdef \tree \times [-1,1]}$. A signed nested set~$\nested \in \nestedComplex\tubes(\tree)$ can be seen as a collection~$\nestedToCurves(\nested)$ of non-crossing curves in~$\treeInterval$ which lift the open subtrees~$\tubeToSubtree(W)$ for~$W \in \nested$. Each curve~$\subtreeToCurve \in \nestedToCurves(\nested)$ splits~$\treeInterval$ into two connected components, one below~$\subtreeToCurve$ and one above~$\subtreeToCurve$. As illustrated in \fref{fig:exmDissection}, the curves of~$\nestedToCurves(\nested)$ dissect~$\treeInterval$ into distinct cells. We say that a (lifted) vertex~$(v, \pm 1)$ of a cell~$C$ is \defn{extremal} if~$v$ is a leaf of the vertical projection of~$C$ and $(v,0) \notin C$. The other vertices are call \defn{intermediate vertices}. In \fref{fig:exmDissection}, we have erased the extremal vertices in each cell so that only the intermediate ones appear. The signed spine~$\nestedToSpine(\nested)$ is precisely the \defn{directed and labeled dual tree}~$\curvesToSpine(\nestedToCurves(\nested))$ of the cell decomposition of~$\treeInterval$ by~$\nestedToCurves(\nested)$, defined as the tree with
\begin{enumerate}[(i)]
\item a node for each cell~$C$ of $\treeInterval \ssm \nestedToCurves(\nested)$, labeled by the intermediate vertices of~$C$,~and
\item an arc for each curve~$\subtreeToCurve$ of~$\nestedToCurves(\nested)$, directed from the cell below~$\subtreeToCurve$ to the cell above~$\subtreeToCurve$.
\end{enumerate}
For example, the first signed spine of \fref{fig:exmSpines} is the dual spine of the dissection of \fref{fig:exmDissection}. This alternative definition of spines will be helpful for further considerations. To conclude, we observe that signed spines were already considered in the situations of unsigned trees and signed paths.

\begin{example}[Unsigned tree, continued]
If~$\tree$ has only negative vertices, the spine~$\nestedToSpine(\nested)$ is just isomorphic to the Hasse diagram of the nested poset of the building blocks of~$\nested$.
\end{example}

\begin{example}[Signed path, continued]
When the ground tree is a signed path~$\pathG$, the spine of a dissection of the corresponding $(\nu+2)$-gon~$\polygon$ is given by its directed and labeled dual tree. See \fref{fig:spinesPath} for some illustrations. These spines have been introduced by C.~ Lange and the author in~\cite{LangePilaud-spines} to revisit C.~Hohlweg and C.~Lange's constructions of the classical associahedron, and they motivated the definition of spines in this paper.

\begin{figure}[h]
  \capstart
  \centerline{\includegraphics[scale=.85]{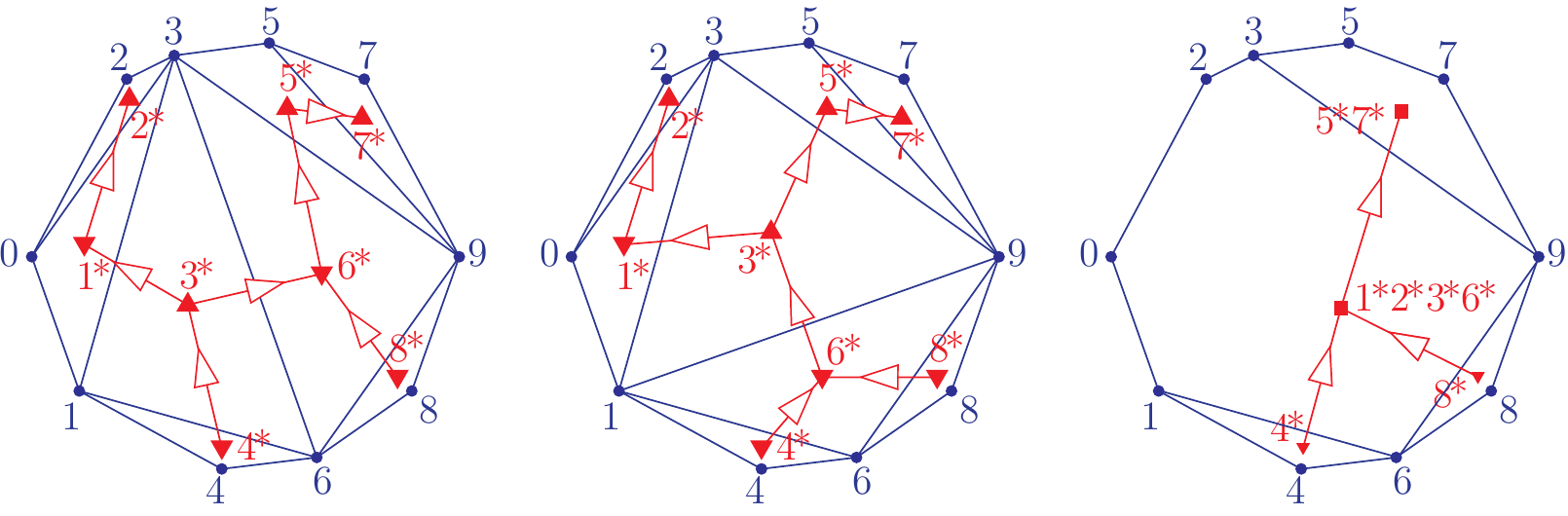}}
  \caption{Correspondence between dissections of the $(\nu+2)$-gon~$\polygon$ and their corresponding signed spines.}
  \label{fig:spinesPath}
\end{figure}
\end{example}

\begin{example}[Tripod, continued]
\fref{fig:tripodSpines} shows all the maximal signed spines on the tripods~$\tripodWhite$ and~$\tripodBlack$, up to automorphisms of these trees. All other spines are obtained from those by contractions of arcs. The reader is invited to fill-in the cells of the diagrams of \fref{fig:tripodNestedComplex} with the corresponding maximal signed spines.

\begin{figure}[h]
  \capstart
  \centerline{\includegraphics[scale=1]{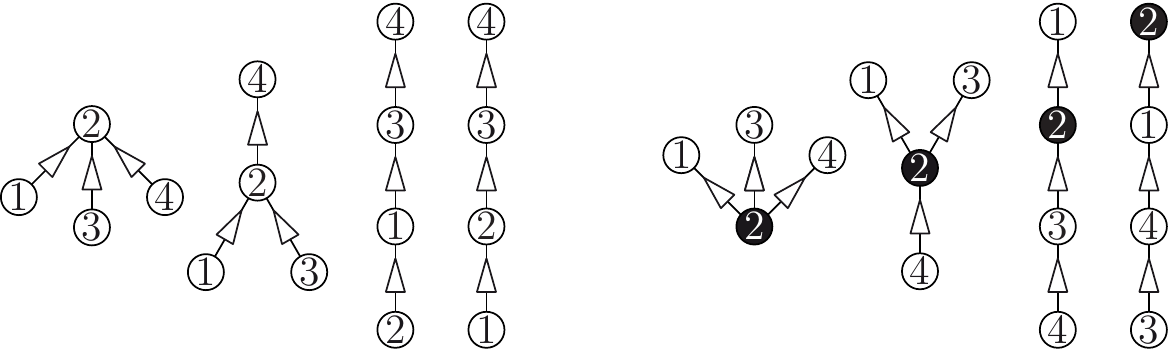}}
  \caption[All possible maximal signed spines (up to tree automorphisms) on the tripods.]{All possible maximal signed spines (up to tree automorphisms) on the tripods~$\tripodWhite$ and~$\tripodBlack$.}
  \label{fig:tripodSpines}
\end{figure}
\end{example}


\subsection{Flips of maximal signed spines}

We now define flips on spines, an operation which transforms a maximal signed spine~$\spine$ on~$\tree$ into a new one~$\spine'$ such that~$\spineToNested(\spine)$ and~$\spineToNested(\spine')$ are adjacent facets of the signed nested complex~$\nestedComplex(\tree)$. Since we deal with maximal spines here, the nodes of~$\spine$ are labeled by singletons. We therefore abuse notation by identifying a node with the unique element of its label.

\begin{definition}
\label{def:flipSpine}
Consider an arc~$r$ in a maximal signed spine~$\spine$ on~$\tree$, from a node~$u \in \ground$ to a node~$v \in \ground$. If they exist, let~$i$ denote the incoming arc of~$\spine$ at~$u$ whose source label set lies in the connected component of~$\tree \ssm \{u\}$ containing~$v$ and let~$o$ denote the outgoing arc of~$\spine$ at~$v$ whose sink label set lies in the connected component of~$\tree \ssm \{v\}$ containing~$u$. Let~$\spine'$ denote the tree obtained from~$\spine$ by reversing the arc~$r$ to an arc~$r'$ from~$v$ to~$u$, and attaching the arc~$i$ to~$v$ and the arc~$o$ to~$u$. We say that~$\spine'$ is obtained from~$\spine$ by \defn{flipping}~$r$, and that~$\spine$ and~$\spine'$ are \defn{related by a flip}. \fref{fig:flip} illustrates the four possible situations, according on whether~$u$ and~$v$ belong to~$\ground^-$ or~$\ground^+$.
\end{definition}

\begin{figure}[h]
  \capstart
  \centerline{\includegraphics[scale=.95]{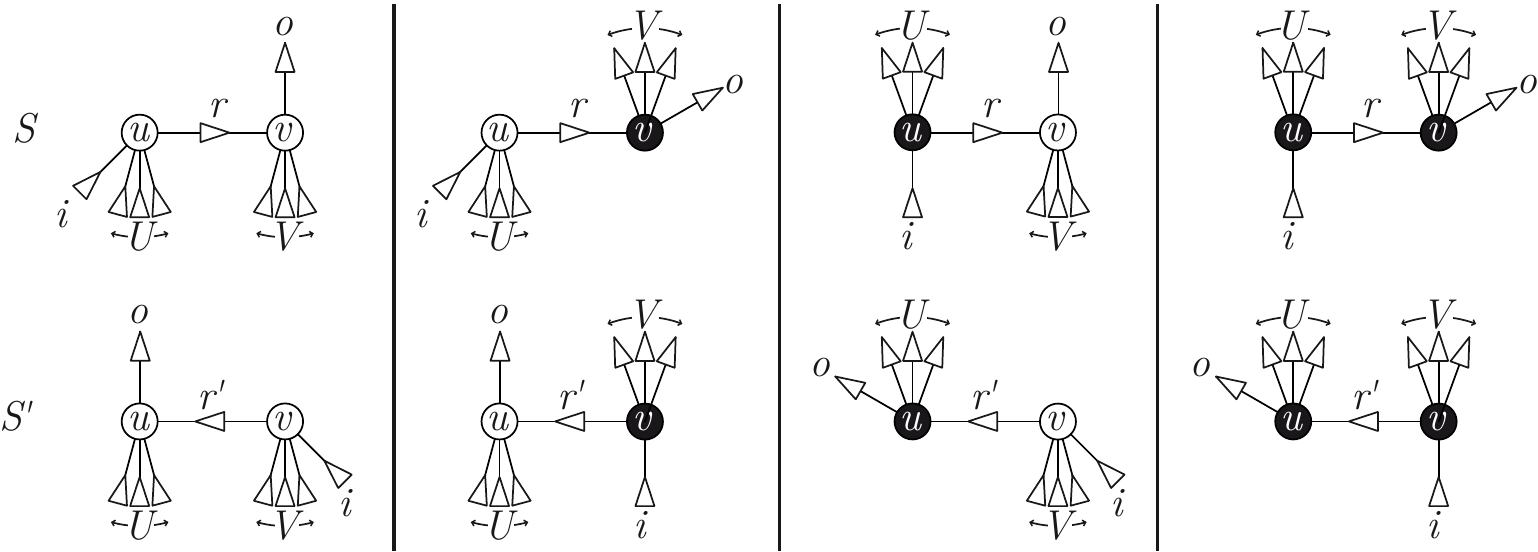}}
  \caption{Flipping an edge in a maximal signed spine. The four possible situations (whether~$u$ and~$v$ belong to~$\ground^-$ or~$\ground^+$) are illustrated. White nodes represent negative nodes while black nodes represent positive nodes.}
  \label{fig:flip}
\end{figure}

For example, the two spines of \fref{fig:exmMaximalSpines} are obtained from each other by flipping the arc joining~$3$ and~$8$.

\begin{example}[Unsigned path, continued]
When the ground tree is a path~$\pathG$ labeled by~$[\nu]$ increasingly from one leaf to the other and with only negative signs, we have seen in Example~\ref{exm:BST} that the maximal spines are precisely the binary search trees with label set~$[\nu]$. The flip operation is then the classical \defn{rotation} in binary search trees, which is used in algorithms to balance them.
\end{example}

\begin{proposition}
\label{prop:flipSpine}
The tree~$\spine'$ is a signed spine on~$\tree$. Moreover, $\spine$ and~$\spine'$ are the only two maximal signed spines on~$\tree$ refining the spine~$\bar\spine$ obtained from~$\spine$ (or~$\spine'$) by contracting the arc~$r$ (or~$r'$).
\end{proposition}

\begin{proof}

We first prove that the tree~$\spine'$ is a maximal spine on~$\tree$. Its vertices are indeed bijectively labeled by the elements of~$\ground$. While performing the flip of~$r$ in~$\spine$, we did not perturb the label and arcs incident to the nodes of~$\spine$ distinct from~$u$ and~$v$. It follows that the local Condition~(ii) of Definition~\ref{def:spine} around these nodes is still fulfilled in~$\spine'$. To prove that this condition also holds around~$u$ and~$v$, we need some notations. Let~$I(u)$ be the set of incoming arcs of~$\spine$ at~$u$ distinct from~$i$, let~$O(u)$ be the set of outgoing arcs of~$\spine$ at~$u$ distinct from~$r$, let~$I(v)$ be the set of incoming arcs of~$\spine$ at~$v$ distinct from~$r$, and let~$O(v)$ be the set of outgoing arcs of~$\spine$ at~$v$ distinct from~$o$. See \fref{fig:flipAll}. As  illustrated in \fref{fig:flip}, $O(u) = \varnothing$ if~$u \in \ground^-$ while~$I(u) = \varnothing$ if~$u \in \ground^+$, and similarly~$O(v) = \varnothing$ if~$v \in \ground^-$ while~$I(v) = \varnothing$ if~$v \in \ground^+$. However, we treat the four possible situations together to avoid a useless case analysis. We also denote by~$X$ the connected component of~$\tree \ssm \{u\}$ containing~$v$ and by~$Y$ the connected component of~$\tree \ssm \{v\}$ containing~$u$. Note that~$X \cup Y = \tree$ and that~$X \cap Y$ is the open subtree between~$u$ and~$v$ in~$\tree$. According to the local Condition~(ii) of Definition~\ref{def:spine} around~$u$ in~$\spine$, the source label sets of the arcs of~$I(u)$ and the sink label sets of the arcs of~$O(u)$ are disjoint from~$X$ and therefore lie in~$Y$. Similarly, the source label sets of the arcs of~$I(v)$ and the sink label sets of the arcs of~$O(v)$ are disjoint from~$Y$ and therefore lie in~$X$.

\begin{figure}
  \capstart
  \centerline{\includegraphics[width=.45\textwidth]{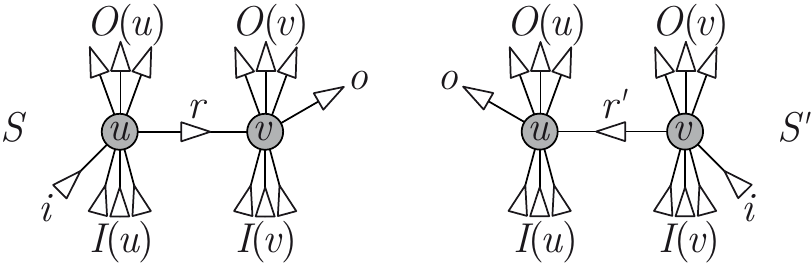}}
  \caption{Flipping an arc~$r$ in a spine. Notations of the proof of Proposition~\ref{prop:flipSpine}.}
  \label{fig:flipAll}
\end{figure}

We first prove that the incoming arcs of~$\spine'$ at~$u$ are separated by~$\{u\}^-$. If~$u \in \ground^+$, then~$u$ has only one incoming arc~$r'$ in~$\spine'$ and there is nothing to prove. Otherwise, the incoming arcs of~$\spine'$ at~$u$ are~$r'$ and the arcs of~$I(u)$. The source label set of~$r'$ in~$\spine'$ is
\[
\source(r') = \{v\} \; \cup \; \source(i) \; \cup \bigcup_{x \in I(v)} \source(x) \; \cup \bigcup_{y \in O(v)} \sink(y),
\]
and is thus contained in~$X$. It follows that~$u$ separates in~$\tree$ the source label sets of the incoming arcs of~$\spine'$ at~$u$. Next, we prove that the outgoing arcs of~$\spine'$ at~$u$ are separated by~$\{u\}^+$ in~$\tree$. If~$u \in \ground^-$, then~$u$ has only one outgoing arc~$o$ in~$\spine'$ and there is nothing to prove. Otherwise, the outgoing arcs of~$\spine'$ at~$u$ are~$o$ and the arcs of~$O(u)$. But the sink set of~$o$ lies in~$X$ while the sink sets of the arcs of~$O(u)$ are disjoint from~$X$. Therefore, $u$ separates in~$\tree$ the sink sets of the outgoing arcs of~$\spine'$ at~$u$. We prove similarly that the local Condition~(ii) of Definition~\ref{def:spine} is fulfilled around the node~$v$ of~$\spine'$.

It is clear that the contraction of~$r$ in~$\spine$ and the contraction of~$r'$ in~$\spine'$ both produce the same tree~$\bar\spine$ since the contracted node is incident to the same arcs with the same subspines. Finally, we prove that~$\spine$ and~$\spine'$ are the only two maximal spines on~$\tree$ which contract to~$\bar\spine$. If~$\spine''$ contracts to~$\bar\spine$, then it has an arc~$r''$ between~$u$ and~$v$ and the incoming and outgoing arcs of~$\bar\spine$ incident to the contracted vertex labeled by~$\{u,v\}$ are distributed in~$\spine''$ between the nodes~$u$ and~$v$. Assume \eg that~$r''$ is directed from~$v$ to~$u$. Since the source label sets of the arcs of~$I(u)$ are disjoint from~$X$, these arcs cannot be incident to~$v$ for the local Condition~(ii) of Definition~\ref{def:spine} around~$u$ to be fulfilled. Moreover, since the sink label sets of the arcs of~$O(u)$ are contained in~$Y$, these arcs cannot be incident to~$v$ for the local Condition~(ii) of Definition~\ref{def:spine} around~$v$ to be fulfilled. Similarly, we prove that the arcs of~$I(v)$ and~$O(v)$ are necessarily incident to~$v$. Finally, $i$ cannot be incident to~$u$ otherwise~$u$ would have two incoming arcs~$i$ and~$r$ with source label set in~$X$, and similarly, $o$ cannot be incident to~$v$. It follows that~$\spine'' = \spine'$. In other words, the maximal spine~$\spine'$ resulting of the flip of~$r$ in~$\spine$ is uniquely determined.
\end{proof}

\begin{corollary}
The signed nested complex~$\nestedComplex(\tree)$ is a closed pseudo-manifold.
\end{corollary}

\begin{definition}
The \defn{spine flip graph} is the graph~$\spineFlipGraph(\tree)$ whose vertices correspond to the maximal signed spines on~$\tree$ and whose arcs correspond to flips between them. In other words, it is the facet-ridge graph of the signed nested complex~$\nestedComplex(\tree)$.
\end{definition}

\begin{remark}
The flip operation can also be transported via the maps studied in the previous section to the signed nested complexes~$\nestedComplex\tubes(\tree)$, $\nestedComplex\building(\tree)$ and~$\nestedComplex\curves(\tree)$. Namely, for any maximal nested set~$\nested \in \nestedComplex\tubes(\tree)$ and any signed tube~$W \in \nested$, there is a unique maximal nested set~$\nested' \neq \nested$ of~$\tree$ and a unique signed tube~$W' \in \nested'$ such that~$\nested \ssm \{W\} = \nested' \ssm \{W'\}$. A similar statement holds in~$\nestedComplex\building(\tree)$. Moreover, in any maximal dissection of~$\treeInterval \eqdef \tree \times [-1,1]$, deleting a single curve produces a cell with two intermediate vertices, which can be uniquely splitted again with a new curve of~$\curves(\tree)$.
\end{remark}

\begin{example}[Signed path, continued]
If the ground tree~$\tree$ is a signed path, this operation is the well-known flip on triangulations of the corresponding polygon~$\polygon$. For example, the two triangulations of \fref{fig:spinesPath} and their spines are obtained from each other by flipping the diagonal~$(3,6)$ to the diagonal~$(1,9)$. Flips on spines of triangulations were studied in details in~\cite{LangePilaud-spines}.
\end{example}


\subsection{Blossoming spines and their cuts}

In this section, we present a slight modification of signed spines needed for further considerations. We have seen that each arc~$r$ of a signed spine on~$\tree$ corresponds to a relevant signed building block~$\source(r) \in \building(\tree)$, or equivalently to a relevant open subtree~${\arcToSubtree(r) \in \subtrees(\tree)}$. The idea here is to attach to the signed spine some \defn{blossoms} (\ie half-arcs) corresponding to the irrelevant open subtrees of~$\subtrees(\tree)$ (\ie the connected components of~$\tree \ssm \ground^-$ and that of~$\tree \ssm \ground^+$). Although clearly equivalent to our original definition of signed spines, blossoming spines are useful for some arguments later.

\begin{definition}
A \defn{blossoming spine} on~$\tree$ is a signed spine on~$\tree$ together with some additional incoming and outgoing blossoms (half-arcs) attached to its nodes such that the incoming and outgoing degrees of a node~$s$ labeled by~$U = U^- \sqcup U^+$ coincide with the number of connected components in~$\tree \ssm U^-$ and~$\tree \ssm U^+$, respectively.
\end{definition}

\fref{fig:exmBlossomingSpines} shows the blossoming spines corresponding to the signed spines of Figures~\ref{fig:exmSpines} and~\ref{fig:exmMaximalSpines}.

\begin{figure}[h]
  \capstart
  \centerline{\includegraphics[scale=1]{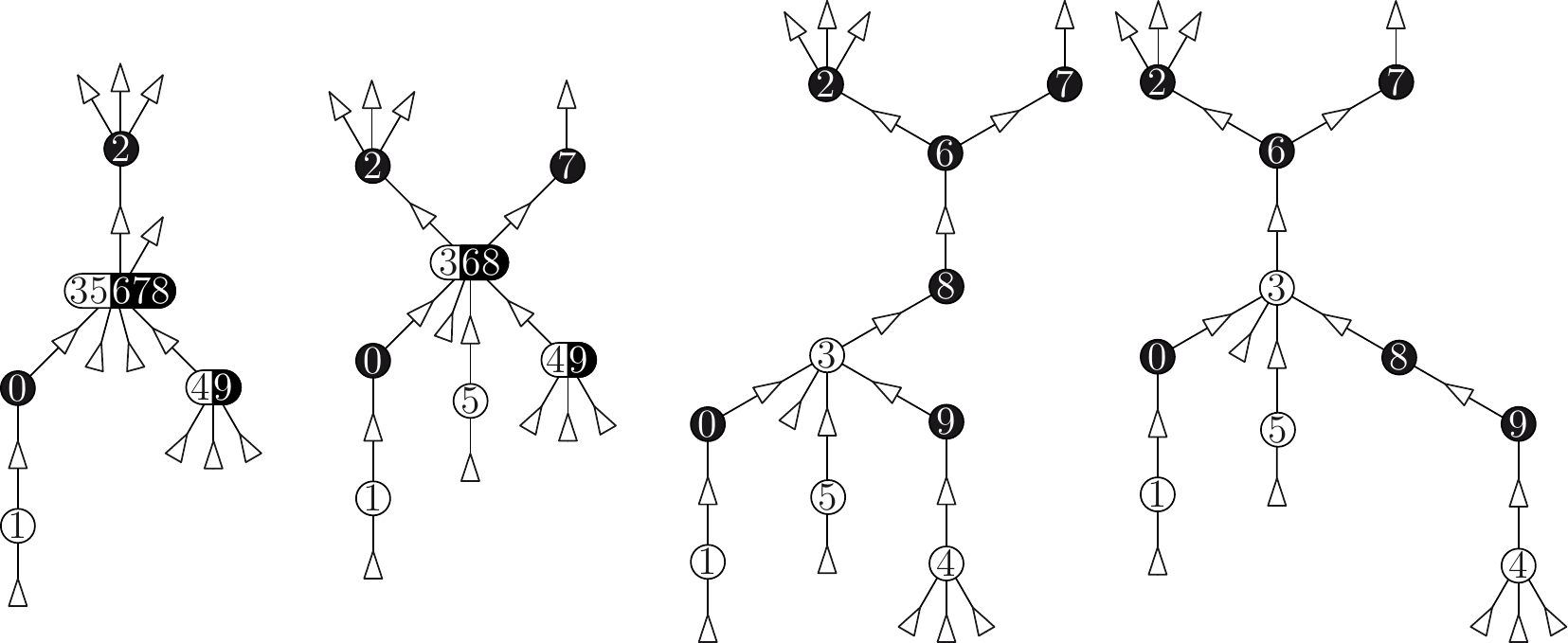}}
  \caption{Four different blossoming spines on the signed ground tree~$\tree\ex$ of \fref{fig:exmTree}.}
  \label{fig:exmBlossomingSpines}
\end{figure}

The incoming blossoms of a blossoming spine~$\spine$ correspond to the negative irrelevant open subtrees of~$\subtrees(\tree)$. Indeed, at a node~$s$ of~$\spine$ labeled by~$U = U^- \sqcup U^+$, we have (by definition) one incoming blossom for each connected component of~$\tree \ssm U^-$ which does not contain the source label set~$\source(r)$ of any incoming arc~$r$ of~$\spine$ at~$s$. Such a connected component contains a unique negative irrelevant open subtree~$Z$ of~$T$ whose boundary meets~$U^-$ but none of the source label sets~$\source(r)$ of the incoming arcs of~$\spine$ at~$s$ (otherwise, we would get a contradiction with Lemma~\ref{lem:arcToBuildingBlock}). Therefore, although we cannot distinguish the different incoming blossoms at~$s$, they precisely correspond to the negative irrelevant open subtrees~$Z$ of~$T$ whose boundary meets~$U^-$ but none of the source label sets~$\source(r)$ of the incoming arcs of~$\spine$ at~$s$. By a slight abuse, we distinguish each blossom and we denote by~$\arcToSubtree(b)$ the negative irrelevant open subtree of~$\tree$ corresponding to an incoming blossom~$b$ of~$\spine$, extending the notation we had for the arcs of~$\spine$. We prove similarly that the outgoing blossoms of~$\spine$ correspond to the positive irrelevant open subtrees of~$\subtrees(\tree)$, and we also denote by~$\arcToSubtree(b)$ the positive irrelevant open subtree of~$\tree$ corresponding to an outgoing blossom~$b$ of~$\spine$.

Alternatively, we can also interpret blossoming trees as the directed and labeled dual trees of the cell decompositions of $\treeInterval \eqdef \tree \times [-1,1]$ described earlier. The difference with our previous description of spines as dual trees of cell decompositions is that we now add blossoms for all boundary components of the cells, except the vertical ones.

\begin{example}[Unsigned path, continued]
Consider a path~$\pathG$ labeled by~$[\nu]$ from one leaf to the other and with only negative signs. The maximal blossoming spines on~$\pathG$ have indegree~$2$ and outdegree~$1$ at each node, except at nodes~$1$ and~$\nu$. If we add one more incoming blossom to the nodes~$1$ and~$\nu$, then the maximal blossoming spines on~$\pathG$ are precisely the \defn{perfect binary trees} (with precisely $2$ children per internal node), while the blossoming spines on~$\pathG$ are the \defn{Schr\"oder trees} (where all internal nodes have at least two children). Note that the labels in all these blossoming spines (maximal or not) can now be reconstructed from the combinatorics of the unlabeled trees by infix labeling.
\end{example}

\begin{example}[Signed path, continued]
When~$\pathG$ is a signed path, the maximal blossoming spines on~$\pathG$ have indegree~$2$ and outdegree~$1$ at each negative node, and indegree~$1$ and outdegree~$2$ at each positive node, except at the nodes~$1$ and~$\nu$. If we add one more blossom to the nodes~$1$ and~$\nu$ of the blossoming spines on~$\pathG$, we precisely obtain the spines defined by C.~Lange and V.~Pilaud in~\cite{LangePilaud-spines} or equivalently the mixed cobinary trees defined by K.~Igusa and J.~Ostroff in~\cite{IgusaOstroff}. These spines originally motivated the present paper. Note that their labels can also be recovered from the combinatorics of the unlabeled spines by an adaptation of infix labeling, see~\cite{LangePilaud-spines}.
\end{example}

The additional blossoms attached to the blossoming spine enable us to obtain the following structural result. A \defn{proper cut}~$\Gamma$ in a blossoming spine~$\spine$ is a directed cut of~$\spine$ which separates all tails of the incoming blossoms from all heads of the outgoing blossoms (it may cut some blossoms). We denote by~$\source(\Gamma)$ the union of all labels in the source set of~$\Gamma$ and by~$\sink(\Gamma)$ the union of all labels in the sink set of~$\Gamma$.

\begin{proposition}
\label{prop:cutSpine}
For any proper cut~$\Gamma$ of a blossoming spine~$\spine$ on~$\tree$, the open subtrees~$\arcToSubtree(r)$ given by the arcs or blossoms~$r$ of~$\spine$ cut by~$\Gamma$ are the connected components of~${\tree \ssm \big( \source(\Gamma)^+ \cup \sink(\Gamma)^- \big)}$.
\end{proposition}

\begin{proof}
Consider a sequence of proper cuts~$\Gamma_0, \dots, \Gamma_p$ which sweeps the spine~$\spine$ from its incoming blossoms to its outgoing blossoms, and passes at each step a single vertex~$v_i$ of~$\spine$ labeled by~${U_i = U_i^- \sqcup U_i^+}$. The cut~$\Gamma_0$ cuts all incoming blossoms of~$\spine$, which indeed correspond to the connected components of~$\tree \ssm \ground^-$, and~$\source(\Gamma_0) = \varnothing$ while $\sink(\Gamma_0) = \ground$. Assume now that the open subtrees corresponding to the arcs of blossoms of~$\spine$ cut by~$\Gamma_i$ are precisely the connected components of~$\tree \ssm \big( \source(\Gamma_i)^+ \cup \sink(\Gamma_i)^- \big)$. By definition of blossoming spines, there is one incoming arc of~$\spine$ at vertex~$v_i$ for each connected component of~$\tree \ssm U_i^-$ and one outgoing arc for each connected component of~$\tree \ssm U_i^+$. Thus, when we sweep vertex~$v_i$, we merge the connected components of~$\tree \ssm \big( \source(\Gamma_i)^+ \cup \sink(\Gamma_i)^- \big)$ incident to a vertex of~$U_i^-$ and split the connected components of~$\tree \ssm \big( \source(\Gamma_i)^+ \cup \sink(\Gamma_i)^- \big)$ containing a vertex of~$U_i^+$. Therefore, the open subtrees corresponding to the arcs of blossoms of~$\spine$ cut by~$\Gamma_{i+1}$ are precisely the connected components of~$\tree \ssm \big( \source(\Gamma_{i+1})^+ \cup \sink(\Gamma_{i+1})^- \big)$. The statement follows since any cut can be reached in a sequence of cuts sweeping~$\spine$ from bottom to top.
\end{proof}


\section{Spine fan}
\label{sec:spineFan}

In this section, we construct a geometric realization of the signed nested complex~$\nestedComplex(\tree)$ as a complete simplicial fan. We call it the \defn{spine fan} since it is obtained from the signed spines on~$\tree$. It coarsens the braid fan, defined by the braid arrangement in~$\R^\ground$. We start by a brief reminder on the braid fan and its relation to preposet cones. More details can be found in~\cite{PostnikovReinerWilliams}. We also refer the reader to~\cite{LangePilaud-spines}, where the motivating situation of paths is treated in details.

\subsection{Braid fan and preposet cones}
\label{subsec:preposet}

We consider the braid arrangement on~$\R^\ground$, defined as the collection of hyperplanes~$\set{\b{x} \in \R^\ground}{x_u = x_v}$ for~$u \ne v \in \ground$. Since this arrangement is not essential (the intersection of all these hyperplanes contains the line directed by~$\one \eqdef \sum_{v \in \ground} e_v$), we consider its intersection with the hyperplane~$\HH$ of~$\R^\ground$ defined by
\[
\HH \eqdef \biggset{\b{x} \in \R^\ground}{\sum_{v \in \ground} x_v = \binom{\nu+1}{2}},
\]
where~$\nu = |\ground|$. On the hyperplane~$\HH$, the braid arrangement defines a pointed complete simplicial fan that we call the \defn{braid fan} and denote by~$\braidFan(\ground)$. It is the normal fan of the permutahedron~$\Perm(\ground)$, see Section~\ref{subsec:vertexFacetDescriptions}.

The $k$-dimensional cones of~$\braidFan(\ground)$ correspond to the \defn{surjections} from~$\ground$ to~$[k+1]$, or equivalently to the \defn{ordered partitions} of~$\ground$ into $k+1$ parts. We pass from ordered partitions to surjections by inversion: the fibers of a surjection from~$\ground$ to~$[k+1]$ define an ordered partition of~$\ground$ with $k+1$ parts, and reciprocally the positions of the elements of~$\ground$ in an ordered partition of~$\ground$ with $k+1$ parts define a surjection from~$\ground$ to~$[k+1]$. For example, the maximal cones of~$\braidFan(\ground)$ correspond to the~$\nu!$ linear orders on~$\ground$, while the rays of~$\braidFan(\ground)$ correspond to the~$2^\nu-2$ proper and non-empty subsets of~$\ground$. See \fref{fig:braidFan}.

\begin{figure}[h]
  \capstart
  \centerline{\includegraphics[scale=.75]{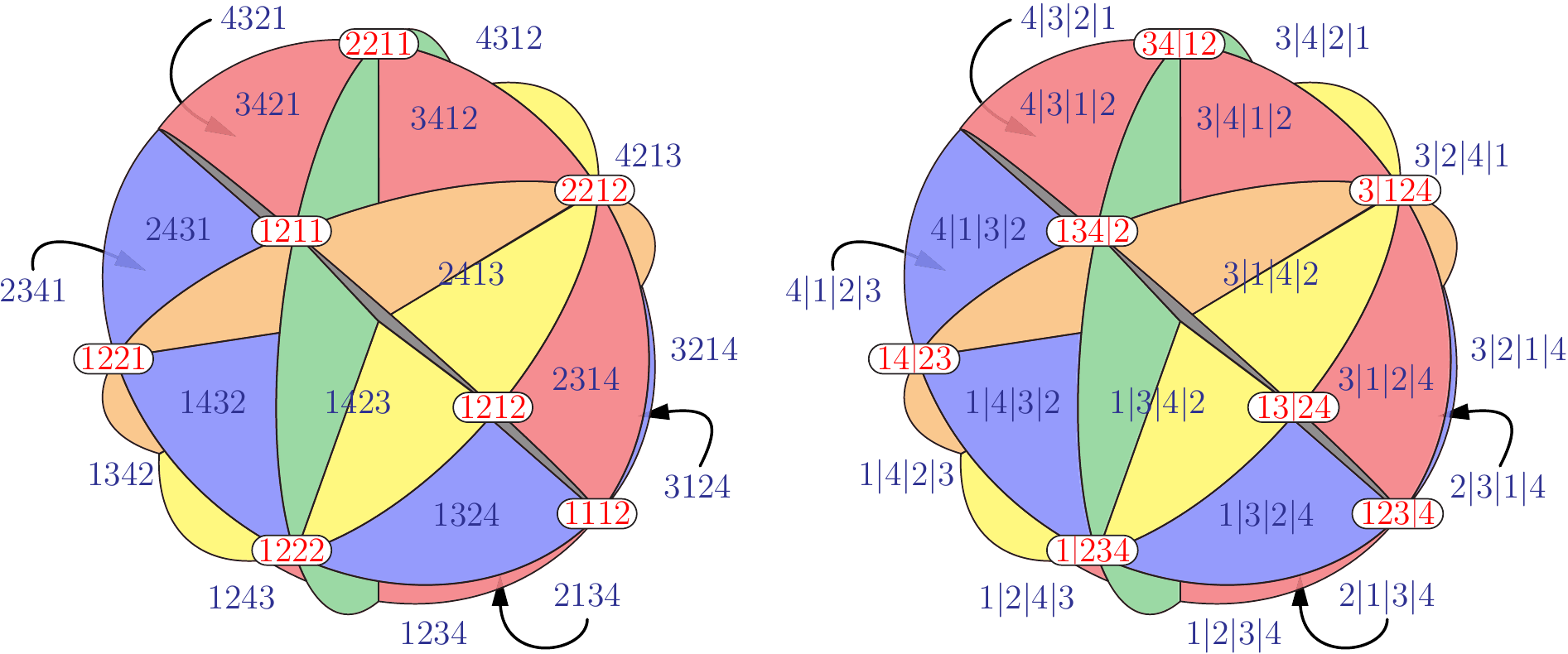}}
  \caption{The $3$-dimensional braid fan~$\braidFan([4])$. Its $k$-dimensional cones correspond equivalently to the surjections from~$[4]$ to~$[k+1]$ (left), or to the ordered partitions of~$[4]$ into~$k+1$ parts (right). Rays are in red while maximal cones are in blue. The reader is invited to label the $2$-dimensional cones accordingly.}
  \label{fig:braidFan}
\end{figure}

In fact, the braid fan~$\braidFan(\ground)$ is useful to provide geometric representations of various order structures on~$\ground$ which are not necessarily linear orders. A \defn{preposet} on the ground set~$\ground$ is a binary relation~$R \subseteq \ground \times \ground$ which is reflexive and transitive. Hence, any equivalence relation is a symmetric preposet and any poset is an antisymmetric preposet. Any preposet~$R$ can in fact be decomposed into an equivalence relation~${\equiv_R \eqdef \set{(u,v) \in R}{(v,u) \in R}}$, together with a poset structure~${\prec_R \eqdef \, R/\!\equiv_R}$ on the equivalence classes of~$\equiv_R$. Consequently, there is a one-to-one correspondence between preposets and acyclic oriented graphs on subsets of~$\ground$ whose vertices partition~$\ground$: a preposet~$R$ corresponds to the Hasse diagram of the poset~$\prec_R$ on the equivalence classes of~$\equiv_R$, and conversely, an acyclic oriented graph whose vertex set partitions~$\ground$ corresponds to its transitive closure.

We define the \defn{braid cone} of a preposet~$R$ on~$\ground$ as the polyhedral cone
\[
\normalCone(R) \eqdef \set{\b{x} \in \HH}{x_u \le x_v, \text{ for all } (u,v) \in R}.
\]
For example, the cones of the braid fan~$\braidFan(\ground)$ are precisely the cones of the \defn{linear preposets}, \ie the preposets~$L$ on~$\ground$ whose associated poset~$\prec_L$ is a linear order on the equivalence classes of~$\equiv_L$. The dimension of~$\normalCone(R)$ is the number of equivalence classes of the relation~$\equiv_R$ minus~$1$. The normal vectors of the facets of~$\normalCone(R)$ are the incidence vectors of the arcs of the Hasse diagram of~$\prec_R$. In other words, the polar cone of~$\normalCone(R)$ is the \defn{incidence cone} of~$R$ defined as the polyhedral~cone
\[
\primalCone(R) \eqdef \cone\set{e_u - e_v}{(u,v) \in R}.
\]
In the definition of~$\normalCone(R)$ and~$\primalCone(R)$, we can obviously restrict to the cover relations of~$R$. We will use the following dictionary between combinatorial properties of preposets and geometric properties of their braid cones (see~\cite{PostnikovReinerWilliams} for details):
\begin{enumerate}
\item If~$R$ and~$R'$ are two preposets on~$\ground$, then the cone~$\normalCone(R)$ contains the cone~$\normalCone(R')$ iff $R'$ is an extension of~$R$, \ie $R \subseteq R'$ as a subset of~$\ground \times \ground$.
\item The cone~$\normalCone(R)$ of any preposet~$R$ is the (disjoint) union of the (relative interiors of the) cones of its linear extensions. In particular, the cone~$\normalCone(\prec)$ of a poset~$\prec$ is the union of the total linear extensions of~$\prec$, \ie the linear orders on~$\ground$ which respect the relations in~$\prec$.
\item The cone~$\normalCone(R)$ is simplicial iff the Hasse diagram of~$\prec_R$ is a directed tree.
\item The rays of the braid cone~$\normalCone(R)$ are the characteristic vectors of the source sets of the minimal directed cuts in the Hasse diagram of~$\prec_R$. In particular, if~$\normalCone(R)$ is simplicial, then its rays are characteristic vectors of the source sets of the arcs of the directed tree given by the Hasse diagram of~$\prec_R$.
\end{enumerate}


\subsection{Spine fan}

Consider a signed spine~$\spine$ on~$\tree$. Since~$\spine$ is a directed tree labeled by a partition of~$\ground$, its transitive closure is a preposet on~$\ground$. We denote by~$\normalCone(\spine)$ the cone of this preposet,~\ie
\[
\normalCone(\spine) \eqdef \set{\b{x} \in \HH}{x_u \le x_v, \; \text{for all paths } u \to v \text{ in } \spine} = \cone\set{\one_B}{B \in \spineToNested(\spine)}.
\]
This cone is simplicial since~$\spine$ is a tree. The cones of the signed spines on~$\tree$ obtained by contraction of the spine~$\spine$ are faces of the cone~$\normalCone(\spine)$. Moreover, the cone~$\normalCone(\spine)$ is obtained by glueing some cones~$\normalCone(\prec)$ of linear preposets~$\prec$ on~$\ground$. As observed above, $\normalCone(\prec)$ is contained in~$\normalCone(\spine)$ iff $\prec$ defines a linear extension of~$\spine$, meaning that $u \prec v$ for any~$u,v \in \ground$ such that~$u$ appears below~$v$ in~$\spine$. The following statement is the main result of this section and prepares the foundations for the polytopal realizations of the signed nested complex to be constructed in Section~\ref{sec:signedTreeAssociahedron}.

\begin{theorem}
\label{theo:spineFan}
The collection of cones~$\spineFan(\tree) \eqdef \set{\normalCone(\spine)}{\spine \in \spinePoset(\tree)}$ defines a complete simplicial fan on~$\HH$, which we call the \defn{spine fan}.
\end{theorem}

\begin{corollary}
For any signed tree~$\tree$, the signed nested complex~$\nestedComplex(\tree)$ is a simplicial sphere.
\end{corollary}

\begin{proof}[Proof of Theorem~\ref{theo:spineFan}]
We show that any linear order~$\prec$ on~$\ground$ is a linear extension of a unique maximal signed spine on~$\tree$. In the next section, we will refer to this maximal signed spine as~$\surjectionPermAsso(\prec)$.

To prove the uniqueness, fix a linear order~$\prec$ on~$\ground$ and consider the bijection~${\sigma : [\nu] \to \ground}$ such that~${\sigma(1) \prec \sigma(2) \prec \dots \prec \sigma(\nu)}$. If~$\prec$ is a linear extension of a maximal signed spine~$\spine$ of~$\tree$, then for any~$0 \le i \le \nu$, there is a proper cut~$\Gamma_i$ of the blossoming spine~$\spine$ whose source set~$\source(\Gamma_i) = \sigma([i])$ contains the first~$i$ elements of~$\ground$ for~$\prec$ and whose sink set~$\sink(\Gamma_i) = \sigma([\nu] \ssm [i])$ contains the last~$\nu-i$ elements of~$\ground$ for~$\prec$. Moreover, the sequence of cuts~$\Gamma_0, \Gamma_1, \dots, \Gamma_{\nu}$ sweeps the blossoming spine~$\spine$ from its incoming blossoms to its outgoing blossoms, with a single node~$\sigma(i)$ of~$\spine$ between two consecutive cuts~$\Gamma_{i-1}$ and~$\Gamma_i$. Therefore, any arc of~$\spine$ is cut by at least one of the cuts~$\Gamma_i$. Using Proposition~\ref{prop:cutSpine}, we can reconstruct the open subtrees~$\arcToSubtree(r)$ corresponding to the arcs~$r$ of~$\spine$ cut by~$\Gamma_i$, knowing only the source set~$\source(\Gamma_i) = \sigma([i])$ and the sink set~$\sink(\Gamma_i) = \sigma([\nu] \ssm [i])$ of~$\Gamma_i$. We can therefore reconstruct the signed spine~$\spine$ from~$\sigma$, or equivalently from~$\prec$.

To prove the existence, we argue geometrically. Namely, consider two adjacent maximal signed spines~$\spine$ and~$\spine'$ in the graph of flips~$\spineFlipGraph(\tree)$, let~$r \eqdef u \to v$ denote the arc of~$\spine$ which is reversed to get~$\spine'$, and let~$\bar\spine$ be the signed spine obtained from~$\spine$ by contracting~$r$. Since the arc~$r$ is reversed, the cones~$\normalCone(\spine)$ and~$\normalCone(\spine')$ lie on opposite sides of the hyperplane~$\set{\b{x} \in \HH}{x_u = x_v}$, and they share a common facet~$\normalCone(\bar\spine)$. Since this is true for all pairs of adjacent maximal spines in the flip graph, we obtain a simplicial fan with no boundary, therefore a complete simplicial fan.
\end{proof}

\begin{example}[Unsigned tree, continued]
For unsigned trees, the spine fan is precisely the \defn{nested fan} which is the normal fan of all known polytopal realizations of the graph associahedron~\cite{CarrDevadoss, Devadoss, Postnikov, FeichtnerSturmfels, Zelevinsky}.
\end{example}

\begin{example}[Signed path, continued]
For signed paths, the spine fans are the type~$A$ Cambrian fans of N.~Reading and D.~Speyer~\cite{ReadingSpeyer}. Following N.~Reading~\cite{Reading-CambrianLattices}, we study below the combinatorial properties of the surjection map defined by these simplicial fans.
\end{example}

\begin{example}[Tripod, continued]
\fref{fig:tripodNormalFan} illustrates the spine fans of the signed ground trees~$\tripodWhite$ and~$\tripodBlack$, which realize the signed nested complexes of \fref{fig:tripodNestedComplex}.

\begin{figure}[h]
  \capstart
  \centerline{\includegraphics[scale=.75]{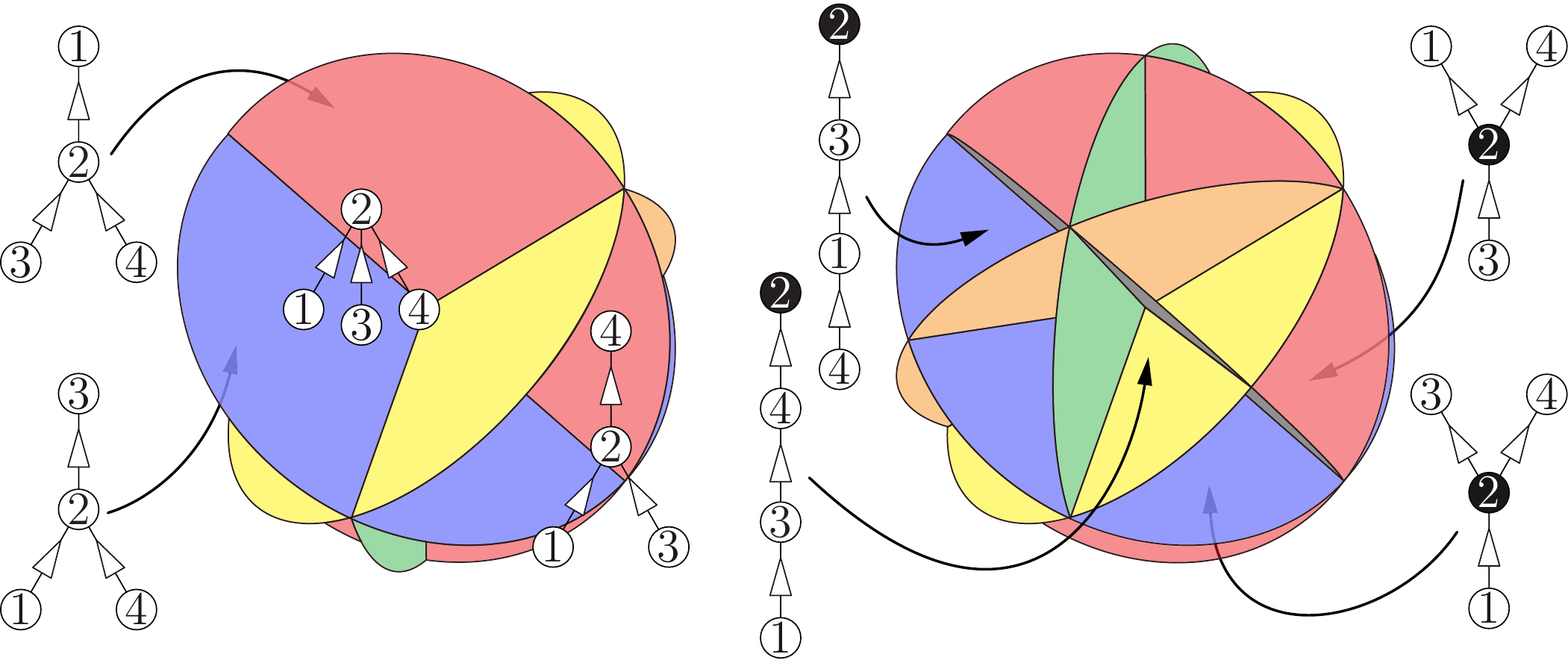}}
  \caption[Spine fans for tripods]{Spine fans of the signed ground trees~$\tripodWhite$ and~$\tripodBlack$. Some maximal cones are labeled by their corresponding spines.}
  \label{fig:tripodNormalFan}
\end{figure}
\end{example}


\subsection{Surjection map}
\label{subsec:surjection}

In the proof of Theorem~\ref{theo:spineFan} we obtained the following statement.

\begin{proposition}
Any linear order~$\prec$ on~$\ground$ extends a unique maximal signed spine~$\surjectionPermAsso(\prec)$~of~$\tree$.
\end{proposition}

This statement defines a surjection~$\surjectionPermAsso$ from the linear orders on~$\ground$ to the maximal signed spines on~$\tree$: the image~$\surjectionPermAsso(\prec)$ of a linear order~$\prec$ is the unique maximal signed spine on~$\tree$ for which~$\prec$ is a linear extension. Since this surjection map~$\surjectionPermAsso$ plays an important role in the rest of this paper, we take here the opportunity to study some of its properties.

First and foremost, the proof of Theorem~\ref{theo:spineFan} implicitly contains a procedure describing~$\surjectionPermAsso$. Consider a linear order~$\prec$ on~$\ground$, and let~$\sigma : [\nu] \to \ground$ be such that~${\sigma(1) \prec \sigma(2) \prec \dots \prec \sigma(\nu)}$. We can construct the blossoming spine~$\surjectionPermAsso(\prec)$ by sweeping it from its incoming blossoms up to its outgoing blossoms in the order given by~$\prec$. We start placing one incoming blossom in each connected component of~$\tree \ssm \ground^-$ and we consider the cut~$\Gamma_0$ passing through all these blossoms. We then construct~$\surjectionPermAsso(\prec)$ and its successive cuts~$\Gamma_1, \dots, \Gamma_\nu$ step by step as follows:
\begin{itemize}
\item the open subtrees~$\arcToSubtree(r)$ corresponding to the arcs~$r$ of~$\surjectionPermAsso(\prec)$ cut by~$\Gamma_i$ are precisely the connected components of~$\tree \ssm \big( \sigma([\nu] \ssm [i])^- \cup \sigma([i])^+ \big)$;
\item to sweep a vertex~$\sigma(i) \in \ground^-$, we merge the arcs~$r$ of~$\surjectionPermAsso(\prec)$ cut by~$\Gamma_{i-1}$ such that~${\sigma(i) \in \boundary(\arcToSubtree(r))}$ into a single arc;
\item to sweep a vertex~$\sigma(i) \in \ground^+$, we split the arc~$r$ of~$\surjectionPermAsso(\prec)$ cut by~$\Gamma_{i-1}$ such that~$\sigma(i) \in \arcToSubtree(r)$ to create as many arcs as the number of connected components of~$\tree \ssm \sigma(i)$.
\end{itemize}
This procedure ends at~$\Gamma_\nu$, cutting one outgoing blossom for each connected component~of~${\tree \ssm \ground^+}$.

There is another equivalent way to describe this procedure in the space~$\treeInterval \eqdef \tree \times [-1,1]$, closer to the original definition of N.~Reading for signed paths~\cite{Reading-CambrianLattices}. Namely, the curves~$\subtreeToCurve(\arcToSubtree(r))$ in~$\treeInterval$ corresponding to the arcs of~$\surjectionPermAsso(\prec)$ cut by~$\Gamma_i$ form the upper hull of the point set~$\sigma([\nu] \ssm [i])^- \cup \sigma([i])^+$ in~$\treeInterval$. Therefore, the spine~$\surjectionPermAsso(\prec)$ is obtained as the dual tree of the union of the upper hulls of the point sets~$\sigma([\nu] \ssm [i])^- \cup \sigma([i])^+$ for~$0 \le i \le \nu$.

This procedure could equivalently be performed sweeping the blossoming spine~${\surjectionPermAsso(\prec)}$ from its outgoing blossoms down to its incoming blossoms. We have chosen the bottom-up version to stick to N.~Reading's presentation in~\cite{Reading-CambrianLattices}.

\begin{example}[Unsigned path, continued]
When the ground tree is a path~$\pathG$ labeled by~$[\nu]$ increasingly from one leaf to the other and with only negative signs, we have seen in Example~\ref{exm:BST} that the maximal spines are precisely the binary search trees with label set~$[\nu]$. Consider a linear order~$\prec$ on~$\ground$, and let~$\sigma : [\nu] \to \ground$ be such that~${\sigma(1) \prec \sigma(2) \prec \dots \prec \sigma(\nu)}$. The spine~$\surjectionPermAsso(\prec)$ is the last tree of the sequence of binary search trees~${\varnothing \defeq \spine_0 \subset \spine_1 \subset \dots \subset \spine_\nu \eqdef \surjectionPermAsso(\prec)}$, where~$\spine_i$ is obtained by insertion of~$\sigma(\nu + 1 - i)$ in ~$\spine_{i-1}$.
\end{example}

Our next step is to understand the fibers of~$\surjectionPermAsso$. In other words, we want to characterize the subsets of linear orders on~$\ground$ which extend the same maximal signed spine on~$\tree$. Observe already that since they correspond to subsets of fundamental chambers inside a cone, these fibers form connected subgraphs of the facet-ridge graph of the braid cone (\ie of the $1$-skeleton of the permutahedron, see Section~\ref{sec:signedTreeAssociahedron}). We now provide a characterization of the adjacent linear orders on~$\ground$ in the same fiber of~$\surjectionPermAsso$.
We say that two linear orders~$\prec$ and~$\prec'$ on~$\ground$ are \defn{adjacent} if they only differ by the order of two consecutive elements~$u$ and~$v$, \ie if we can write
\begin{align*}
& \prec\; \eqdef w_1 \prec\; w_2 \prec\; \dots \prec\; w_{i-1} \prec\; u \prec\; v \prec\; w_{i+2} \prec\; \dots \prec\; w_\nu, \\
\text{and} \quad & \prec' \eqdef w_1 \prec' w_2 \prec' \dots \prec' w_{i-1} \prec' v \prec' u \prec' w_{i+2} \prec' \dots \prec' w_\nu.
\end{align*}
In other words, the cones~$\normalCone(\prec)$ and~$\normalCone(\prec')$ are adjacent in the braid fan~$\braidFan(\ground)$ and separated by the hyperplane~$\bigset{\b{x} \in \HH}{x_u = x_v}$. In this situation, we say that $\prec$ and~$\prec'$ are \defn{$\tree$-congruent} if furthermore there is a vertex~$w$ in between~$u$ and~$v$ in the ground tree~$\tree$ such that
\[
w \in \ground^- \; \text{ and } \; u \prec v \prec w \qquad\text{or}\qquad w \in \ground^+ \; \text{ and } \; w \prec u \prec v.
\]

\begin{lemma}
\label{lem:surjection}
Let~$\prec$ and~$\prec'$ be two adjacent linear orders on~$\ground$. Then $\surjectionPermAsso(\prec) = \surjectionPermAsso(\prec')$ iff $\prec$ and~$\prec'$ are $\tree$-congruent.
\end{lemma}

\begin{proof}
By definition, $\surjectionPermAsso(\prec) = \surjectionPermAsso(\prec') = \spine$ iff both~$\prec$ and~$\prec'$ are linear extensions of~$\spine$. Since~$\prec$ and~$\prec'$ only differ by the order of~$u$ and~$v$, it is equivalent to require that~$u$ and~$v$ are not comparable in~$\spine$. Equivalently, the unique path in~$\spine$ between~$u$ and~$v$ either contains a negative vertex~$w \in \ground^-$ such that~$u$ and~$v$ lie in distinct incoming subspines of~$\spine$ at~$w$, or a positive vertex~$w \in \ground^+$ such that~$u$ and~$v$ lie in distinct outgoing subspines of~$\spine$ at~$w$. In the former case, $u \prec v \prec w$ since~$w$ lies below~$u$ and~$v$ in~$\spine$, while in the latter case, $w \prec u \prec v$ since~$w$ lies above~$u$ and~$v$ in~$\spine$. Moreover, in both cases, the local condition of spines around~$w$ ensures that~$w$ lies in between~$u$ and~$v$ in the ground tree~$\tree$.
\end{proof}

\begin{example}[Signed path, continued]
When the signed ground tree is a signed path~$\pathG$, the $\pathG$-congruence is a type~$A$ Cambrian congruence defined by N.~Reading in~\cite{Reading-CambrianLattices}. In particular, if~$\pathG$ has only negative signs, the $\pathG$-congruence is the sylvester congruence~\cite{HivertNovelliThibon}.
\end{example}

\begin{remark}
\label{rem:extensionSurjection}
Since~$\spineFan(\tree)$ is a complete simplicial fan, there is a unique spine~$\spine$ on~$\tree$ such that the relative interior of~$\normalCone(\spine)$ contains the relative interior of~$\normalCone(\prec)$, for any linear preposet~$\prec$ on~$\ground$. It extends~$\surjectionPermAsso$ to a surjection~$\tilde\surjectionPermAsso$ from all linear preposets to all signed spines on~$\tree$, defined by~$\tilde\surjectionPermAsso(\prec) = \spine$. The spine~$\tilde\surjectionPermAsso(\prec)$ can also be constructed from the linear preposet~$\prec$ as before. Namely, if~$V_1 \sqcup V_2 \sqcup \dots \sqcup V_k = \ground$ denotes the partition of~$\ground$ corresponding to~$\prec$ (\ie such that~$u \prec v$ iff $u \in V_i$ and $v \in V_j$ for some~$i < j$), then the spine~$\tilde\surjectionPermAsso(\prec)$ is the directed and labeled dual tree of the union of the upper hulls of the point sets $\bigcup_{i < j} V_j^- \, \cup \; \bigcup_{j \le i} V_j^+$ for~$0 \le i \le k$. This spine can equivalently be described by a sweeping procedure similar to the one presented above for~$\surjectionPermAsso$. Details are left to the reader.
\end{remark}

\begin{example}[Unsigned path, continued]
When the ground tree is a path~$\pathG$ labeled by~$[\nu]$ increasingly from one leaf to the other and with only negative signs, the spines are precisely all Schr\"oder trees whose label sets partition~$[\nu]$. Consider a linear preposet~$\prec$ on~$\ground$, and let ${V_1 \sqcup V_2 \sqcup \dots \sqcup V_k = \ground}$ denotes the partition of~$\ground$ corresponding to~$\prec$. Then the spine~$\tilde\surjectionPermAsso(\prec)$ is the last tree of the sequence~$\varnothing \defeq \spine_0 \subset \spine_1 \subset \dots \subset \spine_k \eqdef \tilde\surjectionPermAsso(\prec)$, where~$\spine_i$ is obtained by insertion of~$V_{k+1-i}$ in~$\spine_{i-1}$. By insertion of a set~$X$ in a labeled Schr\"oder tree~$\spine$, we mean the following adaptation of the classical insertion of an element in a binary search tree:
\begin{itemize}
\item if~$\spine = \varnothing$, then we insert~$X$ at the root of~$\spine$;
\item otherwise, the root~$\rho$ of~$\spine$ is labeled by a set~$\{u_1 < \dots < u_k\} \subseteq [\nu]$. For each~$0 \le i \le k$, we let~$X_i$ denote the subset of~$X$ strictly inbetween~$u_i$ and~$u_{i+1}$ (where by convention~$u_0 = 0$ and~$u_{k+1} = \nu + 1$), and we insert recursively~$X_i$ in the $i$\ordinal{} child of~$\rho$.
\end{itemize}
\end{example}


\section{Signed tree associahedron}
\label{sec:signedTreeAssociahedron}

In this section, we construct a polytope whose normal fan is the spine fan~$\spineFan(\tree)$ constructed in the previous section. It therefore provides a polytopal realization of the signed nested complex on a tree~$\tree$ that we call \defn{signed tree associahedron} and denote by~$\Asso(\tree)$. It generalizes on the one hand the graph associahedra for trees~\cite{CarrDevadoss, Devadoss, Postnikov} and on the other hand the various realizations of the associahedron by C.~Hohlweg and C.~Lange in~\cite{HohlwegLange}, using ideas from~\cite{LangePilaud-spines}. In particular, our signed tree associahedron is obtained from the classical permutahedron by removing certain well-chosen facets. We therefore study some related geometric properties of these polytopes such as their pairs of parallel facets, their common vertices with the classical permutahedron, and their isometry classes.

\subsection{Vertex and facet descriptions}
\label{subsec:vertexFacetDescriptions}

Before describing our polytopal realization of the signed nested complex, we briefly recall the vertex and facet descriptions of the classical permutahedron. The \defn{permutahedron}~$\Perm(\ground)$ is the polytope obtained as
\begin{enumerate}[(i)]
\item either the convex hull of the vectors~$\b{p}(\sigma) \eqdef \sum_{i \in [\nu]} i \, e_{\sigma(i)} \in \R^\ground$, for all bijections~$\sigma : [\nu] \to \ground$,
\item or the intersection of the hyperplane~$\HH \eqdef \Hyp(\ground)$ with the half-spaces~$\HS(U)$ for~$\varnothing \ne U \subseteq \ground$, where
\[
\Hyp(U) \eqdef \biggset{\b{x} \in \R^\ground}{\sum_{u \in U} x_u = \binom{|U|+1}{2}} \quad \text{and} \quad \HS(U) \eqdef \biggset{\b{x} \in \R^\ground}{\sum_{u \in U} x_u \ge \binom{|U|+1}{2}}.
\]
\end{enumerate}
Its normal fan is precisely the braid fan~$\braidFan(\ground)$. In particular, its $k$-dimensional faces correspond equivalently to the surjections from~$\ground$ to~$[\nu-k]$, to the ordered partitions of~$\ground$ into~$\nu-k$ parts, or to the linear preposets on~$\ground$ of rank~$\nu-k$. See \fref{fig:permutahedron} for an illustration in dimension~$3$.

\begin{figure}[h]
  \capstart
  \centerline{\includegraphics[scale=.75]{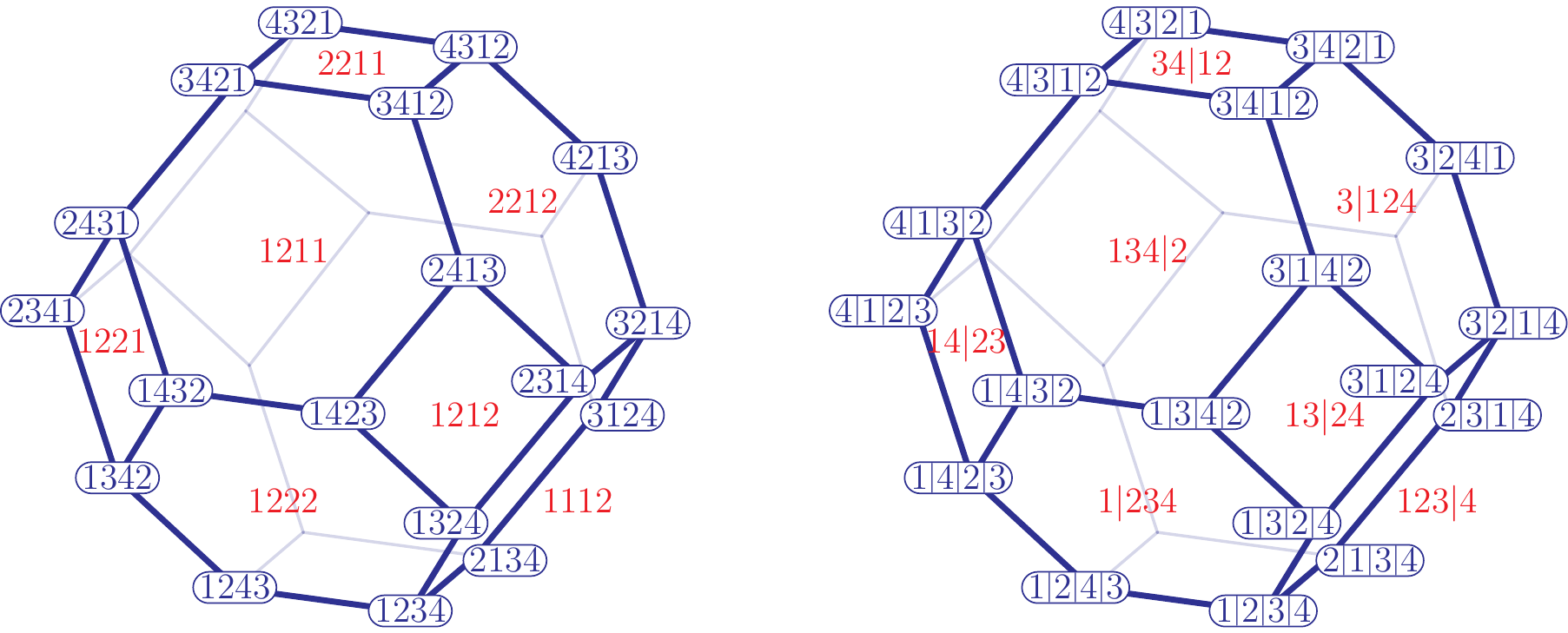}}
  \caption{The $3$-dimensional permutahedron~$\Perm([4])$. Its $k$-dimensional faces correspond equivalently to the surjections from~$[4]$ to~$[4-k]$ (left), or to the ordered partitions of~$[4]$ into~$4-k$ parts (right). Vertices are in blue and facets in red. The reader is invited to label the edges accordingly}
  \label{fig:permutahedron}
\end{figure}

From this polytope, we construct the \defn{signed tree associahedron}~$\Asso(\tree)$, for which we give both vertex and facet descriptions:
\begin{enumerate}[(i)]

\item The facets of~$\Asso(\tree)$ correspond to the signed building blocks of~$\building(\tree)$. Namely, we associate to a signed building block~$B \in \building(\tree)$ the hyperplane~$\Hyp(B)$ and the half-space~$\HS(B)$ defined above for the permutahedron.

\item The vertices of~$\Asso(\tree)$ correspond to the maximal signed nested sets of~$\nestedComplex(\tree)$ or equivalently to the maximal signed spines on~$\tree$. Consider a maximal signed spine~$\spine$ on~$\tree$. Let~$\Pi(\spine)$ denote the set of all (undirected and simple) paths in~$\spine$, including the trivial paths reduced to a single node. At a node labeled by~$v \in \ground$, the signed spine~$\spine$ has a unique outgoing arc~$r_v$ if~$v \in \ground^-$ and a unique incoming arc~$r_v$ if~$v \in \ground^+$. We let~$\vertex{\spine}$ be the point in~$\R^\ground$ whose coordinates are defined by
\[
\vertex{\spine}_v =
\begin{cases}
	\big| \set{\pi \in \Pi(\spine)}{v \in \pi \text{ and } r_v \notin \pi} \big| & \text{if } v \in \ground^-, \\
	\nu + 1 - \big| \set{\pi \in \Pi(\spine)}{v \in \pi \text{ and } r_v \notin \pi} \big| & \text{if } v \in \ground^+.
\end{cases}
\]
\end{enumerate}
To illustrate these definitions, let us compute some explicit facet defining inequalities and vertices of the signed tree associahedron~$\Asso(\tree\ex)$ of the signed ground tree~$\tree\ex$ of \fref{fig:exmTree}. The half-space corresponding to the signed building block~$B = \{0,1,2,3,5,7\}$ of~$\tree\ex$ illustrated in \fref{fig:exmOSSTSBB}~is
\[
\HS(B) = \set{\b{x} \in \HH}{x_0 + x_1 + x_2 + x_3 + x_5 + x_7 \le 21}.
\]
The vertices associated to the maximal signed spines~$\spine$ and~$\spine'$ of \fref{fig:exmMaximalSpines} are
\[
\vertex{\spine} = (2,10,1,14,1,1,7,10,7,2)
\qquad\text{and}\qquad
\vertex{\spine'} = (2,10,1,18,1,1,7,10,3,2).
\]
Observe that both~$\vertex{\spine}$ and~$\vertex{\spine'}$ belong to the hyperplanes~$\HH \eqdef \set{\b{x} \in \R^{\{0,\dots,9\}}}{\dotprod{\b{x}}{\one} = 55}$ and~$\Hyp(B)$. Observe also that~$\vertex{\spine'} - \vertex{\spine} = 4\,(e_8-e_3)$.

\medskip
Using these vertex and facet descriptions, we arrive to the main result of this paper, whose proof is treated in the next section.

\begin{theorem}
\label{theo:signedTreeAssociahedron}
The spine fan~$\spineFan(\tree)$ is the normal fan of the \defn{signed tree associahedron}~$\Asso(\tree)$, defined equivalently as
\begin{enumerate}[(i)]
\item the convex hull of the points~$\vertex{\spine}$ for all maximal signed spines~$\spine$ on~$\tree$,
\item the intersection of the hyperplane~$\HH$ with the half-spaces~$\HS(B)$ for all signed building blocks~$B$ of~$\tree$.
\end{enumerate}
In particular, the boundary complex of the polar of~$\Asso(\tree)$ is isomorphic to the signed nested complex~$\nestedComplex(\tree)$.
\end{theorem}

Before proving this statement, we underline the key feature of~$\Asso(\tree)$. Namely, any signed spine~${\spine \in \spinePoset(\tree)}$ determines the geometry around its corresponding face~$\face{\spine}$ of~$\Asso(\tree)$:
\begin{itemize}
\item the cone of~$\Asso(\tree)$ at~$\face{\spine}$ coincides with the \defn{incidence cone~$\primalCone(\spine)$} of~$\spine$, that is
\begin{gather*}
\cone\bigset{\vertex{\spine''} - \vertex{\spine'}}{\spine', \spine'' \in \spinePoset(\tree) \text{ maximal}, \, \spine' \text{ refining } \spine} \qquad\qquad \\
= \; \cone\bigset{e_v-e_u}{u \le v \text{ in the transitive closure of } \spine}.
\end{gather*}
\item the normal cone of~$\face{\spine}$ coincides with the \defn{braid cone~$\normalCone(\spine)$} of~$\spine$, that is
\begin{gather*}
\cone\bigset{\b{g} \in \HH}{\max_{\b{x} \in \Asso(\tree)} \dotprod{\b{g}}{\b{x}} = \dotprod{\b{g}}{\vertex{\spine'}}, \text{ for all } \spine' \in \spinePoset(\tree) \text{ maximal refining } \spine} \\
\qquad\qquad\; = \; \bigset{\b{x} \in \HH}{x_u \le x_v, \text{ for all } u \le v \text{ in the transitive closure of } \spine}.
\end{gather*}
\end{itemize}

\begin{example}[Signed path, continued]
For signed paths, we obtain C.~Hohlweg and C.~Lange's realizations of the classical associahedron~\cite{HohlwegLange}. Their construction and its interpretation in terms of spines~\cite{LangePilaud-spines} was the guiding light of this work. \fref{fig:pathAssociahedra} represents all possible \mbox{$3$-dimensional} associahedra obtained that way (changing signs of the leaves does not change the geometry of the realization).

\begin{figure}[h]
  \capstart
  \centerline{\includegraphics[width=\textwidth]{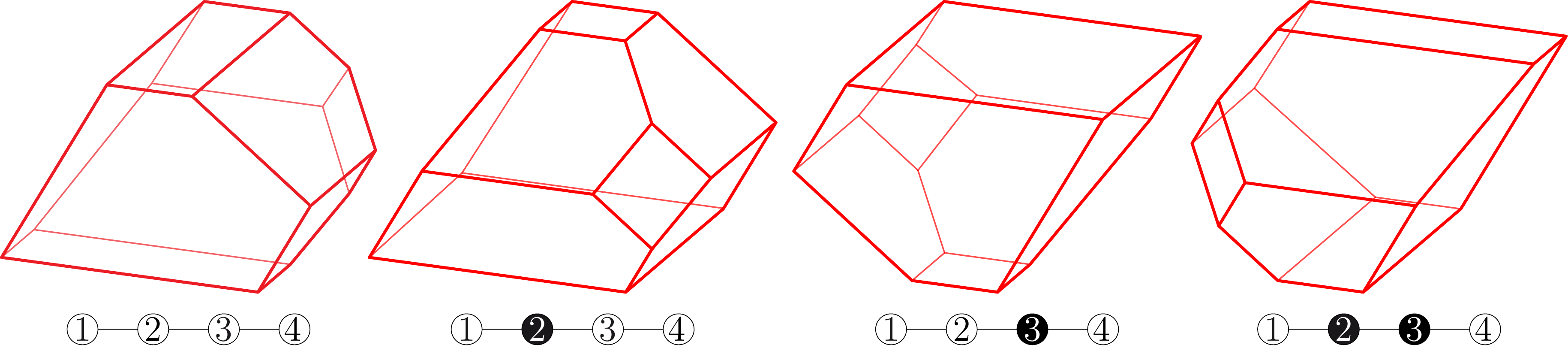}}
  \caption{C.~Hohlweg and C.~Lange's $3$-dimensional associahedra~\cite{HohlwegLange}.}
  \label{fig:pathAssociahedra}
\end{figure}
\end{example}

\begin{example}[Unsigned tree, continued]
When~$\tree$ has only negative vertices, we obtain a new realization of the tree associahedron different from the constructions of M.~Carr and S.~Devadoss~\cite{CarrDevadoss, Devadoss}, A.~Postnikov~\cite{Postnikov}, and A.~Zelevinsky~\cite{Zelevinsky}. The normal fan is the same, but the right hand sides of the inequalities are chosen to coincide with that of the permutahedron. \fref{fig:unsignedTreeAssociahedra} illustrates all $4$-dimensional unsigned tree associahedra constructed this way. The reader is invited to associate a maximal spine to each vertex and a signed tube to each facet of these polytopes to visualize the combinatorics of the nested complex.

\begin{figure}[p]
  \capstart
  \centerline{\includegraphics[width=.8\textwidth]{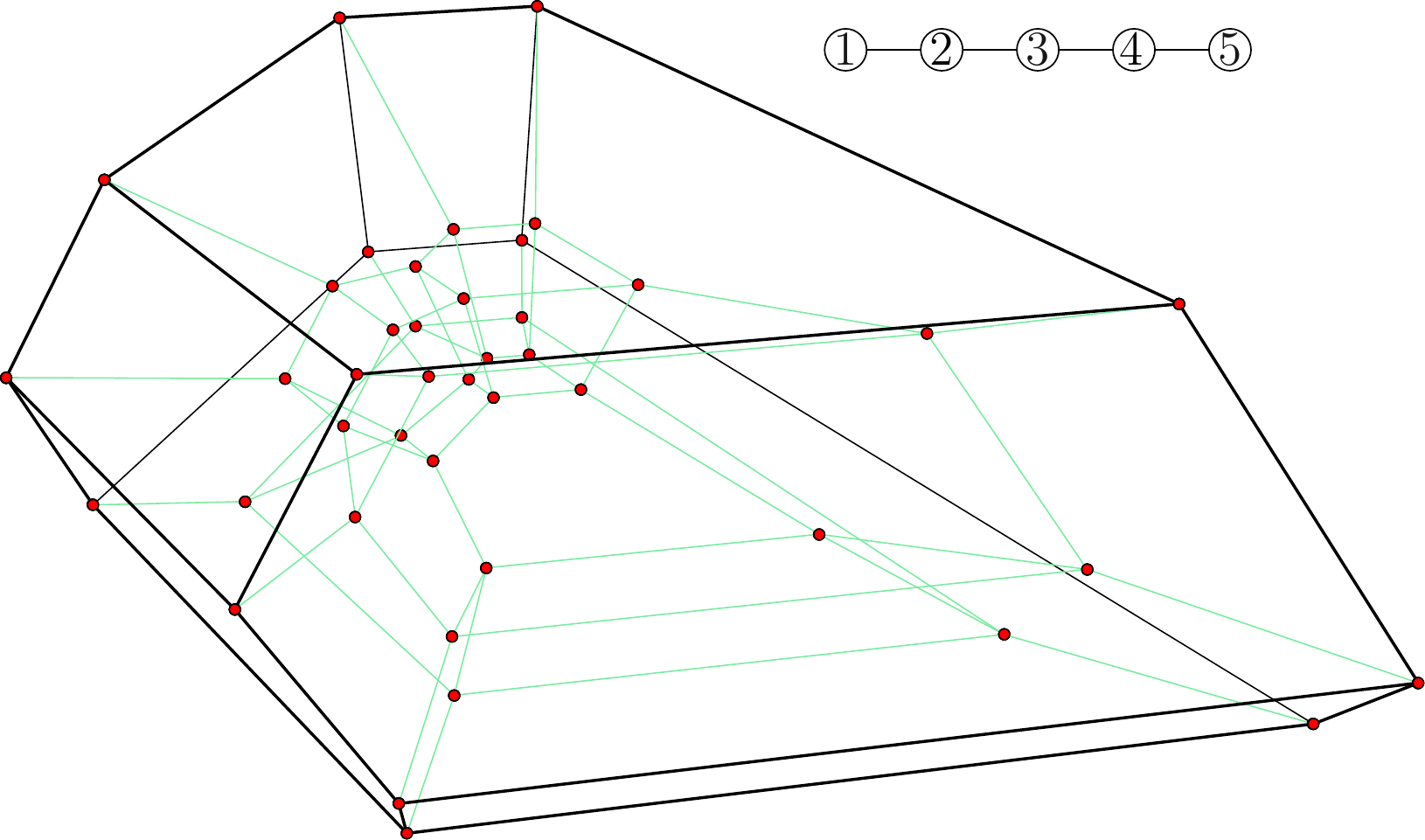}}
  \vspace{.5cm}
  \centerline{\includegraphics[width=.8\textwidth]{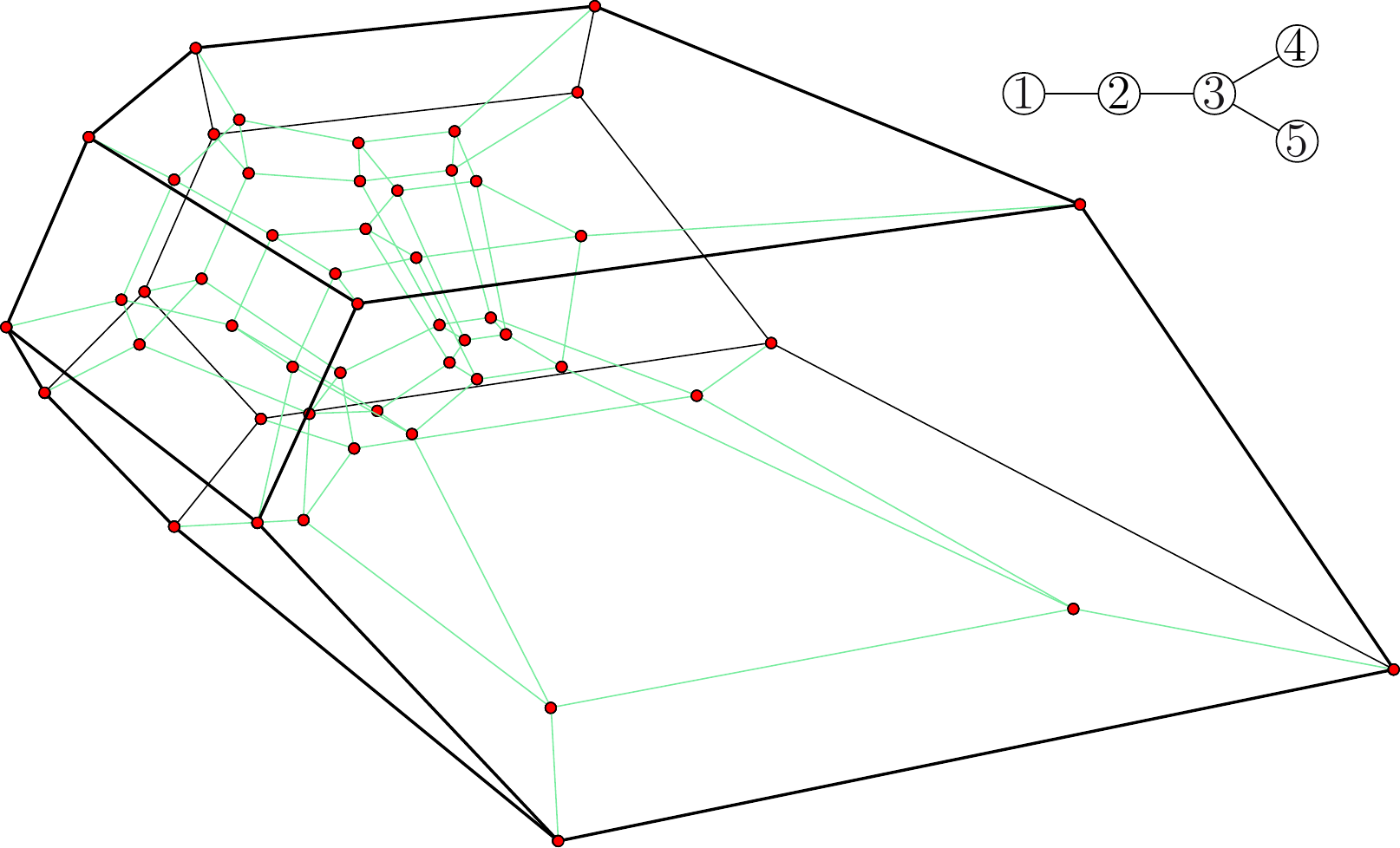}}
  \centerline{\includegraphics[width=.8\textwidth]{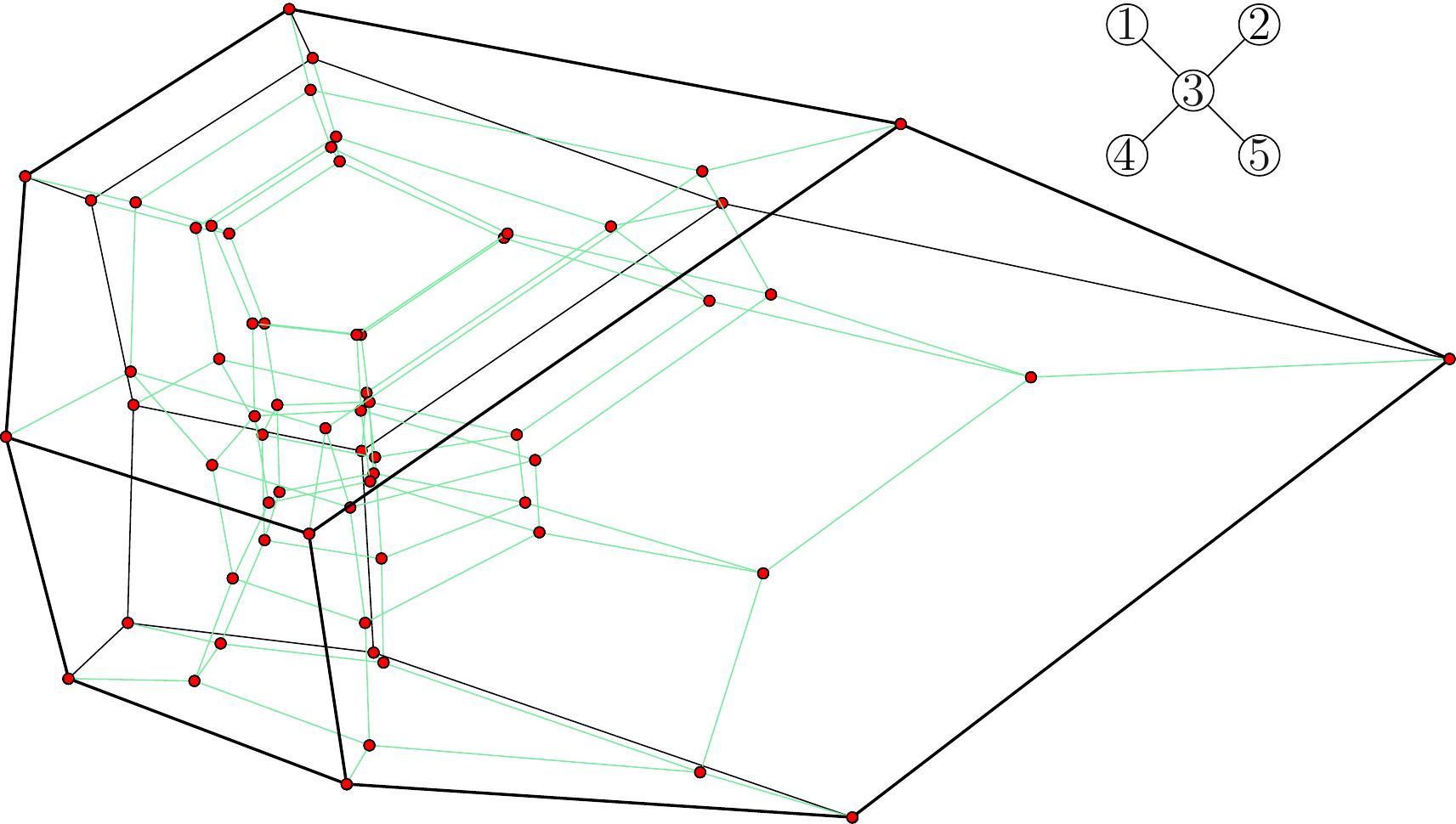}}
  \caption{All $4$-dimensional unsigned tree associahedra.}
  \label{fig:unsignedTreeAssociahedra}
\end{figure}
\end{example}

\begin{example}[Tripod, continued]
Figures~\ref{fig:tripodFacets} and~\ref{fig:tripodVertices} illustrate the facet and vertex descriptions of the signed tree associahedra~$\Asso(\tripodWhite)$ and~$\Asso(\tripodBlack)$ realizing the signed nested complexes of \fref{fig:tripodNestedComplex}.

\begin{figure}[p]
  \capstart
  \centerline{\includegraphics[scale=.75]{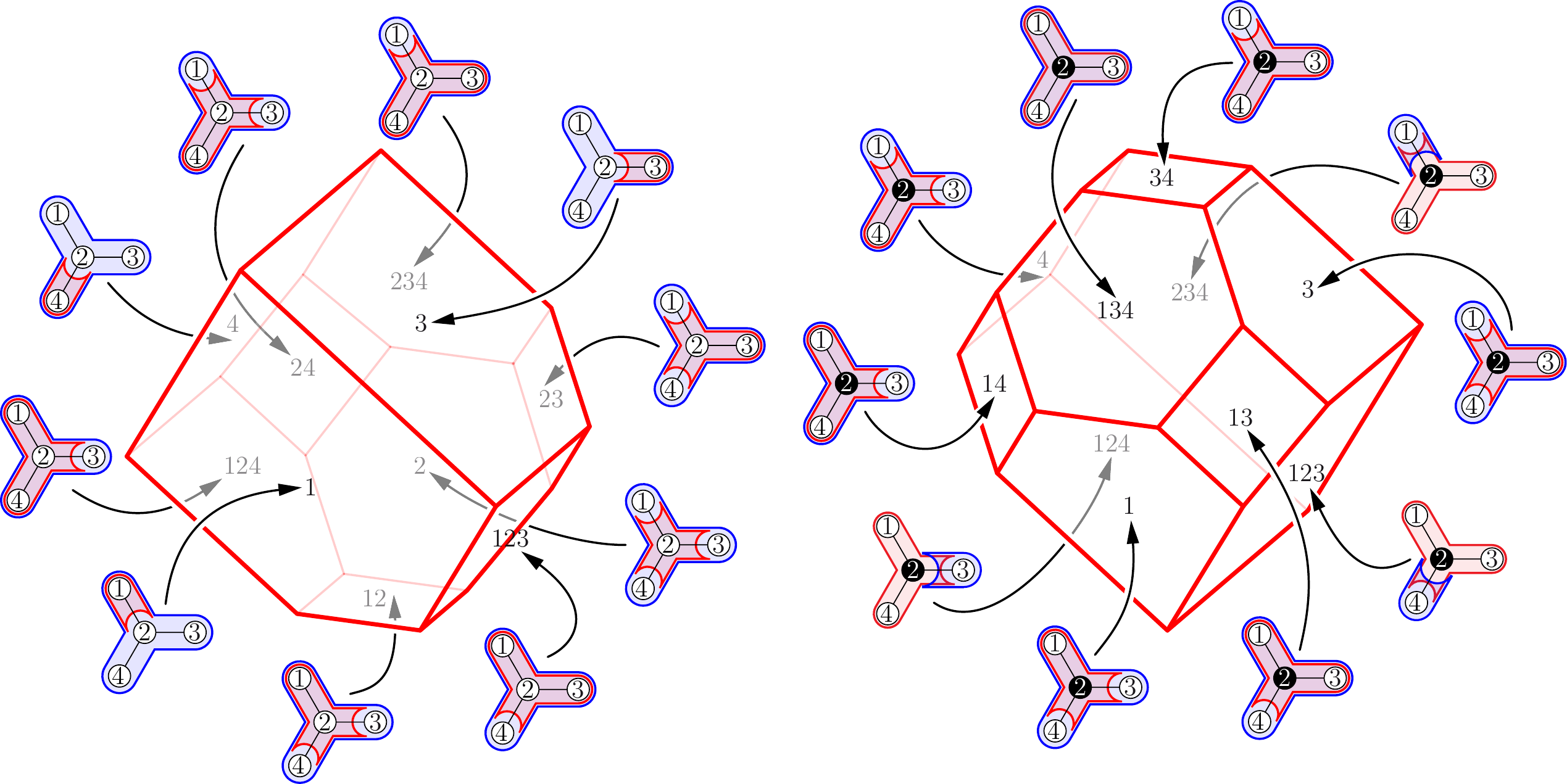}}
  \caption[Facet descriptions of the signed tree associahedra for tripods]{Facet descriptions of the signed tree associahedra~$\Asso(\tripodWhite)$ and~$\Asso(\tripodBlack)$.}
  \label{fig:tripodFacets}
\end{figure}

\begin{figure}[p]
  \capstart
  \centerline{\includegraphics[scale=.75]{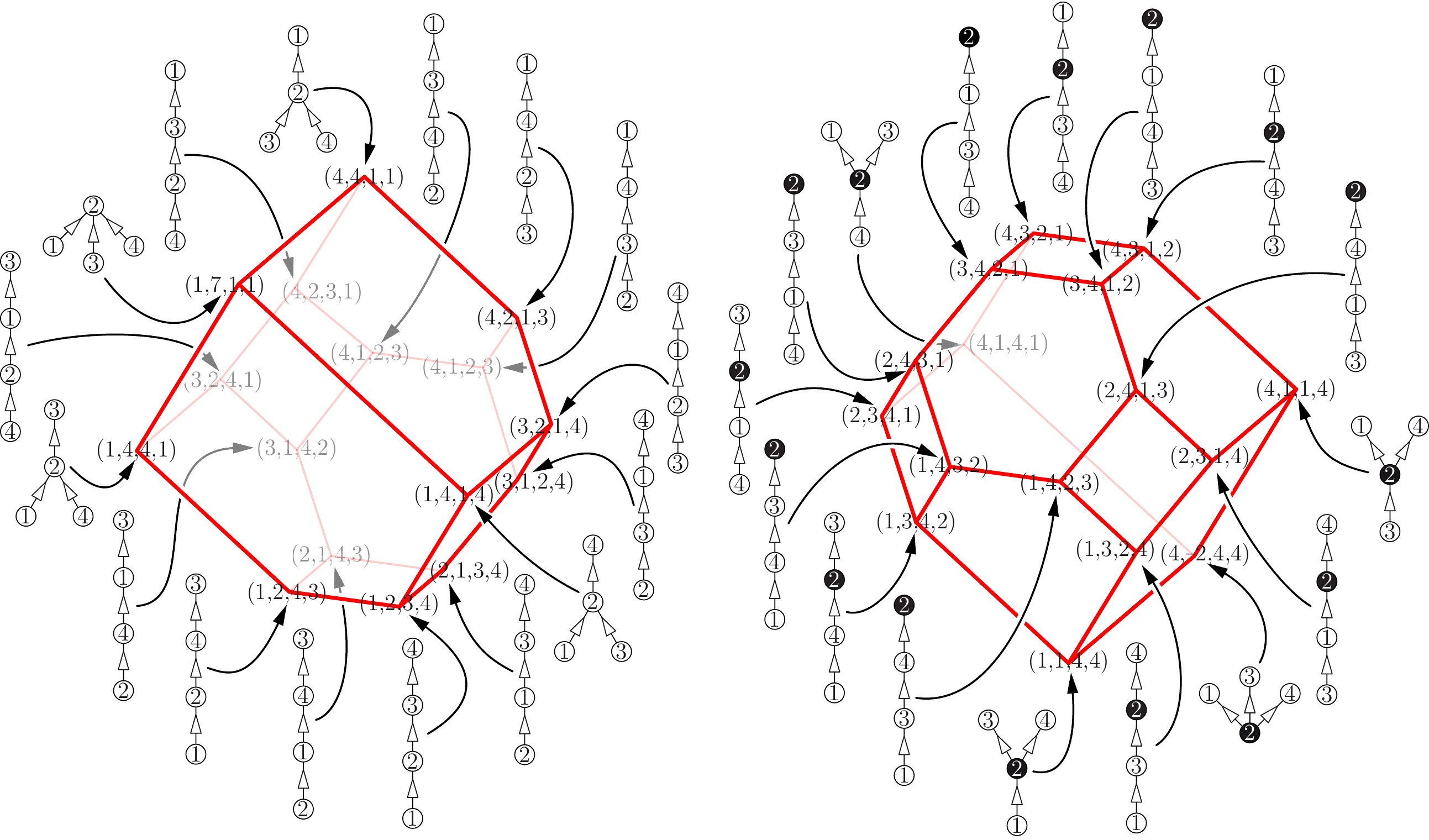}}
  \caption[Vertex descriptions of the signed tree associahedra for tripods]{Vertex descriptions of the signed tree associahedra~$\Asso(\tripodWhite)$ and~$\Asso(\tripodBlack)$.}
  \label{fig:tripodVertices}
  \vspace*{-2cm}
\end{figure}
\end{example}


\subsection{Proof of Theorem~\ref{theo:signedTreeAssociahedron}}

We devote this section to the proof of Theorem~\ref{theo:signedTreeAssociahedron}. We have to show that the polytopes defined by the vertex and facet descriptions above coincide and that their combinatorial structure is that of the signed nested complex. Our first step is to prove that for any maximal signed spine~$\spine$ on~$\tree$, the point~$\vertex{\spine} \in \R^\ground$ is indeed the intersection of the hyperplanes~$\Hyp(B)$ for the signed building blocks~$B$ given by the source sets of~$\spine$. It is a consequence of the following slightly more general statement on directed trees.

\begin{proposition}
\label{prop:directedTrees}
Let~$\spine$ be a directed tree on a signed ground set~$\ground = \ground^- \sqcup \ground^+$ with~$\nu$ elements, such that each node~$v \in \ground^-$ has at most one outgoing arc~$r_v$ while each node~$v \in \ground^+$ has at most one incoming arc~$r_v$. We consider the set~$\Pi(\spine)$ of all (undirected and simple) paths in~$\spine$, including the trivial paths reduced to a single node. The node valuation~$\psi : \ground \to \Z$ defined by
\[
\psi(v) \eqdef
\begin{cases}
	|\{\text{paths in } \Pi(\spine) \text{ containing } v \text{ but not } r_v\}| & \text{if } v \in \ground^-\\
	\nu + 1 -|\{\text{paths in } \Pi(\spine) \text{ containing } v \text{ but not } r_v\}| & \text{if } v \in \ground^+
\end{cases}
\]
fulfills
\[
\sum_{v \in \ground} \psi(v) = \binom{\nu + 1}{2}
\qquad\text{and}\qquad
\sum_{v \in \source(r)} \psi(v) = \binom{|\source(r)| + 1}{2}
\]
for any arc~$r$ of~$\spine$ with source set~$\source(r)$.
\end{proposition}

\begin{proof}
We start with the first identity on the sum of the valuation~$\psi$ over~$\ground$. The proof is based on a double counting argument: instead of summing over all nodes~$v \in \ground$ the contribution to~$\psi(v)$ of each path~$\pi \in \Pi(\spine)$, we rather sum over all paths~$\pi \in \Pi(\spine)$ their contribution to~$\psi(v)$ for each node~$v \in \ground$.

First, we evaluate the contribution of a path~$\pi \in \Pi(\spine)$ to the valuation~$\psi(v)$ of each node~$v \in \ground$. We assume first that the endpoints of~$\pi$ are in~$\ground^-$. By definition, $\pi$ can only contribute to the valuations of its nodes. If~$\pi$ is a trivial path, reduced to a single node~$v \in \ground^-$, then it contributes~$1$ to~$\psi(v)$. Assume now that~$\pi$ is a non-trivial path with distinct endpoints~$u,v \in \ground^-$. Then, $\pi$ contributes~$1$ or~$0$ to~$\psi(u)$ and~$\psi(v)$ according on whether~$\pi$ contains or not the arcs~$r_u$ and~$r_v$, or equivalently on whether~$\pi$ is incoming or outgoing at~$u$ and~$v$. Moreover, $\pi$ only contributes to the valuations~$\phi(w)$ of its internal nodes~$w \in \pi$ where its edge orientation is reversed. More precisely, $\pi$ contributes~$1$ to the valuation~$\psi(w)$ of each node~$w \in \pi$ with two incoming arcs in~$\pi$, $-1$ to the valuation~$\psi(w)$ of each node~$w \in \pi$ with two outgoing arcs in~$\pi$, and~$0$ to the valuations of all other nodes. Observe that the nodes with two incoming arcs in~$\pi$ and that with two outgoing arcs in~$\pi$ alternate along~$\pi$, even if they can be separated by nodes of~$\pi$ with one incoming and one outgoing arc in~$\pi$. Moreover, when we traverse along~$\pi$ from~$u$ to~$v$, this alternating sequence starts (resp.~ends) by a node with two incoming arcs if~$\pi$ is outgoing at~$u$ (resp.~at~$v$), while it starts (resp.~ends) by a node with two outgoing arcs if~$\pi$ is incoming at~$u$ (resp.~at~$v$). Thus, the total contribution of~$\pi$ to the valuations of its internal vertices is always~$1$ if the two endpoints of~$\pi$ are in~$\ground^-$.

It remains to take into account the paths with some endpoints in~$\ground^+$. First, a trivial path reduced to a single node~$v \in \ground^+$ contributes~$-1$ to~$\psi(v)$. Moreover a path~$\pi$ with an endpoint~${v \in \ground^+}$ contributes $0$ to the valuation~$\psi(v)$ of~$v$ if~$\pi$ is incoming at~$v$ and~$-1$ if~$\pi$ is outgoing at~$v$. Therefore, for each positive vertex~$v \in \ground^+$, we over-counted~$2$ in the contribution to~$\psi(v)$ of the trivial path reduced to the node~$v$, and~$1$ in the contribution to~$\psi(v)$ of each path with precisely one endpoint at~$v$. The number of such paths is~$|\ground| - 1 = \nu - 1$ since we just need to choose the other endpoint. We therefore obtain that
\begin{align*}
\sum_{v \in \ground} \psi(v) & = |\ground^+| \cdot (\nu + 1) + \sum_{v \in \ground} \sum_{\pi \in \Pi(\spine)} \text{contribution of } \pi \text{ to } \psi(v) \\
& = |\ground^+| \cdot (\nu + 1) + \bigg( \sum_{\pi \in \Pi(\spine)} 1 \bigg) - |\ground^+| \cdot (2 + \nu - 1) = |\Pi(\spine)| = \binom{\nu + 1}{2},
\end{align*}
which concludes the proof of the first identity of the statement.

Consider now an arc~$r$ of~$\spine$ with source set~$\source(r)$ and sink set~$\sink(r)$, and a path~$\pi$ of~$\Pi(\spine)$. The contribution of~$\pi$ to the sum~$\Sigma \eqdef \sum_{v \in \source(r)} \psi(v)$ depends on the position of its endpoints:
\begin{enumerate}[(i)]
\item If the two endpoints of~$\pi$ are in~$\source(r)$, then all its internal nodes are in~$\source(r)$ and thus the total contribution of~$\pi$ to~$\Sigma$ is still~$1$ minus the number of endpoints of~$\pi$ in~$\ground^+$.
\item If the two endpoints of~$\pi$ are in~$\sink(r)$, then all its internal nodes are in~$\sink(r)$ and~$\pi$ does not contribute to~$\Sigma$.
\item If~$\pi$ has one endpoint~$u$ in~$\source(r)$ and one endpoint~$v$ in~$\sink(r)$, then it contributes $1$ or~$0$ to~$\Sigma$ depending on whether~$u \in \ground^-$ or~$u \in \ground^+$. The argument here is similar to the one used before. Namely, $\pi$ contributes $1$ to the valuation of its internal nodes with two incoming arcs in~$\pi$, and $-1$ to the valuation of its internal nodes with two outgoing arcs in~$\pi$. When we traverse along~$\pi$ from~$u$ to~$v$, these two types of nodes alternate, starting with one or the other type according on whether~$\pi$ is incoming or outgoing at~$u$. Moreover, the last such node along~$\pi$ which lies in~$\source(r)$ is necessarily a node with two outgoing arcs in~$\pi$, since~$r$ is directed from its source to its sink.
\end{enumerate}
Said differently, we count~$1$ for the contribution of~$\pi$ to~$\Sigma$ when its both endpoints are in~$\source(r)$ and $0$ otherwise, but we over-counted~$\nu + 1$ around each node of~$\source(r) \cap \ground^+$. It follows~that
\begin{align*}
\sum_{v \in \ground} \psi(v) & = |\source(r) \cap \ground^+| \cdot (\nu + 1) + \sum_{v \in \source(r)} \sum_{\pi \in \Pi(\spine)} \text{contribution of } \pi \text{ to } \psi(v) \\
& = |\source(r) \cap \ground^+| \cdot (\nu + 1) + \bigg( \sum_{\substack{\pi \in \Pi(\spine) \\ \boundary\pi \subseteq \source(r)}} 1 \bigg) - |\source(r) \cap \ground^+| \cdot (\nu + 1) \\
& = |\set{\pi \in \Pi(\spine)}{\boundary\pi \subseteq \source(r)}| = \binom{|\source(r)| + 1}{2}. \qedhere
\end{align*}

\end{proof}

Our second step is to study the difference of the coordinates of two vertices corresponding to two adjacent maximal signed nested sets of~$\tree$. We need the following statement.

\begin{proposition}
\label{prop:directionFlip}
Let~$\spine$ and~$\spine'$ be two adjacent maximal signed spines on~$\tree$, such that~$\spine'$ is obtained from~$\spine$ by flipping an arc joining the node~$u$ to the node~$v$. Then, the difference~$\vertex{\spine'}-\vertex{\spine'}$ is a positive multiple of~$e_u - e_v$.
\end{proposition}

\begin{proof}
\enlargethispage{.6cm}
We analyze the four different situations of \fref{fig:flip}. In all situations, we let~$U$ be the set of nodes in the subspines of~$\spine$ attached to~$u$ by arcs distinct from~$i$ and~$r$, and similarly we let~$V$ be the set of nodes in the subspines of~$\spine$ attached to~$v$ by arcs distinct from~$o$ and~$r$. See \fref{fig:flip}. For $W \subset \ground$ and~$w \in \ground$, we denote by~$\binomspec{W}{2}{w}$ the number of choices of two nodes of~$W$ in distinct connected components of~$\spine \ssm \{w\}$. In the leftmost situation, we compute
\begin{gather*}
\vertex{\spine}_u = \big( |U| + 1 \big) \cdot \big( |\source(i)| + 1 \big) + \binomspec{U}{2}{u} \quad\text{and}\quad \vertex{\spine}_v = \big( |U| + |\source(i)| + 2 \big) \cdot \big( |V| + 1 \big) + \binomspec{V}{2}{v}, \\
\vertex{\spine'}_u = \big( |U| + 1 \big) \cdot \big( |V| + |\source(i)| + 2 \big) + \binomspec{U}{2}{u} \quad\text{and}\quad \vertex{\spine'}_v = \big( |\source(i)| + 1 \big) \cdot \big( |V| + 1 \big) + \binomspec{V}{2}{v}.
\end{gather*}
Moreover, the coordinates~$\vertex{\spine}_w$ and~$\vertex{\spine'}_w$ coincide if~$w \in \ground \ssm \{u,v\}$. Therefore, we obtain
\[
\vertex{\spine'} - \vertex{\spine} = \big( |U| + 1 \big) \cdot \big( |V| + 1 \big) \cdot (e_u - e_v).
\]
We prove similarly that this last equality is valid in the other three situations of \fref{fig:flip}.
\end{proof}

To conclude the proof of Theorem~\ref{theo:signedTreeAssociahedron}, we can now apply the following result of C.~Hohlweg, C.~Lange, and H.~Thomas~\cite{HohlwegLangeThomas}.

\begin{theorem}[\protect{\cite[Theorem~4.1]{HohlwegLangeThomas}}]
\label{theo:HohlwegLangeThomas}
Given a pointed complete simplicial fan~$\fan$ in~$\R^d$, consider
\begin{enumerate}[(i)]
\item for each ray~$\rho$ of~$\fan$, a half-space~$\HS_\rho$ of~$\R^d$ containing the origin and defined by an hyperplane~$\Hyp_\rho$ orthogonal to~$\rho$,
\item for each maximal cone~$C$ of~$\fan$, a point~$\b{x}_C$ of~$\R^d$ contained in~$\Hyp_\rho$ for each ray~$\rho \in C$.
\end{enumerate}
If for each pair~$C,C'$ of adjacent maximal cones of~$\fan$, the vector~$\b{x}_{C'} - \b{x}_C$ points from~$C$ to~$C'$, then the descriptions
\[
\bigcap_{\rho \text{ ray of } \fan} \HS_\rho \qquad\quad \text{and} \qquad\quad \conv\set{\b{x}_C}{C \text{ maximal cone of } \fan}
\]
define the same polytope whose normal fan is~$\fan$. \qed
\end{theorem}

\begin{proof}[Proof of Theorem~\ref{theo:signedTreeAssociahedron}]
In the context of the signed tree associahedron, the conditions of Theorem~\ref{theo:HohlwegLangeThomas} are guarantied by Propositions~\ref{prop:directedTrees} and~\ref{prop:directionFlip}. It follows that~$\Asso(\tree)$ realizes the signed nested complex~$\nestedComplex(\tree)$ and that its normal fan is indeed the spine fan~$\spineFan(\tree)$.
\end{proof}


\subsection{Further geometric properties}
\label{sec:furtherGeomProp}

To complete this section, we explore further geometric properties of the signed tree associahedron~$\Asso(\tree)$.

\subsubsection{Parallel facets}

We start by considering the pairs of parallel facets of~$\Asso(\tree)$.

\begin{proposition}
\label{prop:complementaryBuildingBlocks}
Given any edge~$e$ of~$\tree$, the vertex sets of the two connected components of~$\tree \ssm \{e\}$ form complementary signed building blocks of~$\building(\tree)$, and therefore define two opposite parallel facets of~$\Asso(\tree)$. Moreover, these are the only pairs of parallel facets of~$\Asso(\tree)$.
\end{proposition}

\begin{proof}
For any edge~$e$ of~$\tree$, we have seen that the vertex sets of the two connected components of~$\tree \ssm \{e\}$ are signed building blocks in Example~\ref{exm:buildingBlocksFromEdge}, and they are clearly complementary. The normal vectors of their corresponding facets are therefore characteristic vectors of complementary subsets of~$\ground$, which ensures that the facets are indeed parallel.

Consider now a pair of complementary signed building blocks of~$\tree$. Since they are complementary, they are both simultaneously negative and positive convex. Therefore, they induce complementary connected subtrees of~$\tree$. It follows that they are the two connected components of~$\tree \ssm \{e\}$ for an edge~$e$ of~$\tree$.
\end{proof}

Note that the pairs of parallel facets of~$\Asso(\tree)$ do not depend on the signs of~$\tree$, but only on its underlying unsigned tree. We now consider the \defn{parallelepiped}~$\Para(\tree)$ defined by the $\nu - 1$ pairs of parallel facets of~$\Asso(\tree)$. We prove in the next statement that~$\Para(\tree)$ is a translated and dilated graphical zonotope of~$\tree$. The \defn{graphical zonotope} of a graph~$\graphG$ on~$\ground$ is the polytope~$\Zono(\graphG) \subset \HH$ defined as the Minkowski sum of the segments~$[e_u, e_v]$, for all edges~${u\!-\!v}$ of~$\graphG$. The $k$-dimensional faces of~$\Zono(\graphG)$ are in bijection with the pairs~$(\Omega, \Gamma)$, where~$\Omega$ is a partition of the vertex set of~$\graphG$ into~$\nu - k$ subsets each inducing a connected subgraph of~$\graphG$, and~$\Gamma$ is an acyclic orientation on the quotient~$\graphG/\Omega$. In particular, the vertices of~$\Zono(\graphG)$ correspond to the acyclic orientations on~$\graphG$. For example, the graphical zonotope~$\Zono(K_\ground)$ of the complete graph~$K_\ground$ on~$\ground$ is (a translate of) the permutahedron~$\Perm(\ground)$, and the graphical zonotope~$\Zono(\tree)$ of a tree~$\tree$ is a parallelepiped. Here, we weight each edge~$e$ of~$\tree$ with the number~$\pi(e)$ of paths in~$\tree$ containing the edge~$e$, that is, by the product of the cardinalities of the connected components of~$\tree \ssm \{e\}$. We obtain the following statement.

\begin{proposition}
\label{prop:paraMinkowskiSum}
The parallelepiped~$\Para(\tree)$ defined by the $\nu - 1$ pairs of parallel facets of~$\Asso(\tree)$ is a translate of the Minkowski sum
\[
\sum_{u - v \text{ in } \tree} \pi(u\!-\!v) \cdot [e_u, e_v],
\]
where~$\pi({u\!-\!v})$ denotes the number of paths in~$\tree$ containing the edge~${u\!-\!v}$.
\end{proposition}

\begin{proof}
Since~$\Para(\tree)$ is a parallelepiped whose facets have the same directions as that of~$\Zono(\tree)$, it can as well be written as (a translate of) a Minkowski sum of dilates of the segments~$[e_u, e_v]$ for the edges~${u\!-\!v}$ in~$\tree$. To determine the dilation factors, we observe that the facet defining inequalities of~$\Para(\tree)$ are also facet defining inequalities of~$\Perm(\ground)$. The latter is (a translate of) the Minkowski sum of all segments~$[e_x, e_y]$ for~$x,y \in \ground$. Pick an edge~${u\!-\!v}$ of~$\tree$, and let~$X$ denote the Minkowski summand of~$\Para(\tree)$ in the direction~${e_u-e_v}$, and~$\rho$ denote the projection on the line~$\R \cdot (e_u-e_v)$ parallel to the hyperplane spanned by~$\set{e_{u'}-e_{v'}}{u'\!\!-\!v' \text{ in } \tree, \; u'\!\!-\!v' \ne u\!-\!v}$. Then the Minkowski summand~$X$ of~$\Para(\tree)$ is the sum over all~$x,y \in \ground$ of the projection~$\rho ([e_x,e_y])$ of the Minkowski summand of~$\Perm(\ground)$. But the projection~$\rho([e_x,e_y])$ is precisely $[e_u,e_v]$ if~${u\!-\!v}$ lies on the path from~$x$ to~$y$ in~$\tree$, and it is just a point otherwise. The dilation factor of~$[e_u,e_v]$ is therefore the number~$\pi({u\!-\!v})$ of paths on~$\tree$ containing the edge~${u\!-\!v}$ of~$\tree$.
\end{proof}

The precise translation from the parallelepiped~$\Para(\tree)$ to the Minkowski sum~${\sum \pi(u\!-\!v) [e_u, e_v]}$ will be explicit in Example~\ref{exm:MinkowskiDecompositionPara}. We now derive some relevant consequences of Proposition~\ref{prop:complementaryBuildingBlocks}.

\begin{corollary}
A vertex~$v \in \ground$ is a leaf of~$\tree$ iff~$\{v\}$ and~$\ground \ssm \{v\}$ are both in~$\building(\tree)$.
\end{corollary}

\begin{corollary}
\label{coro:buildingSetToTree}
Except the signs of its leaves, the signed ground tree~$\tree$ can be reconstructed from its signed building set~$\building(\tree)$.
\end{corollary}

\begin{proof}
Given the signed building set~$\building(\tree)$, we can first deduce all cuts of~$\tree$ from the complementary signed building blocks. We can therefore reconstruct the unsigned tree underlying~$\tree$. Then, any vertex~$v \in \ground$ which is not a leaf of~$\tree$ is in~$\ground^-$ if~$\{v\}$ is a building block, and in~$\ground^+$ if~$\ground \ssm \{v\}$ is a building block. We cannot reconstruct the signs of the leaves of~$\tree$ since they do not influence~$\building(\tree)$ (see Proposition~\ref{prop:nestedComplexPreservingOperations}(iv)).
\end{proof}

\subsubsection{Matriochka polytopes}

The permutahedron~$\Perm(\ground)$, the signed tree associahedron~$\Asso(\tree)$, and the parallelepiped~$\Para(\tree)$ are nested polytopes, meaning that
\[
\Perm(\ground) \; \subset \; \Asso(\tree) \; \subset \; \Para(\tree).
\]

\begin{example}[Tripod, continued]
\fref{fig:tripodPermutahedronAssociahedronCube} represents the signed tree associahedra~$\Asso(\tripodWhite)$ and~$\Asso(\tripodBlack)$ sandwiched between the permutahedron~$\Perm([4])$ and the parallelepiped~$\Para(\tripodWhite)$. Common vertices of these polytopes are marked with colored disks. 

\begin{figure}[h]
  \capstart
  \centerline{\includegraphics[scale=.75]{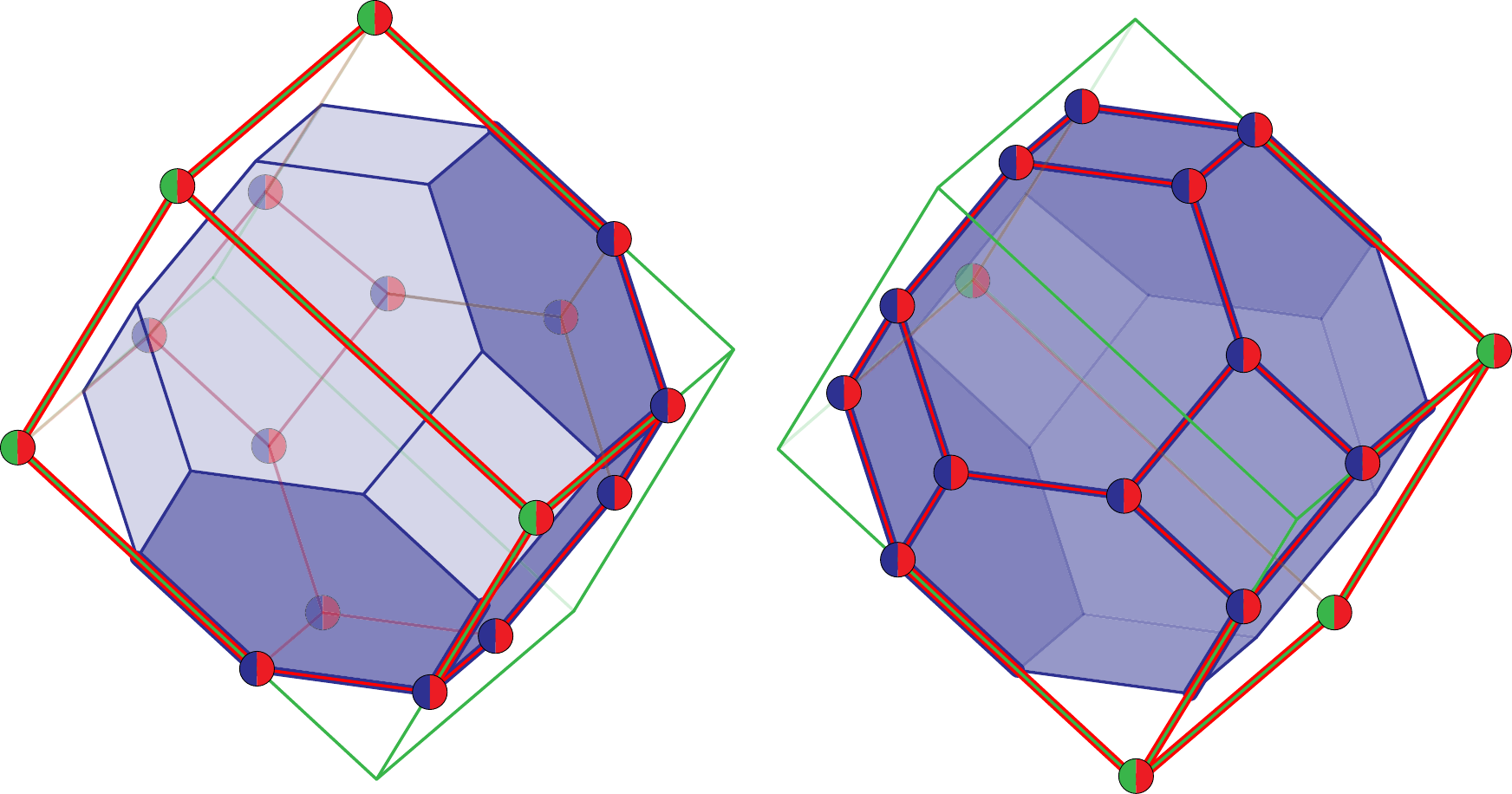}}
  \caption[The signed tree associahedra for tripods are sandwiched between the permutahedron and the parallelepiped defined by their pairs of parallel facets]{The red signed tree associahedra~$\Asso(\tripodWhite)$ and~$\Asso(\tripodBlack)$ are sandwiched between the blue permutahedron~$\Perm([4])$ and the green parallelepiped~$\Para(\tripodWhite)$. Common vertices of~$\Asso(\tree)$ and~$\Perm(\ground)$ (resp.~and~$\Para(\tree)$) are marked with red--blue disks (resp.~red--green disks). Refer to \fref{fig:tripodVertices} to understand which spines on~$\tripodWhite$ and~$\tripodBlack$ do these vertices correspond~to.}
  \label{fig:tripodPermutahedronAssociahedronCube}
\end{figure}
\end{example}

\begin{example}[Signed path, continued]
\fref{fig:pathPermutahedronAssociahedronCube} represents C.~Hohlweg and C.~Lange's associahedra~\cite{HohlwegLange} sandwiched between the permutahedron~$\Perm([4])$ and the parallelepiped~$\Para(\pathG)$. Observe that all these polytopes have two common vertices, corresponding to the two linear orientations on~$\pathG$.

\begin{figure}[h]
  \capstart
  \centerline{\includegraphics[width=\textwidth]{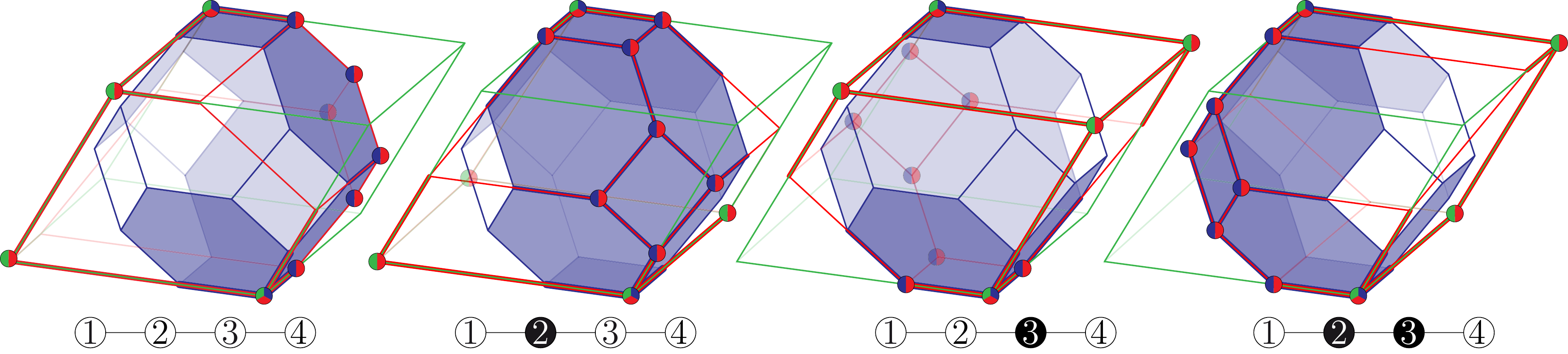}}
  \caption{C.~Hohlweg and C.~Lange's red $3$-dimensional path associahedra~$\Asso(\pathG)$ are sandwiched between the blue permutahedron~$\Perm([4])$ and the green parallelepiped~$\Para(\pathG)$. Common vertices of~$\Asso(\pathG)$ and~$\Perm(\ground)$ (resp.~and~$\Para(\pathG)$) are marked with red--blue disks (resp.~red--green disks), and common vertices of the three polytopes are marked by red--blue--green disks.}
  \label{fig:pathPermutahedronAssociahedronCube}
\end{figure}
\end{example}

\subsubsection{Normal fans}

The inclusions~$\Perm(\ground) \subset \Asso(\tree) \subset \Para(\tree)$ reflect on their normal fans. As already discussed earlier,
\begin{enumerate}[(i)]
\item maximal cones of the normal fan of~$\Perm(\ground)$ correspond to linear orders on~$\ground$,
\item maximal cones of the normal fan of~$\Asso(\tree)$ correspond to maximal signed spines on~$\tree$,
\item maximal cones of the normal fan of~$\Para(\tree)$ correspond to orientations of~$\tree$.
\end{enumerate}
We have seen in details in Section~\ref{subsec:surjection} that the normal fan of~$\Perm(\ground)$ refines that of~$\Asso(\tree)$, thus defining a surjection~$\surjectionPermAsso$ from the linear orders on~$\ground$ to the maximal signed spines on~$\tree$. The fiber of a maximal signed spine~$\spine$ is the set of all linear extensions of the transitive closure of~$\spine$. Similarly, the normal fan of~$\Perm(\ground)$ clearly refines that of~$\Para(\tree)$, thus defining a surjection~$\surjectionPermPara$ from the linear orders on~$\ground$ to the orientations of~$\tree$. To obtain the orientation~$\surjectionPermPara(\prec)$ corresponding to a linear order~$\prec$ on~$\ground$, we just orient each edge of~$\tree$ from its smallest endpoint to its largest endpoint in the order~$\prec$. Again, the fiber of an orientation~$\orientation$ on~$\tree$ is the set of all linear extensions of the transitive closure of~$\orientation$. Finally, the following lemma ensures that the normal fan of~$\Asso(\tree)$ refines that of~$\Para(\tree)$.

\begin{lemma}
For any edge~${u\!-\!v}$ of~$\tree$ and any maximal spine~$\spine$ of~$\tree$, there is an oriented path in~$\spine$ either from~$u$ to~$v$ or from~$v$ to~$u$.
\end{lemma}

\begin{proof}
Consider the unique path~$\pi$ between~$u$ and~$v$ in~$\spine$. If it is not an oriented path, there exists a node~$w$ of~$\spine$ with either two incoming or two outgoing arcs in~$\pi$. In both cases, this node contradicts the local Condition~(ii) of Definition~\ref{def:spine} since~$u$ and~$v$ are adjacent in~$\tree$ and thus in the same connected component of~$\tree \ssm \{w\}$ for any~$w \in \ground \ssm \{u,v\}$.
\end{proof}

This defines a surjection~$\surjectionAssoPara$ from the maximal signed spines on~$\tree$ to the orientations of~$\tree$. To obtain the orientation~$\surjectionAssoPara(\spine)$ corresponding to a maximal signed spine~$\spine$ of~$\tree$, we just orient each edge~${u\!-\!v}$ of~$\tree$ according to the orientation of the unique path between~$u$ and~$v$ in~$\spine$.

\begin{example}[Unsigned paths, continued]
When the ground tree is a path~$\pathG$ with only negative vertices, we have seen in Example~\ref{exm:BST} that the maximal signed spines are precisely the binary search trees with label set~$[\nu]$. The image~$\surjectionAssoPara(\spine)$ of a maximal spine~$\spine$ on~$\pathG$ is then the \defn{canopy} of the binary tree~$\spine$. It can be alternatively defined as the boolean sequence~$(\lambda_1(\spine), \dots, \lambda_{\nu-1}(\spine))$, where~$\lambda_i(\spine)$ determines whether the node~$i$ of~$\spine$ for the infix labeling has or not a right child. See~\cite{Viennot}.
\end{example}

As observed in Remark~\ref{rem:extensionSurjection}, these surjections can be extended on the complete normal fans of~$\Perm(\ground) \subset \Asso(\tree) \subset \Para(\tree)$, not only on their maximal cones. Details are left to the reader.

\subsubsection{Common vertices}

It is combinatorially interesting to characterize the common vertices of the three polytopes~$\Perm(\ground) \subset \Asso(\tree) \subset \Para(\tree)$. We start with the common vertices of~$\Asso(\tree)$ and~$\Para(\tree)$.

\begin{proposition}
\label{prop:commonVerticesAssoPara}
The common vertices of the associahedron~$\Asso(\tree)$ with the parallelepiped~$\Para(\tree)$ are precisely the vertices~$\vertex{\spine}$ where~$\spine$ is any orientation of the tree~$\tree$ such that each negative vertex has outdegree at most one while each positive vertex has indegree at most one.
\end{proposition}

\begin{proof}
Since the facet defining inequalities of~$\Para(\tree)$ are also facet defining inequalities of~$\Asso(\tree)$, characterizing the common vertices of~$\Asso(\tree)$ and~$\Para(\tree)$ boils down to finding common normal cones of~$\Asso(\tree)$ and~$\Para(\tree)$. Since the normal cones of these two polytopes are braid cones, we can represent them by the Hasse diagram of their corresponding preposet (see Section~\ref{subsec:preposet}). We are therefore looking for directed trees on~$\ground$ which are simultaneously:
\begin{enumerate}[(i)]
\item spines on~$\tree$ in order to define a normal cone of the signed tree associahedron~$\Asso(\tree)$, and
\item orientations of the tree~$\tree$ in order to define a normal cone of the parallelepiped~$\Para(\tree)$.
\end{enumerate}
These directed trees are therefore orientations of~$\tree$ such that each negative vertex has outdegree at most one while each positive vertex has indegree at most one. Reciprocally, each such orientation is clearly a signed spine on~$\tree$, and therefore defines a normal cone of both~$\Asso(\tree)$ and~$\Para(\tree)$.
\end{proof}

\begin{example}[Unsigned tree, continued]
When~$\tree$ has only negative vertices, the tree associahedron~$\Asso(\tree)$ has precisely~$\nu$ common vertices with the parallelepiped~$\Para(\tree)$. The corresponding spines are obtained by orienting all edges of~$\tree$ towards a given vertex of~$\tree$.
\end{example}

\begin{question}
Do~$\Asso(\tree)$ and~$\Para(\tree)$ have at least~$\nu$ common vertices for any signed tree~$\tree$?
\end{question}

We now characterize common vertices of~$\Perm(\ground)$ and~$\Asso(\tree)$.

\begin{proposition}
\label{prop:commonVerticesAssoPerm}
Consider a maximal signed spine~$\spine$ on~$\tree$, and a linear order~$\prec$ with~${\sigma : [\nu] \to \ground}$ such that~${\sigma(1) \prec \sigma(2) \prec \dots \prec \sigma(\nu)}$. The following assertions are equivalent:
\begin{enumerate}[(i)]
\item the vertex~$\vertex{\spine}$ of~$\Asso(\tree)$ coincides with the vertex~$\b{p}(\sigma)$ of~$\Perm(\ground)$,
\item the cone~$\primalCone(\spine)$ of~$\Asso(\tree)$ coincides with the cone~$\primalCone(\prec)$ of~$\Perm(\ground)$,
\item the normal cone~$\normalCone(\spine)$ of~$\Asso(\tree)$ coincides with the normal cone~$\normalCone(\prec)$ of~$\Perm(\ground)$.
\item the fiber of~$\spine$ under the surjective map~$\surjectionPermAsso$ is the singleton~$\surjectionPermAsso^{-1}(\spine) = \{\prec\}$,
\item the spine~$\spine$ is a directed path whose $i$\ordinal{} vertex is labeled by~$\{\sigma(i)\}$,
\item the spine~$\spine$ is a directed path whose transitive closure is~$\prec$,
\item for any~$i,j \in [\nu]$, the sets~$\sigma([i])$ and~$\sigma([j])$ are nested source label sets of~$\spine$.
\end{enumerate}
Following~\cite{HohlwegLangeThomas}, we say that $\spine$ and~$\prec$ are \defn{$\tree$-singletons} if these conditions are fullfiled.
\end{proposition}

\begin{proof}
As in the proof of Proposition~\ref{prop:commonVerticesAssoPara}, we observe that characterizing the common vertices of~$\Asso(\tree)$ and~$\Perm(\ground)$ comes down to finding common normal cones of~$\Asso(\tree)$ and~$\Perm(\ground)$, since the facet defining inequalities of~$\Asso(\tree)$ are also facet defining inequalities of~$\Perm(\ground)$. In other words, we already obtain~(i)$\iff$(ii). By polarity, we get~(ii)$\iff$(iii). Moreover, the equivalences (ii)$\iff$(iv)$\iff$(v)$\iff$(vi) are immediate consequences of the discussion in Section~\ref{sec:spineFan}. Finally, a spine is a directed path iff all its source sets are nested, which gives~\mbox{(v)$\iff$(vii)}.
\end{proof}

\begin{corollary}
The three polytopes~$\Perm(\ground) \subset \Asso(\tree) \subset \Para(\tree)$ have two common vertices if~$\tree$ is a signed path and none otherwise.
\end{corollary}

\begin{proof}
The vertices common to the three polytopes correspond to the orientations on~$\tree$ which are directed paths. Such orientations only exist if~$\tree$ is itself a path.
\end{proof}

The following proposition ensures that the $\tree$-singletons of the signed tree associahedron~$\Asso(\tree)$ completely determine its facets, and thus its geometry. This property, already observed by C.~Hohlweg and C.~Lange for their path associahedra~\cite{HohlwegLange}, was the heart of the extension of their results to all Coxeter groups~\cite{HohlwegLangeThomas}.

\begin{proposition}
Any linear preposet on~$\ground$ whose Hasse diagram is a signed spine on~$\tree$ admits a total linear extension whose Hasse diagram is also a signed spine on~$\tree$. In other words, a face of~$\Asso(\tree)$ contains a $\tree$-singleton if and only if its corresponding signed spine is a directed path. In particular, any facet of~$\Asso(\tree)$ contains a $\tree$-singleton.
\end{proposition}

\begin{proof}
Consider a linear signed spine~$\spine$ on~$\tree$ with a node~$s$ labeled by a set~$U$ containing at least two elements. Assume that~$U^- \ne \varnothing$, let~$r$ denote the incoming arc of~$\spine$ at~$s$, and let~$u \in U^-$ be a vertex of~$U^-$ such that~$\source(r)$ belongs to a connected component of~$\tree \ssm U^-$ incident to~$u$. Applying the same transformation as in the proof of Lemma~\ref{lem:opening}, we obtain a new linear signed spine where the node~$s$ has been split into two nodes labeled by~$U \ssm \{u\}$ and~$\{u\}$ respectively. The proof is similar if~$U^- = \varnothing$ but~$U^+ \ne \varnothing$.
\end{proof}

From Proposition~\ref{prop:commonVerticesAssoPerm}, we can derive an inductive formula to count the number~$\singletons(\tree)$ of common vertices of the signed tree associahedron~$\Asso(\tree)$ and of the permutahedron~$\Perm(\ground)$. The idea is to consider which vertex~$v$ of~$\ground$ can be the first element in a $\tree$-singleton~$\prec$, and what order can~$\prec$ induce on~$\ground \ssm \{v\}$. To make this idea precise, we need to work with rooted phantom trees. Remember from Section~\ref{subsec:links} that a phantom tree is a tree where each vertex is either a standard vertex or a phantom. As already mentioned, the definition and properties of spines can be directly extended to these phantom trees, although we focused only on trees to simplify our presentation.

A \defn{rooted} phantom tree~$\phantomTree_\bullet$ is a phantom tree~$\phantomTree$ where we have marked a standard vertex, called~\defn{root} of~$\phantomTree_\bullet$. We call \defn{$\phantomTree_\bullet$-singletons} the $\phantomTree$-singletons~$\prec$ such that the root of~$\phantomTree_\bullet$ is the first element of~$\ground$ for~$\prec$. We let~$\singletons(\phantomTree_\bullet)$ be the number of~$\phantomTree_\bullet$-singletons. We denote by~$N(\phantomTree_\bullet)$ the set of standard vertices of~$\phantomTree_\bullet$ which are neighbors of its root when represented on the boundary of~$\phantomTree_\bullet \times [-1,1]$. For~$v \in N(\phantomTree_\bullet)$, we denote by~$\phantomTree_\bullet^v$ the rooted phantom tree where the root of~$\phantomTree_\bullet$ is turned to a phantom and the vertex~$v$ is marked as the new root. If we group $\tree$-singletons according to the position of their first vertex, we obtain the following formula for~$\singletons(\tree)$:
\[
\singletons(\tree) = \sum_{\bullet \, \in \ground} \singletons(\tree_\bullet).
\]
In fact, it is easy to see that this sum only runs over~$\ground^- \cup \leaves(\tree)^+$, where $\leaves(\tree)^+$ denotes the set of all positive leaves of~$\tree$, since the other positive vertices cannot appear as first vertices of a $\tree$-singleton. In turn, the $\phantomTree_\bullet$-singletons~$\prec$ can be decomposed according to the next vertex in~$\prec$ after their root:
\[
\singletons(\phantomTree_\bullet) = \sum_{v \in N(\phantomTree_\bullet)} \singletons(\phantomTree_\bullet^v).
\]
Similar formulas hold when we reverse the roles of~$-$ and~$+$ and consider the last vertex instead of the first. 

\begin{example}[Unsigned tree, continued]
If~$\tree$ has only negative signs or only positive signs, then the number of $\tree$-singletons satisfies the induction
\[
\singletons(\tree) = \sum_\ell \singletons(\tree\ominus\ell),
\]
where the sum ranges over the leaves~$\ell$ of~$\tree$, and $\tree\ominus\ell$ denotes the subtree of~$\tree$ obtained by erasing~$\ell$ and its incident edge. For example, if~$\pathG$ is a path with $\nu$ negative signs, then~$\singletons(\pathG) = 2^{\nu-1}$.
\end{example}

\begin{example}[Signed path, continued]
For a signed path~$\pathG$, more explicit formulas for the number~$\singletons(\pathG)$ of $\pathG$-singletons can be found in~\cite{LabbeLange}.
\end{example}

\subsubsection{Isometry classes}

We now characterize the isometry classes of signed tree associahedra. Our result is similar to that of~\cite{BergeronHohlwegLangeThomas} for the generalized associahedra of~\cite{HohlwegLangeThomas}.

We say that two signed trees~$\tree$ and~$\tree'$ on signed ground sets~$\ground$ and~$\ground'$ respectively are \defn{isomorphic} (resp.~\defn{anti-isomorphic}) iff there is a bijection~$\theta : \ground \to \ground'$ such that, for all~$u,v \in \ground$,
\begin{itemize}
\item ${u\!-\!v}$ is an edge in~$\tree$ iff ${\theta(u)\!-\!\theta(v)}$ is an edge in~$\tree'$, and
\item the signs of~$u$ and~$\theta(u)$ coincide (resp.~are opposite).
\end{itemize}
We say that~$\tree$ and~$\tree'$ are isomorphic (resp.~anti-isomorphic) \defn{up to the signs of their leaves} if Condition~(ii) already holds for internal vertices of~$\tree$.

\begin{proposition}
Let~$\tree$ and~$\tree'$ be two signed trees on signed ground sets~$\ground$ and~$\ground'$ respectively. Then the signed tree associahedra~$\Asso(\tree)$ and~$\Asso(\tree')$ are isometric if and only if the signed ground trees~$\tree$ and~$\tree'$ are isomorphic or anti-isomorphic, up to the signs of their leaves.
\end{proposition}

\begin{proof}
If~$\theta : \ground \to \ground'$ induces an isomorphism between the sign trees~$\tree$ and~$\tree'$, up to the sign of their leaves, then the transformation~$\Theta : \HH \to \HH$ defined by~$\Theta \big( \sum_{v \in \ground} \alpha_v \, e_v \big) \eqdef \sum_{v \in \ground} \alpha_v \, e_{\theta(v)}$ is an isometry from~$\Asso(\tree)$ to~$\Asso(\tree')$. Moreover, if~$\tree'$ is obtained from~$\tree$ by reversing all its signs, then~$\Asso(\tree')$ is obtained from~$\Asso(\tree)$ by central symmetry around the origin. The if direction of the statement follows from the combination of these two observations.

Reciprocally, assume that there is an isometry~$\Theta$ from~$\Asso(\tree)$ to~$\Asso(\tree')$. Since~$\Theta$ sends parallel facets to parallel facets, it sends~$\Para(\tree)$ to~$\Para(\tree')$ and therefore also~$\Perm(\ground)$ to~$\Perm(\ground')$. In particular,~$\Theta$ preserves the center~$\origin \eqdef \frac{\nu + 1}{2}\cdot \one$ of the permutahedron~$\Perm(\ground)$ and of the parallelepiped~$\Para(\tree)$. For any subset~$\varnothing \ne U \subseteq \ground$ of cardinality~$u \eqdef |U|$, the distance from~$\origin$ to the hyperplane~$\Hyp(U)$~is
\[
d \big( \origin, \Hyp(U) \big) = \frac{\nu \, u\, (\nu - u)}{\sqrt{u^2 + (\nu - u)^2}}.
\]
Observe that the function
\[
x \longmapsto \frac{x \, (1-x)}{\sqrt{x^2 + (1-x)^2}}
\]
is bijective on~$\big[ 0, \frac{1}{2} \big]$. It follows that the isometry~$\Theta$ sends the hyperplane~$\Hyp(U)$, for~${\varnothing \ne U \subseteq \ground}$, to an hyperplane~$\Hyp(U')$ for some~$\varnothing \ne U' \subseteq \ground'$ with~$|U'| = |U|$ or~$|U'| = \nu - |U|$. Observe moreover that the facets of~$\Perm(\ground)$ defined by the hyperplanes~$\Hyp(\{v\})$ for~$v \in \ground$ are pairwise non-adjacent, but that the facet defined by~$\Hyp(\{v\})$ is adjacent to all facets defined by~${\Hyp(\ground \ssm \{w\})}$ for~$w \in \ground \ssm \{v\}$. Therefore, the hyperplanes~$\Hyp(\{v\})$ for~$v \in \ground$ are either all sent to the hyperplanes~$\Hyp(\{v'\})$ for~$v' \in \ground'$, or all sent to the hyperplanes~$\Hyp(\ground' \ssm \{v'\})$ for~$v' \in \ground'$.

Assume first that we are in the former situation. Define a map~$\theta : \ground \to \ground'$ such that the hyperplane~$\Theta \big( \Hyp(\{v\}) \big) = \Hyp(\{\theta(v)\})$ for any~$v \in \ground$. It follows that~$\Theta(e_v) = e_{\theta(v)}$. Since~$\Theta$ is a linear map, it sends the characteristic vector of any subset~$\varnothing \ne U \subseteq \ground$ to the characteristic vector of~$\theta(U)$. Thus, the map~$\theta$ defines an isomorphism from the signed building set~$\building(\tree)$ to the signed building set~$\building(\tree')$. Corollary~\ref{coro:buildingSetToTree} therefore ensures that~$\theta$ is an isomorphism between the sign trees~$\tree$ and~$\tree'$, up to the signs of their leaves.

Finally, if we were in the latter situation above, then we just change all signs of~$\tree'$. It composes~$\Theta$ by a central symmetry of the space and thus places us back in the situation treated above.
\end{proof}

\subsubsection{Vertex barycenters}

To conclude this section, we briefly mention the position of the vertex barycenter~$\bary(\tree)$ of the signed tree associahedron~$\Asso(\tree)$. This point clearly lies on the linear span of the characteristic vectors of the orbits of~$\ground$ under the automorphism group of~$\tree$. For the path associahedra of C.~Hohlweg and C.~Lange~\cite{HohlwegLange}, this barycenter always coincides with the center~$\origin = \frac{\nu + 1}{2} \cdot \one$ of the permutahedron~$\Perm(\ground)$. Refer to~\cite{HohlwegLortieRaymond, PilaudStump-barycenter, LangePilaud-spines} for three different proofs. However, this property is no longer true in general. For example, the vertex barycenters of~$\Asso(\tripodWhite)$ and~$\Asso(\tripodBlack)$ are
\[
\bary(\tripodWhite) = \bigg( \frac{41}{16}, \frac{37}{16}, \frac{41}{16}, \frac{41}{16}\bigg)
\qquad \text{and}\qquad
\bary(\tripodBlack) = \bigg( \frac{39}{16}, \frac{43}{16}, \frac{39}{16}, \frac{39}{16}\bigg).
\]


\section{Increasing flips}
\label{sec:increasingFlips}

In this section, we consider natural edge orientations of the flip graph~$\spineFlipGraph(\tree)$ arising from linear orders on~$\ground$. The resulting directed graphs are acyclic and have a unique source and sink. Their transitive closures are called increasing flip posets. Contrarily to the situation of the classical associahedron, we show that these posets are not quotients of the weak order by order congruences as soon as the ground tree~$\tree$ is not a path. Finally, increasing flips provide shelling orders on the nested complex~$\nestedComplex(\tree)$, from which we can derive a combinatorial description of its $h$-vector.

\subsection{Increasing flip graphs}

Before defining increasing flips and their corresponding posets, let us briefly recall some well-known geometric properties of weak orders. Let~$\linear(\ground)$ be the set of all linear orders on~$\ground$. Fix one of them~$\prec_\circ \; \in \linear(\ground)$, and denote by~$\prec_\bullet$ its reverse order. A \mbox{\defn{$\prec_\circ$-inversion}} of a linear order~$\prec \; \in \linear(\ground)$ is an ordered pair~$(u,v) \in \ground^2$ such that~$u \prec_\circ v$ while~$v \prec u$. The \defn{$\prec_\circ$-inversion set} of~$\prec$ is the set~$\inv_{\prec_\circ}(\prec)$ of all $\prec_\circ$-inversions of~$\prec$. The \defn{$\prec_\circ$-weak order} is the partial order~$\le$ on~$\linear(\ground)$ defined as the inclusion order on $\prec_\circ$-inversion sets, that is
\[
\prec \; \le \; \prec' \text{ in } \prec_\circ\!\text{-weak order} \quad \iff \quad \inv_{\prec_\circ}(\prec) \; \subseteq \; \inv_{\prec_\circ}(\prec').
\]
This order defines a lattice structure on~$\linear(\ground)$ with minimal element~$\prec_\circ$ and maximal element~$\prec_\bullet$. 

The $\prec_\circ$-weak order can also be interpreted in geometrical terms. Namely, consider the bijections ${\sigma_\circ, \sigma_\bullet : [\nu] \to \ground}$ such that~${\sigma_\circ(1) \prec_\circ \dots \prec_\circ \sigma_\circ(\nu)}$ and~${\sigma_\bullet(1) \prec_\bullet \dots \prec_\bullet \sigma_\bullet(\nu)}$. Since~$\prec_\circ$ and~$\prec_\bullet$ are reverse to each other, we have~$\sigma_\bullet(\nu + 1 - i) = \sigma_\circ(i)$ and the vertices $\b{p}(\sigma_\circ)$ and~$\b{p}(\sigma_\bullet)$ are opposite vertices of the permutahedron~$\Perm(\ground)$. Let
\[
\b{g}_{\prec_\circ} \eqdef \b{p}(\sigma_\circ) - \b{p}(\sigma_\bullet) = \sum_{i \in [\nu]} (2i - \nu - 1) \, e_{\sigma_\circ(i)}
\]
denote the vector joining these two vertices. Let~$\prec$ and~$\prec'$ be two adjacent linear orders on~$\ground$, and let~$u,v \in \ground$ be the two elements such that $u \prec v$ while~$v \prec' u$. Then~$\prec \; \le \; \prec'$ in $\prec_\circ$-weak order when the following equivalent assertions hold:
\[
u \prec_\circ v \quad \iff \quad \dotprod{\b{g}_{\prec_\circ}}{\b{p}(\prec') - \b{p}(\prec)} > 0 \quad \iff \quad \normalCone(\prec_\circ) \subset \set{\b{x} \in \HH}{x_u \le x_v}.
\]
This implies that the Hasse diagram of the $\prec_\circ$-weak order is isomorphic to the $1$-skeleton of the permutahedron~$\Perm(\ground)$ oriented in the direction of~$\b{g}_{\prec_\circ}$.

\medskip
We now consider similar orientations on the spine flip graph~$\spineFlipGraph(\tree)$. The following lemma is a direct consequence of the discussion of Section~\ref{sec:spineFan}.

\begin{lemma}
\label{lem:increasingFlip}
Let~$\spine$ and~$\spine'$ be two adjacent maximal spines on~$\tree$ such that~$\spine'$ is obtained from~$\spine$ by flipping the arc from~$u$ to~$v$. Then
\[
u \prec_\circ v \quad \iff \quad \dotprod{\b{g}_{\prec_\circ}}{\vertex{\spine'} - \vertex{\spine}} > 0 \quad \iff \quad \normalCone(\prec_\circ) \subset \set{\b{x} \in \HH}{x_u \le x_v}.
\]
\end{lemma}

\begin{definition}
We say that the flip from~$\spine$ to~$\spine'$ is \defn{$\prec_\circ$-increasing} if the equivalent conditions of Lemma~\ref{lem:increasingFlip} are fulfilled. The \defn{$\prec_\circ$-increasing flip graph} is the directed graph~$\spineFlipGraph_{\prec_\circ}(\tree)$ whose vertices correspond to the maximal signed spines on~$\tree$ and whose arcs correspond to $\prec_\circ$-increasing flips between them.
\end{definition}

By Lemma~\ref{lem:increasingFlip}, the $\prec_\circ$-increasing flip graph~$\spineFlipGraph_{\prec_\circ}(\tree)$ is isomorphic to the $1$-skeleton of the associahedron~$\Asso(\tree)$ oriented in the direction of~$\b{g}_{\prec_\circ}$. It therefore already yields the following corollaries and definition.

\begin{corollary}
The $\prec_\circ$-increasing flip graph~$\spineFlipGraph_{\prec_\circ}(\tree)$ is isomorphic to the quotient of the $1$-skeleton of the permutahedron~$\Perm(\ground)$ oriented in the direction~$\b{g}_{\prec_\circ}$ by the fibers of the surjection map~$\surjectionPermAsso$.
\end{corollary}

\begin{corollary}
The $\prec_\circ$-increasing flip graph~$\spineFlipGraph_{\prec_\circ}(\tree)$ is a directed acyclic graph, with a unique source~$\surjectionPermAsso(\prec_\circ)$ and a unique sink~$\surjectionPermAsso(\prec_\bullet)$.
\end{corollary}

\begin{definition}
The partial order on maximal spines on~$\tree$ defined as the transitive closure of the $\prec_\circ$-increasing flip graph~$\spineFlipGraph_{\prec_\circ}(\tree)$ is called \defn{$(\tree, \prec_\circ)$-increasing flip order}.
\end{definition}

\begin{example}[Signed path, continued]
If~$\pathG$ is a signed path and~$\prec_\circ, \prec_\bullet$ are the two opposite linear orders along that path, then the $(\pathG, \prec_\circ)$-increasing flip order is a Cambrian lattice of type~$A$, defined by N.~Reading in~\cite{Reading-CambrianLattices}. In particular, if the path~$\pathG$ has only negative signs (or only positive ones), then the resulting lattice is the Tamari lattice~\cite{TamariFestschrift}. 
\end{example}


\subsection{Lattice congruences}

Cambrian lattices are lattice quotients of the weak order by the fibers of the surjection~$\surjectionPermAsso$. This means that the fibers of~$\surjectionPermAsso$ are the congruence classes of a lattice congruence on the weak order. A \defn{lattice congruence} on a finite lattice~$L$ is an equivalent relation~$\equiv$ on~$L$ compatible with joins~$\vee$ and meets~$\wedge$ of~$L$: if $x \equiv x'$ and~$y \equiv y'$, then $x \vee y \equiv x' \vee y'$ and~$x \wedge y \equiv x' \wedge y'$. Lattice congruences are particular examples of order congruences: an \defn{order congruence} on a poset~$P$ is an equivalence relation~$\equiv$ on~$P$ such that
\begin{enumerate}[(i)]
\item Every equivalence class under~$\equiv$ is an interval of~$P$.
\item The projection~$\projDown : P \to P$, which maps an element of~$P$ to the minimal element of its equivalence class, is order preserving.
\item The projection~$\projUp : P \to P$, which maps an element of~$P$ to the maximal element of its equivalence class, is order preserving.
\end{enumerate}
The \defn{quotient}~${P/\!\equiv}$ is a poset on the equivalence classes of~$\equiv$. The order relation is defined by~$X \le Y$ in~${P/\!\equiv}$ iff there exists $x \in X$ and~$y \in Y$ such that~$x \le y$ in~$P$. It is isomorphic to the subposet of~$P$ induced by~$\projDown(P)$ (or equivalently by~$\projUp(P)$). We refer to~\cite{Reading-LatticeCongruences, Reading-CambrianLattices} for further details on lattice and poset congruences. In the situation of signed tree associahedra, we show the following negative result.

\begin{proposition}
\label{prop:noCambrian}
If~$\tree$ is not a path, then the fibers of the surjection map~$\surjectionPermAsso$ are never the congruence classes of an order congruence of the $\prec_\circ$-weak order, no matter the choice of~$\prec_\circ$.
\end{proposition}

\begin{proof}
Let~$\tree$ be a tree on~$\ground$ which is not a path, and~$\prec_\circ$ be any linear order on~$\ground$. Consider a vertex~$u$ of~$\tree$ with at least three neighbors. Let~$v_1, \dots, v_k$ denote the neighbors of~$u$ in~$\tree$ and~$V_1, \dots, V_k$ the vertex sets of the connected components of~$\tree \ssm \{u\}$ such that~$v_i \in V_i$, and let~$V_i' \eqdef V_i \ssm \{v_i\}$. Without loss of generality, we can assume that~$v_1 \prec_\circ v_2 \prec_\circ v_3$ (up to relabeling) and that~$u$ has negative sign (up to reversing all signs of~$\tree$). We now distinguish two cases.

\smallskip
\paragraph{\bf First case} When~$V_2' \ne \varnothing$, we show that the fibers of~$\surjectionPermAsso$ are not all intervals in $\prec_\circ$-weak order. Since~$V_2' \ne \varnothing$, consider a neighbor~$w$ of~$v_2$ in~$V_2'$, let~$W$ be the connected component of~$\tree \ssm \{v_2\}$ containing~$w$, and define~$W' \eqdef W \ssm \{w\}$ and~$V_2'' \eqdef V_2' \ssm W$. Assume here that~$v_2$ has a negative sign, the other case being similar. All these notations are summarized in \fref{fig:noCambrian1}\,(left), where we have colored in grey the vertices for which we do not need to know the sign. Note that both partially directed trees of \fref{fig:noCambrian1} represent spines on~$\tree$: each undirected edge should point towards (resp.~from) its incident grey vertex if this vertex is negative (resp.~positive). 

\begin{figure}[h]
  \capstart
  \centerline{\includegraphics[scale=1]{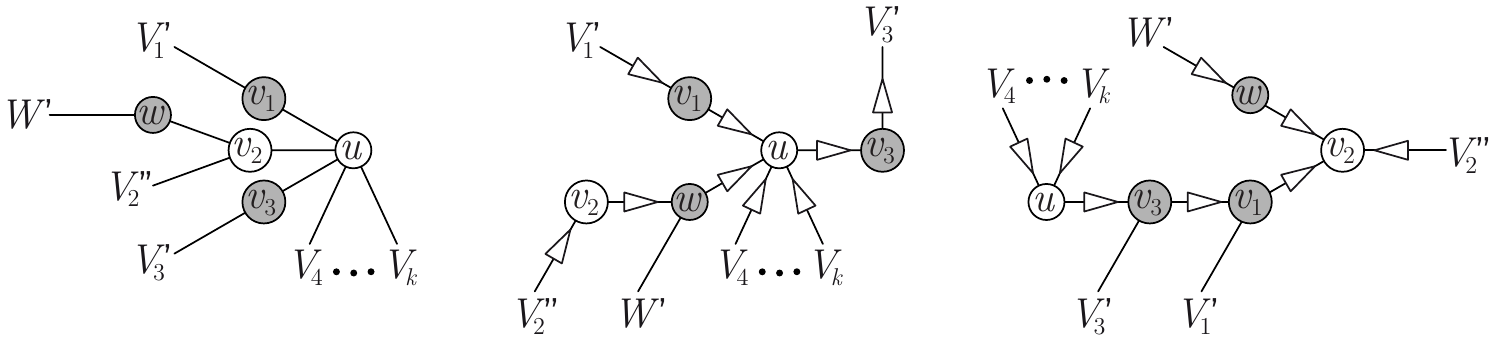}}
  \caption{First case: $V_2 \ne \varnothing$. The ground tree~$\tree$ (left) and two of its spines (middle and right).}
  \label{fig:noCambrian1}
\end{figure}

We now distinguish three subcases.
\begin{description}
\item[if $w \prec_\circ v_1$] Consider the middle spine of \fref{fig:noCambrian1}. There clearly exist two linear extensions~$\prec$ and~$\prec'$ of this spine for which $v_1 \prec v_2 \prec w$ and $v_2 \prec' w \prec' v_1$. In other words, $(v_1,v_2) \notin \inv_{\prec_\circ}(\prec)$ and $(w,v_1) \notin \inv_{\prec_\circ}(\prec')$. Therefore, neither $(v_1,v_2)$ nor~$(w,v_1)$ are in the inversion set of the meet~$\prec''$ of~$\prec$ and~$\prec'$, so that~$w \prec'' v_2$. It follows that~$\prec''$ is not a linear extension of the leftmost spine of \fref{fig:noCambrian1}, so that this fiber of~$\surjectionPermAsso$ is not an interval in $\prec_\circ$-weak order.
\item[if $v_3 \prec_\circ w$] We argue similarly, exchanging the role of~$v_1$ and~$v_3$.
\item[if $v_1 \prec_\circ w \prec_\circ v_3$] Consider the rightmost spine of \fref{fig:noCambrian1}. There clearly exist two linear extensions~$\prec$ and~$\prec'$ of this spine for which $v_3 \prec v_1 \prec w$ and $w \prec' v_3 \prec' v_1$. In other words, $(v_1,w) \notin \inv_{\prec_\circ}(\prec)$ and $(w,v_3) \notin \inv_{\prec_\circ}(\prec')$. Therefore, neither $(v_1,w)$ nor~$(w,v_3)$ are in the inversion set of the meet~$\prec''$ of~$\prec$ and~$\prec'$, so that~$v_1 \prec'' v_3$. It follows that~$\prec''$ is not a linear extension of the rightmost spine of \fref{fig:noCambrian1}, so that this fiber of~$\surjectionPermAsso$ is not an interval in $\prec_\circ$-weak order.
\end{description}
Note that there exists ground trees~$\tree$ for which the fibers of~$\surjectionPermAsso$ are not all intervals of the $\prec_\circ$-weak order, no matter the choice of~$\prec_\circ$. See Example~\ref{exm:doubleTripod} below.

\begin{figure}
  \capstart
  \centerline{\includegraphics[scale=1]{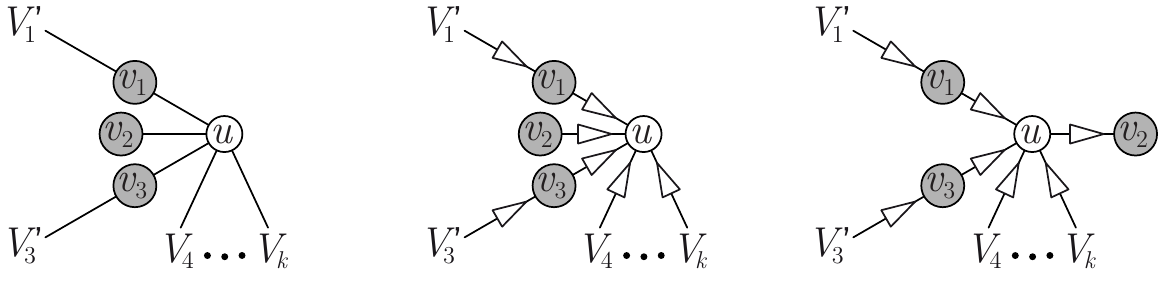}}
  \caption{Second case: $V_2 = \varnothing$. The ground tree~$\tree$ (left) and two of its spines (middle and right).}
  \label{fig:noCambrian2}
\end{figure}

\medskip
\paragraph{\bf Second case} When~$V_2' = \varnothing$, the fibers of~$\surjectionPermAsso$ may or may not be all intervals. If they are not intervals, we have nothing to prove. If they are intervals, we prove that the projections~$\projDown$ and~$\projUp$ cannot be both order-preserving. We make here the assumption that~$u \prec_\circ v_2$, the proof in the other case being symmetric (to be precise, see the last sentence of the proof). The two trees represented in \fref{fig:noCambrian2} are spines on~$\tree$. Let~$\prec_c$, $\prec_y$, and~$\prec_z$ be three linear orders on~$\ground$ such that
\begin{gather*}
V_1' \prec_x V_3' \prec_x V_4 \prec_x \dots \prec_x V_k \prec_x v_1 \prec_x v_2 \prec_x v_3 \prec_x u, \\
V_1' \prec_y V_3' \prec_y V_4 \prec_y \dots \prec_y V_k \prec_y v_1 \prec_y v_3 \prec_y v_2 \prec_y u, \\
V_1' \prec_z V_3' \prec_z V_4 \prec_z \dots \prec_z V_k \prec_z v_1 \prec_z v_3 \prec_z u \prec_z v_2,
\end{gather*}
where the sets~$V_i'$ and~$V_j$ are arbitrarily ordered, but similarly in the three linear orders. The orders~$\prec_y$ and~$\prec_z$ are $\tree$-congruent, so that they have the same image by~$\surjectionPermAsso$ according to Lemma~\ref{lem:surjection}. Moreover, $\surjectionPermAsso(\prec_y) = \surjectionPermAsso(\prec_z)$ refines the spine of \fref{fig:noCambrian2}\,(middle) while $\surjectionPermAsso(\prec_x)$ refines the spine of \fref{fig:noCambrian2}\,(right). Observe now that~$\projDown$ is not order-preserving since $\prec_z \; \le \; \prec_y$ but $\projDown(\prec_z) \not\le \projDown(\prec_y)$ in~$\prec_\circ$-weak order. The former relation comes from our assumption that~$u \prec_\circ v_2$. To see the latter, observe that the pair~$(v_3,v_2)$ is not in the inversion set of~$\prec_x$ and therefore not in the inversion set of the projection $\projDown(\prec_x) = \projDown(\prec_y)$. In contrast, $(v_3,v_2)$ is in the inversion set of~$\projDown(\prec_z)$ since~$v_3 \prec v_2$ for any~$\prec$ such that~$\surjectionPermAsso(\prec)$ refines the spine of \fref{fig:noCambrian2}\,(right). If~$v_2 \prec_\circ u$, we would have done the same proof exchanging~$v_1$ and~$v_3$ and considering~$\projUp$ instead of~$\projDown$.
\end{proof}

\begin{example}[Double tripod]
\label{exm:doubleTripod}
Consider the subdivided tripod, obtained by subdividing once each edge of the tripod~$\tripodWhite$ of Example~\ref{exm:tripod}. Fix a linear order~$\prec_\circ$ on~$\ground$, let~$u$ be the central vertex of the double tripod, and $v_1, v_2, v_3$ be the neighbors of~$u$ such that~$v_1 \prec_\circ v_2 \prec_\circ v_3$. Then the vertex~$v_2$ has another neighbor than~$u$. It thus follows from the first case of the previous proof that the fibers of~$\surjectionPermAsso$ are not all intervals in the $\prec_\circ$-weak order, no matter the choice of~$\prec_\circ$. 
\end{example}


\subsection{$\tree$-Narayana numbers}

The $\prec_\circ$-increasing flip graph~$\spineFlipGraph_{\prec_\circ}(\tree)$ defined above also carries on combinatorial information concerning the $f$- and $h$-vectors of the signed tree associahedron~$\Asso(\tree)$.

The \defn{$f$-vector} of a polytope~$P \subset \R^d$ is the vector~$f(P) \eqdef (f_0(P), f_1(P), \dots, f_d(P))$ whose $k$\ordinal{} coordinate~$f_k(P)$ is the number of $k$-dimensional faces of~$P$. For example, $f_k(\Perm(\ground))$ is the number of ordered partitions of~$[\nu]$ with~$\nu-k$ parts, while~$f_k(\Asso_n)$ is the number of dissections of a convex $(n+3)$-gon into~$n+1-k$ cells. The \defn{$f$-polynomial} of~$P$ is the polynomial~$f_P(X) \eqdef \sum_{k = 0}^d f_k(P) \, X^k$.

The \defn{$h$-vector}~$h(P) \eqdef (h_0(P), h_1(P), \dots, h_d(P))$ and the \defn{$h$-polynomial}~$h_P(X) \eqdef \sum_{\ell = 0}^d h_\ell(P) \, X^\ell$ of a \emph{simple} polytope~$P$ are defined by the relation
\[
f_P(X) = h_P(X+1),
\]
or equivalently by the equalities
\[
f_k(P) = \sum_{\ell=0}^d \binom{\ell}{k} h_\ell(P) \quad \text{for } 0 \le k \le d
\quad\text{or}\quad
h_\ell(P) = \sum_{k=0}^d (-1)^{k+\ell} \binom{k}{\ell} f_k(P) \quad \text{for } 0 \le \ell \le d.
\]
For example, $h_\ell(\Perm(\ground))$ is the \defn{Eulerian number}, that is, the number of permutations of~$[\nu]$ with $\ell$ descents, while~$h_\ell(\Asso_n)$ is the Narayana number~\cite{Narayana}
\[
h_\ell(\Asso_n) = \Narayana(n,\ell) \eqdef \frac{1}{n} \binom{n}{\ell} \binom{n}{\ell-1}.
\]

Given any simple polytope~$P \subset \R^d$ and any generic linear functional~$\lambda \in (\R^d)^*$, it is known that the $\ell$\ordinal{} entry~$h_\ell(P)$ of the $h$-vector of~$P$ equals the number of vertices of out-degree~$\ell$ in the $1$-skeleton of~$P$ oriented by increasing values of~$\lambda$. In particular, this ensures that the \mbox{$h$-vector} is symmetric (since it gives the same vector for the functionals~$\lambda$ and~$-\lambda$). We apply this characterization of $h$-vectors to compute~$h(\Asso(\tree))$ using the orientation~$\b{g}_\prec$ defined in the previous section. We say that an arc of a maximal spine~$\spine$ is \defn{ordered} if its source is smaller than its target, and \defn{inverted} otherwise.

\begin{proposition}
\label{prop:Narayana}
The $\ell$\ordinal{} coordinate of $h$-vector of~$\Asso(\tree)$ is the number of maximal spines on~$\tree$ with~$\ell$ ordered arcs.
\end{proposition}

\begin{proof}
Direct consequence of the above-mentioned characterization of the $h$-vectors entries, together with Lemma~\ref{lem:increasingFlip}.
\end{proof}

We hope that this interpretation can be used to settle the following question.

\begin{question}
\label{qu:fVector}
Although not necessarily isomorphic (see Section~\ref{subsec:isomorphisms}), do signed nested complexes for different signatures on the same underlying tree have the same $f$- and $h$-vectors?
\end{question}

The first relevant example is the ground tree with four leaves and two vertices of degree~$3$. Up to the operations of Proposition~\ref{prop:nestedComplexPreservingOperations}, it has two possible signatures: the signs on the two vertices of degree~$3$ can either differ or coincide. The signed nested complexes associated to these two signatures are not isomorphic, but their $f$-vector is $(1,27,182,478,535,214)$ for both signatures.

To motivate further Question~\ref{qu:fVector}, we observe that we already know that the number~$f_0$ of vertices of~$\nestedComplex(\tree)$ is independent of the signature on~$\tree$. Indeed, $f_0$ counts the relevant open subtrees of~$\tree$. But the definition of open subtrees of~$\tree$ does not depend on the signature, and the number of irrelevant open subtrees is always~$\nu+2$ (there are~$|\ground^-| + 1$ negative ones and~$|\ground^+| + 1$ positive~ones).


\section{Minkowski decomposition}
\label{sec:MinkowskiDecomposition}

In this section, we study the decomposition of the signed tree associahedron~$\Asso(\tree)$ as a Minkowski sum and difference of dilated faces of the standard simplex. This decomposition exists by a general result of F.~Ardila, C.~Benedetti, and J.~Doker~\cite{ArdilaBenedettiDoker}, since the polytope~$\Asso(\tree)$ is a generalized permutahedron as defined by A.~Postnikov in~\cite{Postnikov}. The coefficients in this decomposition can be computed by M\"obius inversion of the right hand sides of the supporting hyperplanes of~$\Asso(\tree)$ normal to the characteristic vectors of proper non-empty subsets of~$\ground$. Following the  line of research initiated by C.~Lange in~\cite{Lange} for the associahedra of~\cite{HohlwegLange}, we generalize the approach of~\cite{LangePilaud-spines} to give direct combinatorial formulas for these coefficients, thus avoiding the exponential cost of M\"obius inversion. We first briefly review the general theory of Minkowski decompositions of generalized permutahedra before presenting explicit coefficients for the signed tree associahedra.

\subsection{Deformed permutahedra, Minkowski decompositions, and M\"obius inversion}

The following class of polytopes, intimately related to the braid arrangement and to the permutahedron~$\Perm(\ground)$, was defined in~\cite{Postnikov} and studied in~\cite{ArdilaBenedettiDoker, PostnikovReinerWilliams}. 

\begin{definition}[\cite{Postnikov, PostnikovReinerWilliams}]
\label{def:deformedPerm}
A \defn{deformed permutahedron}\footnote{Although A.~Postnikov called them \emph{generalized permutahedra}, we prefer to use the term \emph{deformed} to distinguish from the other natural generalization of the permutahedra to finite Coxeter groups.} is a polytope defined as
\[
\Defo \big( \{z_U\}_{U \subseteq \ground} \big) \eqdef \biggset{\b{x} \in \R^\ground}{\sum_{v \in \ground} x_v = z_{\ground} \text{ and } \sum_{u \in U} x_u \ge z_U \text{ for } \varnothing \ne U \subsetneq \ground}
\]
for a family $\{z_U\}_{U \subseteq \ground} \in \R^{2^\ground}$ such that $z_\ground = \binom{\nu + 1}{2}$ and~$z_U + z_V \le z_{U \cup V} + z_{U \cap V}$ for all $U, V \subseteq \ground$.
\end{definition}

In other words, a deformed permutahedron is obtained as a deformation of the classical permutahedron~$\Perm(\ground)$ by moving its facets while keeping the direction of their normal vectors and staying in its deformation cone~\cite{PostnikovReinerWilliams}. As already mentioned in the introduction, the deformed permutahedra are precisely the polytopes whose normal fans coarsen that of~$\Perm(\ground)$. This family of polytopes contains many relevant families of combinatorial polytopes: permutahedra, associahedra, cyclohedra, and more generally all graph-associahedra~\cite{CarrDevadoss} and nestohedra~\cite{Postnikov, Zelevinsky}.

Consider two polytopes~$P$ and~$Q$ in~$\R^\ground$. The \defn{Minkowski sum} of~$P$ and~$Q$ is the polytope $P + Q \eqdef \set{p + q}{p \in P, q \in Q} \subset \R^\ground$, and its normal fan is the common refinement of the normal fans of~$P$ and~$Q$. If there exists a polytope~$R \subset \R^\ground$ such that~$P = Q + R$, then we call~$R$ the \defn{Minkowski difference} of~$P$ and~$Q$ and denote it by~$P-Q$. Note that the Minkowski difference of~$P$ and~$Q$ only exists when~$Q$ is a Minkowski summand of~$P$. Since
\[
\Defo(\{z_U\}) + \Defo(\{z'_U\}) = \Defo(\{z_U + z'_U\}),
\]
the class of deformed associahedra is closed by Minkowski sum and difference.

For any~$\varnothing \ne V \subseteq \ground$, we consider the face~$\simplex_V \eqdef \conv\set{e_v}{v \in V}$ of the standard simplex $\simplex_{\ground}$. For any~$\{y_V\}_{V \subseteq \ground} \in \R^{2^\ground}$, if the Minkowski sum and difference
\[
\Mink \big( \{y_V\}_{V \subseteq \ground} \big) \eqdef \sum_{\varnothing \ne V \subseteq \ground} y_V \cdot \simplex_V
\]
is well-defined, then it is a deformed permutahedron. Reciprocally, the following statement, due to F.~Ardila, C.~Benedetti, and J.~Doker~\cite{ArdilaBenedettiDoker}, affirms that any deformed permutahedron $\Defo(\{z_U\})$ can be decomposed into a Minkowski sum and difference $\Mink(\{y_V\})$, and that $\{y_V\}_{V \subseteq \ground}$ is derived from $\{z_U\}_{U \subseteq \ground}$ by M\"obius inversion when all the inequalities defining $\Defo(\{z_U\})$ are tight.

\begin{proposition}[\cite{Postnikov, ArdilaBenedettiDoker}]
\label{prop:minkowskiSumDiff}
Every deformed permutahedron can be written uniquely as a Minkowski sum and difference of faces of the standard simplex:
\[
\forall \{z_U\}_{U \subseteq \ground} \in \R^{2^\ground}, \quad \exists \{y_V\}_{V \subseteq \ground} \in \R^{2^\ground}, \qquad \Defo \big( \{z_U\}_{U \subseteq \ground} \big) = \Mink \big( \{y_V\}_{V \subseteq \ground} \big).
\]
Moreover, if all inequalities $\sum_{u \in U} x_u \ge z_U$ are tight, the coefficients~$\{y_V\}$ and~$\{z_U\}$ are connected by M\"obius inversion
\[
z_U = \sum_{V \subseteq U} y_V
\qquad\text{and}\qquad
y_V = \sum_{\varnothing \ne U \subseteq V} (-1)^{|V \ssm U|} z_U.
\]
\end{proposition}

\begin{example}
The classical permutahedron~$\Perm(\ground)$ can be written as
\[
\Perm(\ground) = \Defo \bigg( \bigg\{ \binom{|U|+1}{2} \bigg\}_{U \subseteq \ground} \bigg) = \Mink \big( \{\one_{|V| \le 2}\}_{V \subseteq \ground} \big),
\]
where~$\one_{|V| \le 2} = 1$ if~$|V| \le 2$ and~$0$ otherwise. Up to the translation of vector~$\sum_{v \in \ground} e_v$, it is therefore the Minkowski sum of all segments~$[e_u,e_v]$ for~$u \ne v \in \ground$, as already mentioned in Section~\ref{sec:furtherGeomProp}.
\end{example}

\begin{example}
\label{exm:MinkowskiDecompositionPara}
Consider the parallelepiped~$\Para(\tree)$ formed by the parallel facets of~$\Asso(\tree)$ studied in Section~\ref{sec:furtherGeomProp}. We have seen in Proposition~\ref{prop:paraMinkowskiSum} that~$\Para(\tree)$ is a translate of the Minkowski sum
\[
\Mink \big( \{\pi(e)\}_{e \in \tree} \big) = \sum_{e \text{ edge of } \tree} \pi(e) \cdot \Delta_e,
\]
where~$\pi(e)$ denotes the number of paths in~$\tree$ containing the edge~$e$. To compute the precise translation from~$\Para(\tree)$ to~$\Mink \big( \{\pi(e)\}_{e \in \tree} \big)$, we observe that the barycenter of~$\Para(\tree)$ is the barycenter~$\origin \eqdef \frac{\nu + 1}{2} \cdot \one$ of the permutahedron~$\Perm(\ground)$, while the barycenter of~$\Mink \big( \{\pi(e)\}_{e \in \tree} \big)$ is given by
\[
\sum_{u \in \ground}\sum_{u - v \text{ in } \tree} \frac{\pi({u\!-\!v})}{2} \cdot e_u.
\]
From this observation, we derive that~$\Para(\tree) = \Defo \big( \{z_U\}_{U \subseteq \ground} \big) = \Mink \big( \{y_V\}_{V \subseteq \ground} \big)$, where
\begin{align*}
z_U & = \frac{|U|(\nu + 1)}{2} - \sum_{u \in U} \sum_{\substack{v \notin U \\ u - v \text{ in } \tree}} \frac{\pi({u\!-\!v})}{2} \\
\qquad\text{and}\qquad
y_V & = \begin{cases}
	\mathlarger{\frac{\nu + 1}{2} - \sum\limits_{u-v \text{ in } \tree} \frac{\pi({u\!-\!v})}{2}} & \text{if } V = \{u\} \\
	\pi(e) & \text{if } V = e \text{ is an edge of } \tree \\
	0 & \text{otherwise.}
\end{cases}
\end{align*}
We invite the reader to check on these formulas that~$z_U = \sum_{V \subseteq U} y_V$. Moreover, observe that for any edge~$e$ of~$\tree$, if~$U$ and~$\ground \ssm U$ denote the connected components of~$\tree \ssm \{e\}$, then the corresponding facet defining inequalities of~$\Para(\tree)$ are indeed facet defining inequalities for~$\Perm(\ground)$ since
\[
z_U = \frac{|U|(\nu + 1)}{2} - \frac{|U|(\nu - |U|)}{2} = \binom{|U| + 1}{2}
\qquad\text{and similarly}\qquad
z_{\ground \ssm U} = \binom{\nu - |U| + 1}{2}.
\]
\end{example}

We now focus our attention on the signed tree associahedron~$\Asso(\tree)$. Since its normal fan coarsens the braid fan~$\braidFan(\ground)$, it is a deformed permutahedron. Let~$\{z_U(\tree)\}_{U \subseteq \ground}$ and~$\{y_V(\tree)\}_{V \subseteq \ground}$ be the coefficients such that
\[
\Asso(\tree) = \Defo \big( \{z_U(\tree)\}_{U \subseteq \ground} \big) = \Mink \big( \{y_V(\tree)\}_{U \subseteq \ground} \big).
\]
Moreover, we assume that the values $\{z_U(\tree)\}_{U \subseteq \ground}$ are tight, \ie that they define supporting hyperplanes of~$\Asso(\tree)$, so that~$\{z_U(\tree)\}_{U \subseteq \ground}$ and~$\{y_V(\tree)\}_{V \subseteq \ground}$ are connected by M\"obius inversion.
The goal of this section is to give combinatorial formulas for these coefficients~$z_U(\tree)$ and~$y_V(\tree)$.


\subsection{Tight right hand sides}

\enlargethispage{.3cm}
We first focus on the coefficients~$z_U(\tree)$. Note that
\[
z_B(\tree) = \binom{|B|+1}{2},
\]
for any signed building block~$B \in \building(\tree)$. We now want to express all other coefficients~$z_U(\tree)$ as simple combinations of these coefficients~$z_B(\tree)$. Although the setting and presentation are different, we follow the same ideas as in~\cite{Lange}.

Consider a spine~$\spine$ on~$\tree$. We denote by~$\source(\spine)$ and~$\sink(\spine)$ the sources (nodes with only outgoing arcs) and sinks (nodes with only incoming arcs) of~$\spine$. We say that~$\spine$ is a \defn{two-level spine} if all its nodes are in~$\source(\spine) \cup \sink(\spine)$.

\begin{proposition}
\label{prop:twoLevelSpine}
For any subset~$U \subseteq \ground$, there exists a unique two-level spine~$\setToSpine(U)$ such that ${U = \bigcup \source(\setToSpine(U))}$. Therefore, the set~$U$ decomposes into
\[
U = \bigsqcup_s \bigcap_r \source(r),
\]
where $s$ ranges over all nodes of~$\source(\setToSpine(U))$ and $r$ ranges over all arcs of~$\setToSpine(U)$ incident to~$s$.
\end{proposition}

\begin{proof}
This result follows from our study of the spine fan in Section~\ref{sec:spineFan}, see in particular Remark~\ref{rem:extensionSurjection}. Namely, the spine~$\setToSpine(U)$ is the projection~$\tilde\surjectionPermAsso(\prec)$ of the linear preposet~$\prec$ defined by~$u \prec v$ iff~$u \in U$ and~$v \in \ground \ssm U$. In other words, it is the unique signed spine~$\spine$ on~$\tree$ such that the relative interior of the cone~$\normalCone(\spine)$ contains the relative interior of the braid cone~$\normalCone(\prec)$.

Other more elementary proofs are possible for this result. For the convenience of the reader, let us just mention here a construction of~$\setToSpine(U)$ based on flips and contractions. Starting from any maximal signed spine~$\spine$ on~$\tree$, we first flip all the arcs directed from an element of~$\ground \ssm U$ to an element of~$U$, and then contract all the arcs between two elements of~$U$ or between two elements of its complement~$\ground \ssm U$. The resulting signed spine is a two-level spine: its sources are the nodes containing elements of~$U$ and its sinks are the nodes containing elements of its complement~$\ground \ssm U$.

As already mentioned, the label of a node~$s$ in any spine~$\spine$ is given by
\[
\bigg( \bigcap_{o \in O} \source(o) \bigg) \ssm \bigg( \bigcup_{i \in I} \source(i) \bigg),
\]
where~$I$ and~$O$ denote the set of incoming and outgoing arcs of~$\spine$ at~$s$, respectively.
\end{proof}

\fref{fig:exmSetToSpine} illustrates Proposition~\ref{prop:twoLevelSpine} with the spines~$\setToSpine(\{0, 1, 4, 6, 9\})$, $\setToSpine(\{2, 4, 7, 9\})$, $\setToSpine(\{2, 3, 6, 8, 9\})$, and $\setToSpine(\{2, 3, 6\})$ on the signed ground tree~$\tree\ex$ of \fref{fig:exmTree}.

\begin{figure}[h]
  \capstart
  \centerline{\includegraphics[scale=1]{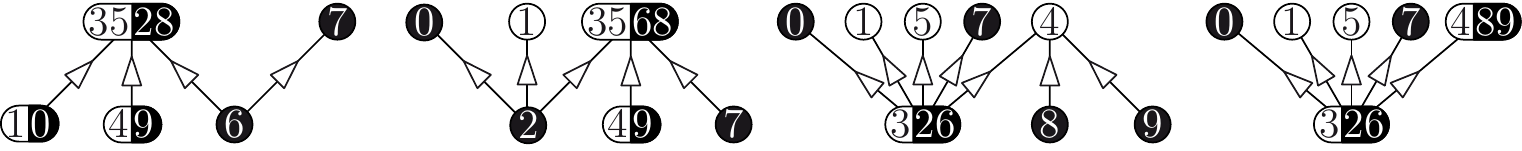}}
  \caption{The two-level spines~$\setToSpine(\{0, 1, 4, 6, 9\})$, $\setToSpine(\{2, 4, 7, 9\})$, $\setToSpine(\{2, 3, 6, 8, 9\})$, and $\setToSpine(\{2, 3, 6\})$ on the signed ground tree~$\tree\ex$ of \fref{fig:exmTree}.}
  \label{fig:exmSetToSpine}
\end{figure}


\begin{remark}
\label{rem:twoLevelSpine}
The decomposition~$U = \bigsqcup_s \bigcap_r \source(r)$ induced by the two-leveled spine of Proposition~\ref{prop:twoLevelSpine} admits the following alternative descriptions, closer to the up and down decompositions of~\cite[Definition~3.2]{Lange} for the associahedra of~\cite{HohlwegLange}.
\begin{enumerate}[(i)]
\item In terms of open subtrees of~$\tree$, the nodes~$s_1, \dots, s_k$ of~$\source(\setToSpine(U))$ correspond to the connected components~$T_1, \dots, T_k$ of~$\tree \ssm (\ground \ssm U^-)$ containing at least one vertex of~$U$. The arcs~$r$ of~$\setToSpine(U)$ incident to a node~$s_i$ then correspond to the open subtrees of~$\tree$ defined by the connected components of~$T_i \ssm U^+$.
\item In the space~$\treeInterval \eqdef \tree \times [-1,1]$, the curves of~$\curves(\tree)$ corresponding to~$\setToSpine(U)$ form the upper hull of the set~$(\ground \ssm U^-) \cup U^+$.
\end{enumerate}
The proof of these alternative descriptions is straightforward and left to the reader.
\end{remark}

\begin{proposition}
\label{prop:zValues}
For any subset~$U \subseteq \ground$, the face of the signed tree associahedron~$\Asso(\tree)$ which minimizes the linear functional~$\b{x} \mapsto \sum_{u \in U} x_u$ is the face corresponding to the spine~$\setToSpine(U)$. Therefore, the right hand side~$z_U(\tree)$ of the supporting hyperplane of~$\Asso(\tree)$ normal to~$\one_U$ is defined~by
\[
z_U(\tree) = \sum_{\substack{r \text{ arc} \\ \text{of } \setToSpine(U)}} \!\! z_{\source(r)}(\tree) - z_\ground(\tree) \!\!\!\!\! \sum_{s \in \source(\setToSpine(U))} \!\!\!\! ( \deg_{\setToSpine(U)}(s) - 1),
\]
where~$\deg_{\setToSpine(U)}(s)$ denotes the degree of the node~$s$ in the spine~$\setToSpine(U)$. 
\end{proposition}

\begin{proof}
As earlier, let~$\prec$ denote the linear preposet defined by~$u \prec v$ iff $u \in U$ and~$v \in \ground \ssm U$. The characteristic vector~$\one_U$ of~$U$ belongs to the braid cone~$\normalCone(\prec)$, and therefore to the cone~$\normalCone(\setToSpine(U))$. It follows that the face corresponding to~$\setToSpine(U)$ minimizes the linear functional~$\b{x} \mapsto \sum_{u \in U} x_u$.

The computation of~$z_U$ then follows the same lines as in the proof of~\cite[Proposition~3.8]{Lange}. Since our setting and notations are slightly different, we detail here this computation for the convenience of the reader:
\begin{align*}
\sum_{u \in U} x_u & = \!\!\!\! \sum_{s \in \source(\setToSpine(U))} \sum_{u \in s\phantom{)}} x_u
					= \!\!\!\! \sum_{s \in \source(\setToSpine(U))} \!\!\! \bigg( z_\ground(\tree) - \sum_{u \notin s} x_u \bigg)
					= \!\!\!\! \sum_{s \in \source(\setToSpine(U))} \!\!\! \bigg( z_\ground(\tree) - \sum_{r \ni s\phantom{)}} \sum_{u \in \sink(r)} x_u \bigg) \\
				  & \ge \!\!\!\! \sum_{s \in \source(\setToSpine(U))} \!\!\! \bigg( \! z_\ground(\tree) - \sum_{r \ni s\phantom{)}} \! \bigg( z_\ground(\tree) - z_{\source(r)}(\tree) \bigg) \!\! \bigg)
					= \!\! \sum_{\substack{r \text{ arc} \\ \text{of } \setToSpine(U)}} \!\! z_{\source(r)}(\tree) - z_\ground(\tree) \!\!\!\!\! \sum_{s \in \source(\setToSpine(U))} \!\!\!\!\! (\deg_{\setToSpine(U)}(s) - 1).
\end{align*}
In the computation, we also used that the source label set~$\source(r)$ and the sink label set~$\sink(r)$ of any arc~$r$ of~$\setToSpine(U)$ partition the ground set~$\ground$, and therefore that
\[
\sum_{u \in \sink(r)} x_u = \sum_{v \in \ground} x_v - \sum_{u \in \source(r)} x_u \le z_\ground(\tree) - z_{\source(r)}(\tree).
\qedhere
\]
\end{proof}

\begin{example}[Tripod, continued]
Consider the tripods~$\tripodWhite$ and~$\tripodBlack$ from Example~\ref{exm:tripod}. Table~\ref{table:zvalues} gives the two-leveled spines~$\setToSpine(U)$ on~$\tripodWhite$ and~$\tripodBlack$, and the values of~$z_U(\tripodWhite)$ and~$z_U(\tripodBlack)$, for all subsets~$\varnothing \ne U \subseteq [4]$ up to isomorphisms of the trees.

\begin{table}[h]
  \capstart
  \centerline{
  \renewcommand*{\arraystretch}{1.2}
  \begin{tabular}{|c||c|c|c|c|c|c|c|}
    \hline
    subset~$U$ & $\{1\}$ & $\{2\}$ & $\{1,2\}$ & $\{1,3\}$ & $\{1,2,3\}$ & $\{1,3,4\}$ & $\{1,2,3,4\}$ \\
    \hline\hline
    spine~$\setToSpine(U)$ on~$\tripodWhite$ & $\vcenter{\hbox{\includegraphics{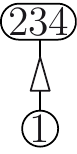}}}$ & $\vcenter{\hbox{\includegraphics{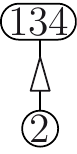}}}$ & $\vcenter{\hbox{\includegraphics{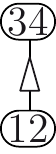}}}$ & $\vcenter{\hbox{\includegraphics{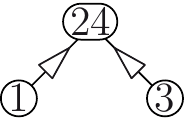}}}$ & $\vcenter{\hbox{\includegraphics{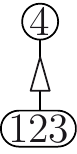}}}$ & $\vcenter{\hbox{\includegraphics{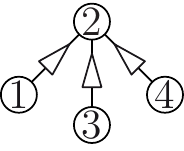}}}$ & $\vcenter{\hbox{\includegraphics{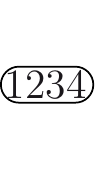}}}$ \\
    \hline
    value of~$z_U(\tripodWhite)$ & $1$ & $1$ & $3$ & $2$ & $6$ & $3$ & $10$ \\
    \hline\hline
    spine~$\setToSpine(U)$ on~$\tripodBlack$ & $\vcenter{\hbox{\includegraphics{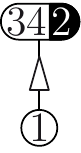}}}$ & $\vcenter{\hbox{\includegraphics{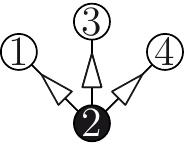}}}$ & $\vcenter{\hbox{\includegraphics{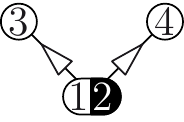}}}$ & $\vcenter{\hbox{\includegraphics{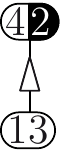}}}$ & $\vcenter{\hbox{\includegraphics{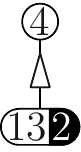}}}$ & $\vcenter{\hbox{\includegraphics{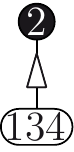}}}$ & $\vcenter{\hbox{\includegraphics{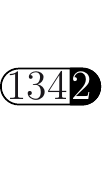}}}$ \\
    \hline
    value of~$z_U(\tripodBlack)$ & $1$ & $-2$ & $2$ & $3$ & $6$ & $6$ & $10$ \\
    \hline
  \end{tabular}
  }
  \bigskip
  \caption[The two-leveled spines~$\setToSpine(U)$ and the values of~$z_U$, for all subsets ${\varnothing \ne U \subseteq [4]}$ up to isomorphisms of the trees.]{The two-leveled spines~$\setToSpine(U)$ on~$\tripodWhite$ and~$\tripodBlack$, and the values of~$z_U(\tripodWhite)$ and~$z_U(\tripodBlack)$, for all subsets ${\varnothing \ne U \subseteq [4]}$ up to isomorphisms of the trees.}
  \vspace*{-.5cm}
  \label{table:zvalues}
\end{table}
\end{example}


\subsection{Minkowski decomposition}

We now study the coefficients~$y_V(\tree)$. Since we have computed the tight right hand sides for all inequalities of~$\Asso(\tree)$, Proposition~\ref{prop:minkowskiSumDiff} ensures that these coefficients~$y_V(\tree)$ are obtained by M\"obius inversion of the coefficients~$z_U(\tree)$, \ie for any~$V \subseteq \ground$,
\[
y_V(\tree) = \sum_{U \subseteq V} (-1)^{|V \ssm U|} z_U(\tree) = \sum_{U \subseteq V} (-1)^{|V \ssm U|} \bigg( \! \sum_{\substack{r \text{ arc} \\ \text{of } \setToSpine(U)}} z_{\source(r)}(\tree) - z_\ground(\tree) \!\!\!\!\! \sum_{s \in \source(\setToSpine(U))} \!\!\!\!\! (\deg_{\setToSpine(U)}(s) - 1) \bigg).
\]
Following the line of research initiated by C.~Lange in~\cite{Lange} for the associahedra of~\cite{HohlwegLange}, we now give direct combinatorial formulas for~$y_V(\tree)$, thus avoiding the exponential cost of M\"obius inversion.

To state our result, we need some preparation. For any~$p,q \in \ground$, we denote by~$[p,q]_\tree$ the unique path in the ground tree~$\tree$ between the vertices~$p$ and~$q$, and by~$(p,q)_\tree \eqdef [p,q]_\tree \ssm \{p,q\}$ the corresponding open path. A subset~$P \subseteq \ground$ is a \defn{negative path} of~$\tree$ if it has two \defn{endpoints}~${p,q \in P}$ such that $P^- = [p,q]_\tree \cap \ground^-$ and~$P^+ \subseteq [p,q]_\tree \cap \ground^+$. In other words, $P$ consists of all negative and some positive vertices on the path~$[p,q]_\tree$ between its two endpoints~$p$ and~$q$. We denote by~$\paths(\tree)$ the set of all negative paths of~$\tree$, and by~$\paths(\tree,p,q)$ the negative paths of~$\paths(\tree)$ with endpoints~$p$ and~$q$. Define the \defn{weight}~$\Omega(P)$ of a negative path~$P \in \paths(\tree,p,q)$ to be
\[
\Omega(P) \eqdef 
\begin{cases}
	(-1)^{|P^+|} \, \omega(p,q) \, \omega(q,p) & \text{if } p \ne q, \\
	\widetilde \omega(p) & \text{if } p = q,
\end{cases}
\]
where
\begin{itemize}
\item $\omega(p,q)$ is $1$ plus the sum of the cardinalities of the connected components of~$\tree \ssm \{p\}$ not containing~$q$ if~$p \in \ground^+$, and $\omega(p,q) = -1$ if~$p \in \ground^-$;
\item $\widetilde \omega(p)$ is the sum of the products~$(1 + |C|)(1 + |C'|)$ over all pairs~$\{C,C'\}$ of distinct connected components of~$\tree \ssm \{p\}$ if~$p \in \ground^+$, and $\widetilde \omega(p) = 1$ if~$p \in \ground^-$.
\end{itemize}
We are now ready to state the main result of this section. Surprisingly, the only non-vanishing coefficients~$y_V(\tree)$ correspond to the negative paths of~$\tree$. They can be computed directly as follows.

\begin{proposition}
\label{prop:MinkowskiDecomposition}
For any~$V \subseteq \ground$, the coefficient~$y_V(\tree)$ is given by
\[
y_V(\tree) = \begin{cases}
	0 & \text{if } V \text{ is not a negative path,} \\
	\Omega(V) + \dfrac{\nu \cdot \deg_\tree(v)}{2} + 1 & \text{if } V = \{v\} \text{ with } v \in \ground^+, \\
	\Omega(V) & \text{otherwise.}
\end{cases}
\]
\end{proposition}

\begin{proof}
We just have to check the equality~$\sum_{V \subseteq U} y_V(\tree) = z_U(\tree)$ for all~${U \subseteq \ground}$. Fix a set~$U \subseteq \ground$, and consider the spine~$\setToSpine(U)$ of~$\tree$ such that~$U = \bigcup \source(\setToSpine(U))$. For any arc~$r$ of~$\setToSpine(U)$, we denote by~$\arcToSubtree(r) \in \subtrees(\tree)$ the corresponding relevant open subtree of~$\tree$. Moreover, for any~$p \in \boundary \arcToSubtree(r)^+$, we define~$\omega(p,r)$ to be $1$ plus the sum of the cardinalities of the connected components of~$\tree \ssm \{p\}$ not containing~$\arcToSubtree(r)$. In the computation below, we need the following observations:
\begin{enumerate}[(i)]
\item For any arc~$r$ of~$\setToSpine(U)$, we have
\[
|\source(r)| = |\arcToSubtree(r)^-| + \sum_{p \in \boundary \arcToSubtree(r)^+} \omega(p,r).
\]
\item For any vertex~$p \in U^+$, we can decompose~$\widetilde \omega(p)$ into
\[
\widetilde \omega(p) = \sum_r \frac{1}{2} \, \omega(p,r) \big(\nu + 1 - \omega(p,r) \big)
\]
where~$r$ ranges over all arcs of~$\setToSpine(U)$ such that~$p \in \boundary \arcToSubtree(r)$.
\item Any negative path of~$\tree$ contained in~$U$ is in fact contained in the label set of a single source of~$\setToSpine(U)$. Indeed, by Remark~\ref{rem:twoLevelSpine}, any two distinct sources of~$\setToSpine(U)$ are separated by a negative vertex in~$\ground^- \ssm U$.
\item For any~$p,q \in \ground$, the sum of the weights of the negative paths of~$\paths(\tree,p,q)$ included in~$U$ vanishes as soon as $(p,q)_\tree \cap U^+ \ne \varnothing$. Indeed, if there exists~$w \in (p,q)_\tree \cap U^+$, we can decompose this sum as
\[
\sum_{\substack{P \in \paths(\tree,p,q) \\ P \subseteq U}} \!\!\! \Omega(P) = \!\!\! \sum_{\substack{P \in \paths(\tree,p,q) \\ w \in P \subseteq U}} \!\!\! \Omega(P) \; + \!\!\! \sum_{\substack{P \in \paths(\tree,p,q) \\ w \notin P \subseteq U}} \!\!\! \Omega(P) = \!\!\! \sum_{\substack{P \in \paths(\tree,p,q) \\ w \in P \subseteq U}} \!\!\! \Omega(P) \; + \!\!\! \sum_{\substack{P \in \paths(\tree,p,q) \\ w \notin P \subseteq U}} \!\!\! -\Omega(P \cup \{w\}) = 0.
\]
\end{enumerate}
It follows from Observations (iii) and~(iv) that the sum of the weights of all paths of~$\tree$ included in~$U$ is just the sum of the weights of all paths~$[p,q]_\tree$ for~$p,q \in \arcToSubtree(r) \cup \boundary \arcToSubtree(r)$ for all arcs~$r$ of~$\setToSpine(U)$. Decomposing this sum according on whether~$p$ and~$q$ are in~$\arcToSubtree(r)$ or in~$\boundary \arcToSubtree(r)$, we can thus write
\begin{align*}
\sum_{\substack{P \in \paths(\tree) \\ P \subseteq U}} \!\! \Omega(P)
& = \!\! \sum_{\substack{r \text{ arc} \\ \text{of } \setToSpine(U)}} \!\!\! \bigg( \sum_{\{p,q\} \subseteq \arcToSubtree(r)^-} \!\!\!\! 1
+ \!\!\!\! \sum_{\substack{p \in \boundary \arcToSubtree(r)^+ \\ q \in \arcToSubtree(r)^-}} \!\!\!\! \omega(p,r)
+ \!\!\!\!\!\!\!\! \sum_{\substack{\{p, q\} \subseteq \boundary \arcToSubtree(r)^+ \\ p \ne q \phantom{)}}} \!\!\!\!\!\!\!\! \omega(p,r) \, \omega(q,r)
- \!\!\!\! \sum_{p \in \boundary \arcToSubtree(r)^+} \!\!\!\!\!\! \frac{\omega(p,r) \big(\nu + 1 - \omega(p,r) \big)}{2} \bigg) \\
& = \!\! \sum_{\substack{r \text{ arc} \\ \text{of } \setToSpine(U)}} \!\! \bigg( \!\! \binom{|\arcToSubtree(r)^-| + \!\!\! \mathlarger{\sum_{p \in \boundary \arcToSubtree(r)^+} \!\!\! \omega(p,r)} + 1}{2} - \frac{\nu + 2}{2} \!\! \sum_{p \in \boundary \arcToSubtree(r)^+} \!\! \omega(p,r) \bigg) \\
& = \!\! \sum_{\substack{r \text{ arc} \\ \text{of } \setToSpine(U)}} \binom{|\source(r)| + 1}{2} - \frac{\nu + 2}{2} \sum_{\substack{r \text{ arc} \\ \text{of } \setToSpine(U)}} \sum_{p \in \boundary \arcToSubtree(r)^+} \!\! \omega(p,r) \\
& = \!\! \sum_{\substack{r \text{ arc} \\ \text{of } \setToSpine(U)}} z_{\source(r)}(T) - \frac{\nu + 2}{2} \sum_{p \in U^+} (\nu \cdot \deg_\tree(p) - \nu + 1).
\end{align*}
Let us briefly explain this computation. In the first equality, the first three terms should be clear. The last one corresponds to the sum of the weights $\Omega(\{p\}) = \widetilde \omega(p)$ for~$p \in \boundary\arcToSubtree(r)^+$. We used Observation~(ii) above to decompose these weights as sums over the arcs~$r$ of~$\setToSpine(U)$. The second equality is obtain by a simple rearrangement of the first line. To obtain the third equality, we used Observation~(i). The last equality comes from the fact that for any~$p \in U^+$, there are~$\deg_\tree(p)$ many arcs~$r$ of~$\setToSpine(U)$ with~$p \in \boundary \arcToSubtree(r)$. Putting pieces together, we finally obtain that
\begin{align*}
\sum_{V \subseteq U} y_V(\tree) & = \sum_{\substack{P \in \paths(\tree) \\ P \subseteq U}} \!\! \Omega(P) + \sum_{p \in U^+} \bigg( \dfrac{\nu \cdot \deg_\tree(p)}{2} + 1 \bigg) \\
& = \!\! \sum_{\substack{r \text{ arc} \\ \text{of } \setToSpine(U)}} z_{\source(r)}(T) - \frac{1}{2} \sum_{p \in U^+} \big( (\nu + 2) (\nu \cdot \deg_\tree(p) - \nu + 1) - \nu \cdot \deg_\tree(p) - 2 \big) \\
& = \!\! \sum_{\substack{r \text{ arc} \\ \text{of } \setToSpine(U)}} z_{\source(r)}(T) - \binom{\nu + 1}{2} \sum_{p \in U^+} ( \deg_\tree(p) - 1)
= z_U(\tree),
\end{align*}
where the last equality holds since
\[
\deg_{\setToSpine(U)}(s) - 1 = \sum_{u \in U^+ \cap s} ( \deg_\tree(u) - 1)
\]
for all nodes~$s \in \source(\setToSpine(U))$ by Remark~\ref{rem:twoLevelSpine}.
\end{proof}

\begin{example}[Signed paths, continued]
When~$\pathG$ is a signed path, the coefficients~$z_U(\pathG)$ and~$y_V(\pathG)$ where described by C.~Lange in~\cite{Lange}. Note that his description is more general than ours, since it is valid for all associahedra constructed from any type~$A$ permutahedron by removing facets, while we force our construction to start from the classical permutahedron. However, we believe that the present approach to prove Proposition~\ref{prop:MinkowskiDecomposition} saves an important part of the efforts needed in~\cite{Lange}. In~\cite{LangePilaud-spines}, we use this approach to revisit C.~Lange's results in the simplified context of the associahedra of~\cite{HohlwegLange}, where negative paths appear as big tops.
\end{example}

\begin{example}[Unsigned trees, continued]
If~$\tree$ has only negative vertices, all coefficients~$y_V(\tree)$ belong to~$\{0,1\}$. More precisely, for any~$V \subseteq \ground$, we have~$y_V(\tree) = 1$ if $V$ is the vertex set of a path in~$\tree$, and $y_V(\tree) = 0$ otherwise. We say that~$\Asso(\tree)$ is the Minkowski sum of all paths in~$\tree$. In particular, for a path~$\pathG$ with only negative vertices, J.-L.~Loday's associahedron~$\Asso(\pathG)$ is the Minkowski sum of all intervals. We observe again that our tree associahedron~$\Asso(\tree)$ differs from the classical realizations of~\cite{Postnikov, FeichtnerSturmfels}: they consider the Minkowski sum of the faces of the standard simplex corresponding to all subtrees of~$\tree$, while we restrict the summation over all paths in~$\tree$. In a current project with C.~Lange, we explore extensions of this observation to larger classes of graph associahedra.
\end{example}

\begin{example}[Tripod, continued]
\label{exm:tripodMinkowski}
Consider the tripods~$\tripodWhite$ and~$\tripodBlack$ from Example~\ref{exm:tripod}. We have computed in Table~\ref{table:zvalues} the values of~$z_U(\tripodWhite)$ and~$z_U(\tripodBlack)$, for all subsets~$\varnothing \ne U \subseteq [4]$ up to isomorphisms of the trees. Table~\ref{table:yvalues} gives the values of~$y_V(\tripodWhite)$ and~$y_V(\tripodBlack)$, for all subsets~$\varnothing \ne V \subseteq [4]$ up to isomorphisms of the trees.

\begin{table}[h]
  \capstart
  \centerline{
  \renewcommand*{\arraystretch}{1.2}
  \begin{tabular}{|c||c|c|c|c|c|c|c|}
    \hline
    subset~$V$ & $\{1\}$ & $\{2\}$ & $\{1,2\}$ & $\{1,3\}$ & $\{1,2,3\}$ & $\{1,3,4\}$ & $\{1,2,3,4\}$ \\
    \hline
    value of~$y_V(\tripodWhite)$ & $1$ & $1$ & $1$ & $0$ & $1$ & $0$ & $0$ \\
    \hline
    value of~$y_V(\tripodBlack)$ & $1$ & $-2$ & $3$ & $1$ & $-1$ & $0$ & $0$ \\
    \hline
  \end{tabular}
  }
  \bigskip
  \caption[The values of~$y_V$, for all subsets ${\varnothing \ne U \subseteq [4]}$ up to isomorphisms of the trees.]{The values of~$y_V(\tripodWhite)$ and~$y_V(\tripodBlack)$, for all subsets ${\varnothing \ne U \subseteq [4]}$ up to isomorphisms of the trees.}
  \vspace*{-.5cm}
  \label{table:yvalues}
\end{table}

\begin{figure}[b]
  \capstart
  \centerline{\includegraphics[scale=.6]{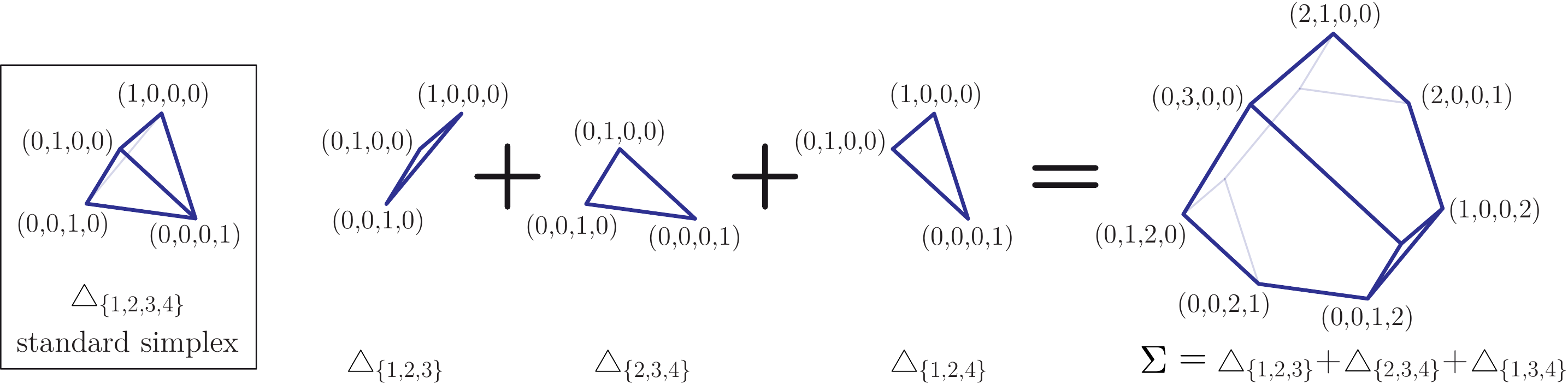}}
  \caption[The Minkowski summand~$\Sigma = \simplex_{\{1,2,3\}} +\simplex_{\{2,3,4\}} +\simplex_{\{1,2,4\}}$.]{The Minkowski summand~$\Sigma \eqdef \simplex_{\{1,2,3\}} +\simplex_{\{2,3,4\}} +\simplex_{\{1,2,4\}}$.}
  \label{fig:Minkowski1}
\end{figure}

We have represented in \fref{fig:Minkowski1} the Minkowski sum~$\Sigma$ of the faces~$\{1,2,3\}$, $\{2,3,4\}$ and~$\{1,2,4\}$ of the standard simplex~$\simplex_{[4]}$ (which is represented in the box on the left of the picture). This polytope appears as a summand in the Minkowski decompositions of both~$\Asso(\tripodWhite)$ and~$\Asso(\tripodBlack)$. These decompositions can be visualized on \fref{fig:Minkowski2} (for convenience, we have repeated the standard simplex~$\simplex_{[4]}$ in the box on the right of the figure). The first line of \fref{fig:Minkowski2} illustrates

\begin{landscape}
\begin{figure}
  \capstart
  \vspace*{-.8cm}
  \centerline{\hspace*{-1.5cm}\includegraphics[scale=.6]{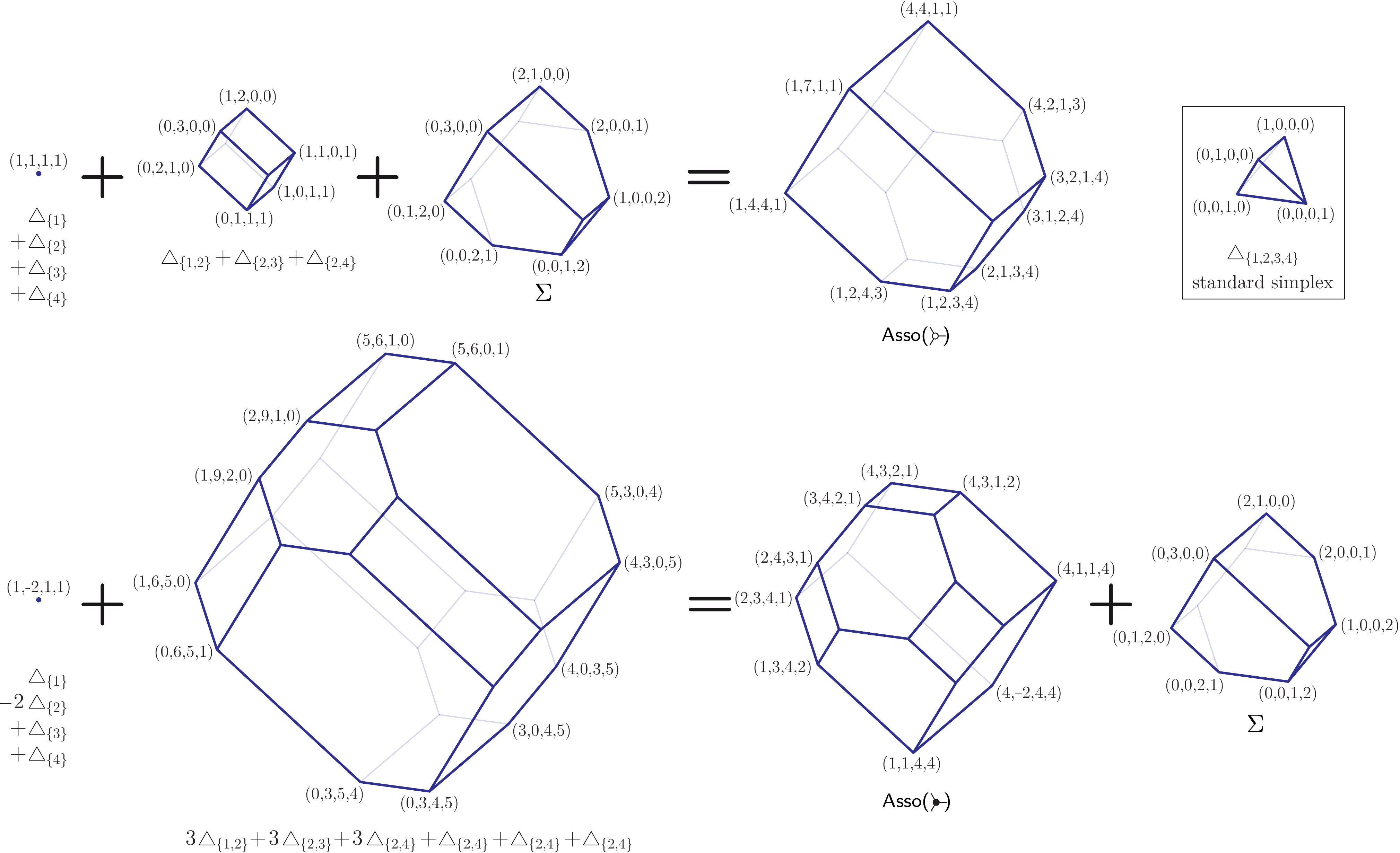}}
  \caption[Minkowski decompositions of the signed tree associahedra.]{Minkowski decompositions of the signed tree associahedra~$\Asso(\tripodWhite)$ and~$\Asso(\tripodBlack)$. See Example~\ref{exm:tripodMinkowski} for explanations.}
  \label{fig:Minkowski2}
\end{figure}
\end{landscape}

\noindent
the Minkowski decomposition of the signed tree associahedron~$\Asso(\tripodWhite)$, grouping its summands by dimension. Namely, the polytope $\Asso(\tripodWhite)$ is the sum of the point~$(1,1,1,1)$, of the parallelepiped ${\simplex_{\{1,2\}} + \simplex_{\{2,3\}} + \simplex_{\{2,4\}}}$, and of the polytope~$\Sigma = \simplex_{\{1,2,3\}} +\simplex_{\{2,3,4\}} +\simplex_{\{1,2,4\}}$. The second line of \fref{fig:Minkowski2} illustrates the Minkowski decomposition of the signed tree associahedron~$\Asso(\tripodBlack)$, grouping again its summands by dimension. Namely, the sum of the point~$(1,-2,1,1)$ with the permutahedron~$3 \simplex_{\{1,2\}} + 3 \simplex_{\{2,3\}} + 3 \simplex_{\{2,4\}} + \simplex_{\{1,3\}} + \simplex_{\{1,4\}} + \simplex_{\{3,4\}}$ coincides with the sum of~$\Asso(\tripodBlack)$ with the polytope~$\Sigma = \simplex_{\{1,2,3\}} +\simplex_{\{2,3,4\}} +\simplex_{\{1,2,4\}}$.

Finally, we have represented in Figure~\ref{fig:MobiusInversion} the M\"obius inversion from~$z_U(\tripodWhite)$ to~$y_V(\tripodWhite)$, and from~$z_U(\tripodBlack)$ to~$y_V(\tripodBlack)$.

\begin{figure}[h]
  \capstart
  \medskip
  \newcommand{\sep}{\,|\,}
  \newcommand{\block}[3]{
  	\begin{tabular}{|c|c|}\hline
      \multicolumn{2}{|c|}{$#1$} \\ \hline
      $#2$ & $#3$ \\ \hline 
    \end{tabular}
  }
  \newlength{\lskip}\setlength{\lskip}{1.6cm} 
  \newlength{\cskip}\setlength{\cskip}{.75cm} 
  \centerline{
  \begin{tikzpicture}[inner sep=0pt, align=center, row sep = 0.5cm, column sep = 2cm]
    \node        at ( -4\cskip, 0\lskip) {\block{U = V}{z_U(\tripodWhite)}{y_V(\tripodWhite)}};
    \node (e)    at ( 0\cskip, 0\lskip) {\block{\varnothing}{0}{0}};
    \node (1)    at (-3\cskip, 1\lskip) {\block{1}{1}{1}};
    \node (2)    at (-1\cskip, 1\lskip) {\block{2}{1}{1}};
    \node (3)    at ( 1\cskip, 1\lskip) {\block{3}{1}{1}};
    \node (4)    at ( 3\cskip, 1\lskip) {\block{4}{1}{1}};
    \node (12)   at (-5\cskip, 2\lskip) {\block{12}{3}{1}};
    \node (13)   at (-3\cskip, 2\lskip) {\block{13}{2}{0}};
    \node (14)   at (-1\cskip, 2\lskip) {\block{14}{2}{0}};
    \node (23)   at ( 1\cskip, 2\lskip) {\block{23}{3}{1}};
    \node (24)   at ( 3\cskip, 2\lskip) {\block{24}{3}{1}};
    \node (34)   at ( 5\cskip, 2\lskip) {\block{34}{2}{0}};
    \node (123)  at (-3\cskip, 3\lskip) {\block{123}{6}{1}};
    \node (124)  at (-1\cskip, 3\lskip) {\block{124}{6}{1}};
    \node (134)  at ( 1\cskip, 3\lskip) {\block{134}{3}{0}};
    \node (234)  at ( 3\cskip, 3\lskip) {\block{234}{6}{1}};
    \node (1234) at ( 0\cskip, 4\lskip) {\block{1234}{10}{0}};
    \draw
      (e.north)   edge (1.south)
      (e.north)   edge (2.south)
      (e.north)   edge (3.south)
      (e.north)   edge (4.south)
      (1.north)   edge (12.south)
      (1.north)   edge (13.south)
      (1.north)   edge (14.south)
      (2.north)   edge (12.south)
      (2.north)   edge (23.south)
      (2.north)   edge (24.south)
      (3.north)   edge (13.south)
      (3.north)   edge (23.south)
      (3.north)   edge (34.south)
      (4.north)   edge (14.south)
      (4.north)   edge (24.south)
      (4.north)   edge (34.south)
      (12.north)  edge (123.south)
      (12.north)  edge (124.south)
      (13.north)  edge (123.south)
      (13.north)  edge (234.south)
      (14.north)  edge (124.south)
      (14.north)  edge (134.south)
      (23.north)  edge (123.south)
      (23.north)  edge (234.south)
      (24.north)  edge (124.south)
      (24.north)  edge (234.south)
      (34.north)  edge (134.south)
      (34.north)  edge (234.south)
      (123.north) edge (1234.south)
      (124.north) edge (1234.south)
      (134.north) edge (1234.south)
      (234.north) edge (1234.south)
    ;
  \end{tikzpicture}
  \qquad
  \begin{tikzpicture}[inner sep=0pt, align=center, row sep = 0.5cm, column sep = 2cm]
    \node        at ( 4\cskip, 0\lskip) {\block{U = V}{z_U(\tripodBlack)}{y_V(\tripodBlack)}};
    \node (e)    at ( 0\cskip, 0\lskip) {\block{\varnothing}{0}{0}};
    \node (1)    at (-3\cskip, 1\lskip) {\block{1}{1}{1}};
    \node (2)    at (-1\cskip, 1\lskip) {\block{2}{-2}{-2}};
    \node (3)    at ( 1\cskip, 1\lskip) {\block{3}{1}{1}};
    \node (4)    at ( 3\cskip, 1\lskip) {\block{4}{1}{1}};
    \node (12)   at (-5\cskip, 2\lskip) {\block{12}{2}{3}};
    \node (13)   at (-3\cskip, 2\lskip) {\block{13}{3}{1}};
    \node (14)   at (-1\cskip, 2\lskip) {\block{14}{3}{1}};
    \node (23)   at ( 1\cskip, 2\lskip) {\block{23}{2}{3}};
    \node (24)   at ( 3\cskip, 2\lskip) {\block{24}{2}{3}};
    \node (34)   at ( 5\cskip, 2\lskip) {\block{34}{3}{1}};
    \node (123)  at (-3\cskip, 3\lskip) {\block{123}{6}{-1}};
    \node (124)  at (-1\cskip, 3\lskip) {\block{124}{6}{-1}};
    \node (134)  at ( 1\cskip, 3\lskip) {\block{134}{6}{0}};
    \node (234)  at ( 3\cskip, 3\lskip) {\block{234}{6}{-1}};
    \node (1234) at ( 0\cskip, 4\lskip) {\block{1234}{10}{0}};
    \draw
      (e.north)   edge (1.south)
      (e.north)   edge (2.south)
      (e.north)   edge (3.south)
      (e.north)   edge (4.south)
      (1.north)   edge (12.south)
      (1.north)   edge (13.south)
      (1.north)   edge (14.south)
      (2.north)   edge (12.south)
      (2.north)   edge (23.south)
      (2.north)   edge (24.south)
      (3.north)   edge (13.south)
      (3.north)   edge (23.south)
      (3.north)   edge (34.south)
      (4.north)   edge (14.south)
      (4.north)   edge (24.south)
      (4.north)   edge (34.south)
      (12.north)  edge (123.south)
      (12.north)  edge (124.south)
      (13.north)  edge (123.south)
      (13.north)  edge (234.south)
      (14.north)  edge (124.south)
      (14.north)  edge (134.south)
      (23.north)  edge (123.south)
      (23.north)  edge (234.south)
      (24.north)  edge (124.south)
      (24.north)  edge (234.south)
      (34.north)  edge (134.south)
      (34.north)  edge (234.south)
      (123.north) edge (1234.south)
      (124.north) edge (1234.south)
      (134.north) edge (1234.south)
      (234.north) edge (1234.south)
    ;
  \end{tikzpicture}}
  \caption[M\"obius inversion from~$z_U to~$y_V$. In each box, the first row contains the set~$U = V$ and the second one contains first~$z_U$ and then~$y_V$.]{M\"obius inversion from~$z_U(\tripodWhite)$ to~$y_V(\tripodWhite)$ (left), and from~$z_U(\tripodBlack)$ to~$y_V(\tripodBlack)$ (right). In each box, the first row contains the set~$U = V$ and the second one contains first~$z_U$ and then~$y_V$.}
  \label{fig:MobiusInversion}
\end{figure}
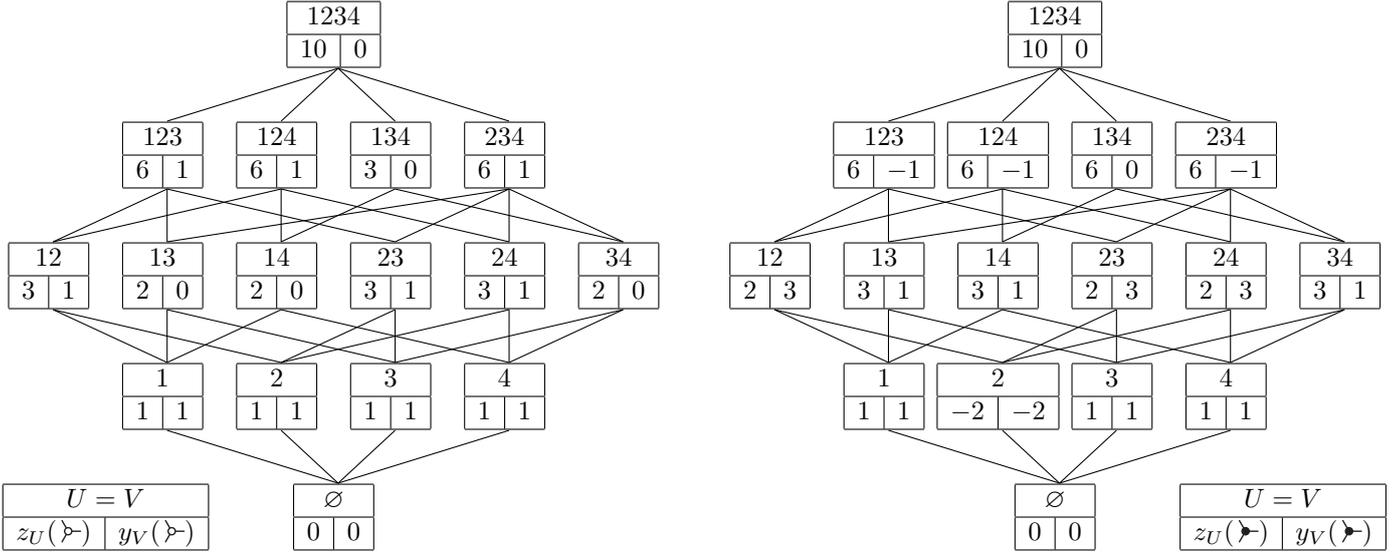
\end{example}

To conclude, we want to relate the Minkowski decomposition studied in this section to other possible decompositions. In their study of brick polytopes of sorting networks~\cite{PilaudSantos-brickPolytope}, V.~Pilaud and F.~Santos showed that the associahedra of~\cite{HohlwegLange} can all be decomposed as Minkowski sums of matroid polytopes. In their decomposition, all coefficients are positive, but the summands are not as simple as faces of the standard simplex. As far as we know, the connection between the decomposition of~\cite{Lange} and that of~\cite{PilaudSantos-brickPolytope} are not understood. It raises the following questions.

\begin{remark}
Can all signed tree associahedra be decomposed as Minkowski sums of matroid polytopes? If the answer is affirmative, what is the connection between the two possible decompositions, as Minkowski sums and differences of dilates of faces of the standard simplex, and as Monkowski sums of matroid polytopes?
\end{remark}


\section*{Acknowledgement}

This paper is a first step towards a direction of research which I investigate in collaboration with S.~\v Cuki\'c and C.~Lange. The motivation of this paper came on one side from various discussions we had on our independent approaches to define signed nestohedra, and on the other side from my joint work with C.~Lange on spines and associahedra. I thank S.~\v Cuki\'c and C.~Lange for sharing their preliminary ideas with me and for our fruitful collaboration. More particularly, I am deeply grateful to C.~Lange for his attention and patience, his multiple comments and his help on this project: this paper would probably not exist without our discussions. I also thank Stefan Forcey, Christophe Hohlweg, J.-P.~Labb\'e, and Francisco Santos for their comments on a preliminary version of this paper.


\bibliographystyle{alpha}
\bibliography{signedTreeAssociahedra}
\label{sec:biblio}

\end{document}